\documentclass[11pt]
{amsart}

\usepackage{hyperref}
\usepackage{bbm}
\usepackage{cite}
\usepackage{amssymb}
\usepackage{amsmath}
\usepackage{amsthm}
\usepackage{amsfonts}
\usepackage{graphicx}
\usepackage{enumitem}
\usepackage{cancel}
\usepackage{comment}
\usepackage{xcolor}
\usepackage[toc,page]{appendix}
\usepackage{mathtools}
\usepackage{mathrsfs}
\usepackage{upgreek}
\usepackage{dsfont}
\usepackage{color}
\usepackage{pdfsync}
\usepackage{color}
\usepackage{etoolbox}

\numberwithin{equation}{section}
\newtheorem{theorem}{Theorem}[section]
\newtheorem{proposition}[theorem]{Proposition}
\newtheorem{lemma}[theorem]{Lemma}
\newtheorem{corollary}[theorem]{Corollary}
\theoremstyle{definition}

\newtheorem{definition}[theorem]{Definition}

\newtheorem{notation}[theorem]{Notation}
\theoremstyle{remark}
\newtheorem{remark}[theorem]{Remark}

\def\C{{\mathbb{C}}}

\def\N{{\mathbb N}}
\def\R{{\mathbb R}}

\def\i{\textnormal {i}}

\newcommand{\trace}{\mathrm{tr}}
\renewcommand{\ge}{\geqslant}
\renewcommand{\le}{\leqslant}
\renewcommand{\Re}{\operatorname{Re}}
\renewcommand{\Im}{\operatorname{Im}}
\newcommand{\Span}{\operatorname{Span}}
\newcommand{\Range}{\operatorname{Ran}}

\newcommand{\HS}{\operatorname{HS}}
\newcommand{\Pf}{\operatorname{Pf}}

\newcommand{\rank}{\operatorname{rank}}
\newcommand{\Ind}{\operatorname{Ind}}
\renewcommand{\liminf}{\mathop{\underline{\lim}}\limits}

\newcommand{\masha}[1]
{{\color{red} Masha says: #1}}

\newcommand{\rohan}[1]
{{\color{blue} Rohan says: #1}}

\begin{document}
	
	\title[Spectral theory]{Spectral theory for L\'evy and L\'evy-Ornstein-Uhlenbeck semigroups on step 2 Carnot groups}
	
	\makeatletter
	\@namedef{subjclassname@2020}{%
		\textup{2020} Mathematics Subject Classification}
	\makeatother
	
	\author[Gordina]{Maria Gordina{$^{\dag}$}}
	\address{ Department of Mathematics\\
		University of Connecticut\\
		Storrs, CT 06269,  U.S.A.}
	\thanks{\footnotemark {$\dag$} Research was supported in part by NSF Grant DMS-2246549.}
	\email{maria.gordina@uconn.edu}
	
	\author[Sarkar]{Rohan Sarkar{$^{\dag }$}}
	\address{ Department of Mathematics and Statistics\\
		Binghamton University\\
		Binghamton, NY 13795,  U.S.A.}
	\email{rsarkar2@binghamton.edu}

	\keywords{Carnot groups; sub-Laplacian; generalized Fourier transform; L\'evy processes on Lie groups; Ornstein-Uhlenbeck semigroups; spectrum; intertwining relationship; non-self-adjoint integro-differential operators}
	
	\subjclass[2020]{35P05; 22E25; 60E07; 42B15}
	
	
\begin{abstract}
We consider non-local perturbations $\Delta^\psi_G$ of sub-Laplacians on a step $2$ Carnot group $G$. The perturbations are by translation-invariant non-local operators acting along the vertical directions in $G$. We use harmonic analysis on $G$ to obtain intertwining relationship between the semigroups generated by $\Delta^\psi_G$ and some strongly continuous contraction semigroups on Euclidean spaces with purely continuous spectrum, and as a result we identify the spectrum of $\Delta^\psi_G$. Further we introduce the L\'evy-Ornstein-Uhlenbeck (OU) semigroup corresponding to $\Delta^\psi_G$. We prove that these Markov semigroups are ergodic, though they are not normal operators on  $L^2$ space with respect to the invariant distribution $\mathsf{p}_\psi$. The intertwining relationships allow us to show that all L\'evy-OU generators on $G$ are isospectral, that is, they have the same eigenvalues with the same multiplicities. As a byproduct, we obtain a precise description of the eigenspaces, and also derive explicit formula for the co-eigenfunctions corresponding to some eigenvalues.
\end{abstract}
	
\maketitle
	
\tableofcontents

\section{Introduction}

	Spectral theory of Markov semigroups are of special interest as they are related to various functional inequalities such as Poincar\'e inequality, logarithmic Sobolev inequality, hypercontractivity. In particular,  the spectral gap is often used to find the rate of convergence of a Markov semigroup to its invariant distribution (when it exists). One of our goals is to study the spectrum of a sub-Laplacian $\Delta_{\mathcal{H}}$ and its L\'evy-type perturbations on a step $2$ Carnot group $G\cong\R^n\times\R^m$, where $\R^n$ is identified with the horizontal layer, see \eqref{eq:group_law} below for details. When $m=0$, that is, $G$ is a Euclidean space, the spectral theory of L\'evy semigroups relies on the Fourier transform of such semigroups, see e.g. \cite{Applebaum2019}. However, for a non-commutative locally compact Lie group, the Fourier transform is not scalar-valued, therefore making the spectral analysis more intricate. For the $3$-dimensional Heisenberg group, the spectrum of $\Delta_{\mathcal{H}}$ was obtained in \cite{DasguptaMolahajlooWong2011} using the Fourier-Wigner transform on the Heisenberg group. We also refer to \cite{FurutaniSagamiOtsuki1993} for the spectrum of the Laplace-Beltrami operator on Carnot groups of step $2$, where the authors proved that the spectrum equals $[0,\infty)$, and for Heisenberg groups, the spectrum is purely continuous. In \cite{NiuLuo2000}, the authors obtained spectrum of some second order differential operators on step $2$ nilpotent Lie groups under some mild conditions on the group, and their methods rely on the unitary representations of such groups. 
	
	As the spectrum of the Laplace operator on Euclidean spaces is $[0, \infty)$, \cite[Theorem~1]{FurutaniSagamiOtsuki1993} makes it natural to expect some connection between spectral properties of  $\Delta_{\mathcal H}$ and the Laplace operator on Euclidean space. We prove that the spectrum of the sub-Laplacian on any Carnot group of step $2$ is purely continuous, and it is equal to $[0,\infty)$. The novelty of our approach is to rely on intertwining relationships with some strongly continuous contraction semigroups on Euclidean spaces, which not only strengthens the results in \cite{FurutaniSagamiOtsuki1993}, but also extends those to a larger class of non-local operators on $G$.
	
Two bounded linear operators $P, Q$ on Banach spaces $\mathcal{X}_1, \mathcal{X}_2$ are said to be \emph{intertwined} if there exists a bounded linear operator $\Lambda:\mathcal{X}_2\longrightarrow\mathcal{X}_1$ such that 
	\begin{align*}
		P\Lambda=\Lambda Q.
	\end{align*}
	When $\Lambda$ is an invertible operator, $P, Q$ have the same spectrum. Although $\Lambda$ is not assumed to be injective in general, having an intertwining relationship often enables us to obtain some spectral information on $P$ (resp.~$Q$) if the same information on $Q$ (resp.~$P$) is known. The idea of using intertwining to study the spectral and ergodicity properties of non-self-adjoint integro-differential Markov operators have been used previously, and we refer to \cite{PatieSarkar2023, MicloPatieSarkar2022, Patie_Savov} for spectral theory of self-similar and related Markov semigroups using intertwining relationships with self-adjoint operators.
	
 We now briefly describe the main idea that allows us to identifying the spectrum of $\Delta_{\mathcal{H}}$ on $L^2(G)$, the space of all square-integrable functions on $G$ with respect to the (bi-invariant) Haar measure. Denote by $\mathbb{Q}_t$  the semigroup on $L^2(G)$ whose infinitesimal generator is $\Delta_{\mathcal{H}}$, we first prove that there exists a closed subspace $\mathcal{L}\subset L^2(G)$ such that 
	\begin{align}\label{eq:L_beta}
		\mathcal{L}=\bigoplus_{\beta\in\N^d_0}^\infty \mathcal{L}_\beta,
	\end{align}
where $\mathcal{L}_\beta\cong L^2(\R^{m+k})$, $m,k$ are the dimension of the second layer of the Lie algebra of $G$ and the \emph{radical} respectively, as defined in Section~\ref{sec:GFT_step2}, and for each $\beta\in\N^d_0$ there exists a strongly continuous contraction semigroup $Q^\beta$ and a unitary operator $\operatorname{M}_\beta:L^2(\R^{m+k})\longrightarrow\mathcal{L}_\beta$ such that 
	\begin{enumerate}
		\item\label{it:intro_1} $\mathbb{Q}_{t}\operatorname{M}_\beta=\operatorname{M}_\beta Q^\beta_t$   on $\mathcal{L}_\beta$; 
		\item\label{it:intro_2} $\sigma(Q^\beta_t)=\sigma_c(Q^\beta_t)=[0,1]$ for all $t>0$ and $\beta\in\N$.
	\end{enumerate}
	The novelty of this approach lies in characterizing the subspaces $\mathcal{L}_\beta$ in \eqref{eq:L_beta} by using the generalized Fourier transform (GFT) on Carnot groups. We refer to Proposition~\ref{prop:decomp} and Section~\ref{sec:5} for details about the construction of Hilbert spaces $\mathcal{L}_\beta$ and the semigroups $Q^\beta_{t}$. In particular, for $H$-type groups and non-isotropic Heisenberg groups, $Q^\beta_{t}$ corresponds to a rescaled Poisson semigroup on $\R^{m}$, where $m$ is the dimension of the center of the Lie algebra. 
	
	The decomposition technique in \eqref{it:intro_1} and \eqref{it:intro_2}  extends further to non-local perturbations of the sub-Laplacian. Namely, we introduce L\'evy-type perturbations of $\Delta_{\mathcal{H}}$ along the vertical directions in $G$, denoted by $\Delta^{\!\psi}_{G}$, where $\psi$ is the L\'evy-Khintchine exponent of a L\'evy process. Such operators are generators of L\'evy processes on $G$, and  we show in Theorem~\ref{thm:characterization} that $\Delta^{\!\psi}_{G}$ characterizes all L\'evy processes on $G$ whose \emph{horizontal projection} is the Brownian motion on $\mathbb{R}^n$. Using the decomposition in \eqref{eq:L_beta} and intertwining relationships, we describe the spectrum of $\Delta^{\!\psi}_{G}$ in Theorem~\ref{thm:spec_dGp}, in particular, that the spectrum is purely continuous for all such perturbations. Our results also indicate that the spectrum of $\Delta^{\!\psi}_{G}$ depends on perturbations, and it is given explicitly in terms of the L\'evy-Khintchine exponent $\psi$  in Theorem~\ref{thm:spec_dGp}. 
	
	For a general treatment of L\'evy processes on Lie groups we refer to \cite{LiaoMingBookLevyProcessesinLieGroups}, and finding a core for the infinitesimal generator of such a process is a non-trivial question in general. In Theorem~\ref{thm:heat_perurbation}, we show that compactly supported smooth functions on $G$, denoted by $C^\infty_c(G)$, is a core for $\Delta^{\!\psi}_{G}$, therefore generating a Feller semigroup, and a standard approximation argument using mollifiers proves that $C^\infty_c(G)$ is a core for $\Delta^{\!\psi}_{G}$ in $L^2(G)$. We also refer to \cite{AsaadGordina2016} for the analysis of horizontal heat kernel on a (filiform) $n$-step nilpotent Lie group, and to \cite{GordinaHaga2014} for the analysis of L\'evy processes on $3$-step nilpotent Lie groups. Both of these references use Kirillov's representation theory and the generalized Fourier transform relevant to the current article.
	
	In addition to the operators mentioned above, we  introduce the Ornstein-Uhlenbeck (OU) operators associated with $\Delta^{\!\psi}_{G}$, that is, we define 
	\begin{align}\label{eq:intro_L}
		\mathbb{L}_\psi=\Delta^{\!\psi}_{G}+\mathbb{D},
	\end{align}
	where $\mathbb{D}$ is the generator of the dilation semigroup on $G$, that is, $\delta_{e^{-t}}=e^{t\mathbb{D}}$ for all $t\geqslant 0$. $\mathbb{L}_\psi$ can be realized as the generator of the OU process on $G$ driven by the L\'evy process with generator $\Delta^{\!\psi}_G$.
	In Theorem~\ref{prop:P_gamma} we prove that the closure of $(\mathbb{L}_\psi, C^\infty_c(G))$ generates a Feller semigroup $\mathbb{P}^\psi$ on $G$ and it is ergodic with respect to a smooth invariant density $\mathsf{p}_\psi$. In particular, if we denote by $\mathbb{L}$ the operator $\mathbb{L}_\psi$ when $\psi\equiv 0$, then $\mathbb{L}$ is the generator of the OU semigroup corresponding to the horizontal heat semigroup on $G$. This operator can also be realized as the generator of \emph{Mehler's semigroup} in the sense of \cite[Equation (10)]{Lust-Piquard2010}. Unlike in the Euclidean case, the OU semigroup on $G$ is not a self-adjoint operator. In fact, it is not normal on $L^2$ with respect to the invariant distribution, as we show in Corollary~\ref{cor:non-normal}. Using an argument involving the scale invariance property of the sub-Laplacian, see \cite[Lemma~5, Equation (17)]{Lust-Piquard2010}, Lust-Piquard proved that the spectrum of $\mathbb{L}$ on such an $L^2$ space consists of eigenvalues and equals $-\N_0$, a property also satisfied by OU operators on Euclidean spaces, see \cite{MetafunePallaraPriola2002}.

	Due to the presence of non-local perturbations in $\mathbb{L}_\psi$, one cannot rely on scale invariance for $\Delta^{\!\psi}_{G}$. In fact, the scale invariance does not even hold for second order local perturbations, as the scaling index for the horizontal and vertical variables in $G$ are different in general. As a result, the approach in \cite{Lust-Piquard2010} is not applicable to study the spectrum of $\mathbb{L}_\psi$. Instead we rely on intertwining relationships in Theorem~\ref{thm:OU_intertwining} to prove isospectrality, that is, that all operators $\mathbb{L}_\psi$ for $\psi$ with the L\'evy measure satisfying the exponential moment condition \eqref{eq:exponential_moment} have the same eigenvalues with the same algebraic and geometric multiplicities. The intertwining operators are Fourier multiplier operators on $G$, which can also be viewed as Markov operators on $G$. In particular, if $G$ is given a Riemannian structure, then our result implies that the OU operator corresponding to the Laplace-Beltrami operator on $G$ has discrete point spectrum and equals $-\N_0$. Besides the intertwining relationship, another main ingredient in the proof of isospectrality is showing that the generalized eigenfunctions of $\mathbb{L}_\psi$ are polynomials on $G$. This involves proving regularity estimates for the semigroup $\mathbb{P}^\psi$, and we refer to Section~\ref{sec:regularity} for details. We note that the phenomenon of isospectrality is somewhat surprising as the spectrum of $\mathbb{L}_\psi$ does not depend on $\psi$ at all, while the spectrum of $\Delta^{\!\psi}_{G}$ in \eqref{eq:intro_L} indeed depends on $\psi$. We refer to \cite{Sarkar2025} for a similar  observation in a very different context of L\'evy-OU semigroups on Euclidean spaces.

	Such resemblance also reflects deep connections between the OU operators on $G$ and those on  Euclidean spaces. In this vein, we prove the existence of Markov operators $\Pi:L^p(\R^n,\widetilde{\mu})\longrightarrow L^p(G,\mathsf{p}_\psi)$ and $\Lambda_\psi:L^p(\R^m,\mu)\longrightarrow L^p(G,\mathsf{p}_\psi)$ such that 
	\begin{align*}
		\mathbb{P}^\psi_t\Pi&=\Pi \widetilde{P}_t \text{ on }  L^p(\R^n,\widetilde{\mu}) \\
		P_t\Lambda_\psi &=\Lambda_\psi \mathbb{P}^\psi_t \text{ on }   L^p(G,\mathsf{p}_\psi)
	\end{align*}
	where $\widetilde{P}$ (resp. $P$) is a self-adjoint diffusion OU semigroup on $\R^n$ (resp. $\R^m$), and $\widetilde{\mu}$, $\mu$ are the unique invariant distributions of $\widetilde{P}$ and $P$ respectively. We refer to Theorem~\ref{thm:OU_intertwining} and Proposition~\ref{prop:horizontal_intertwining} for details. With the help of the above intertwining relations, we obtain a very different behavior of the spectrum for $p=1$, in which case the spectrum equals the unit disc $\{z\in\mathbb{C}: |z|\le 1\}$, see Theorem~\ref{thm:L_1_spec} for details. This result again conforms with the OU operators on  Euclidean spaces as observed in \cite{MetafunePallaraPriola2002}. Using \eqref{it:intro_2}, we provide explicit formula for the eigenfunctions of the adjoint of $\mathbb{P}^\psi$ for even eigenvalues. To the best of our knowledge, such a formula has not appeared in the literature until now.
	
	The rest of the paper is organized as follows: in Section~\ref{sec:2} we introduce the notation including the basics of Carnot groups. In Section~\ref{sec:GFT_step2} we briefly review the representation theory for Carnot groups of step $2$, which is crucial for defining the generalized Fourier transform. In Section~\ref{sec:4} we prove properties of GFT on Carnot groups of step $2$, extending previous results on Heisenberg groups in \cite{ThangaveluBook1998}. These areessential to study the spectral properties of the sub-Laplacian and its perturbations. In Section~\ref{sec:5} and \ref{sec:6} we study the spectrum of the sub-Laplacian and its L\'evy-type perturbations. Section~\ref{sec:7} is devoted to the spectral theory of L\'evy-OU semigroups on $G$ using intertwining relationships. 
	
	\section{Notation}
	For a simply connected Lie group $G$, we denote by $B_b(G)$, $C_b(G)$, $C^\infty(G)$, $C^\infty_c(G)$, $C_0(G)$ and $C^\infty_0(G)$ the spaces of bounded measurable functions, bounded continuous functions, smooth functions, smooth functions with compact support, continuous functions vanishing at infinity, and smooth functions with all derivatives vanishing at infinity respectively. $C_0(G)$ and $C_b(G)$ are equipped with the supremum norm, that is, for any $f\in C_0(G)$ or $C_b(G)$, we define $\|f\|_\infty:=\sup_{g\in G}|f(g)|$.

	If $\mu$ is a $\sigma$-finite measure on  $\left( G, \mathcal{B}\left( G\right) \right)$, then  by $L^p(G,\mu), p\in [1,\infty]$ we denote the space of complex-valued functions on $G$ that are $L^p$-integrable with respect to $\mu$. We identify $L^\infty(G,\mu)$ with $C_b(G)$. For a unimodular Lie group $G$ we denote by $L^{p}(G)$ the space of $L^p$-integrable functions with respect to a (bi-invariant) Haar measure.

	For any closed operator $\left( A, \mathcal{D}(A) \right)$ on a Banach space $X$, we denote the spectrum of $A$ in $X$ by $\sigma(A; X)$, that is,
	\begin{align*}
		\sigma(A;X)=\{\zeta\in\C: (A-\zeta I) \text{ does not have a bounded inverse} \},
	\end{align*}
	and we denote point, continuous and residual spectrum of $A$ by $\sigma_{\rm p}(A;X), \sigma_{\rm c}(A;X), \sigma_{\rm r}(A;X)$ respectively, where 
	\begin{align*}
		\sigma_{\rm p}(A; X)&:=\{\zeta\in\sigma(A; X): \ker(A-\zeta I)\neq\{0\}\},
		\\
		\sigma_{\rm c}(A; X)&:=\{\zeta\in\sigma(A; X): \ker(A-\zeta I)=\{0\}, (A-\zeta I) \text{ is not surjective and has dense range}\}, 
		\\
		\sigma_{\rm r}(A; X)&:=\sigma(A; X)\setminus (\sigma_{\rm p}(A; X)\cup\sigma_{\rm c}(A; X)).
	\end{align*}
	We also use $\mathbb{N}_{0}$ to denote the space of all non-negative integers. 
	
	\section{Preliminaries}\label{sec:2}
	\subsection{Carnot groups}
	In this paper we concentrate on a particular class of metric spaces, namely, homogeneous Carnot groups of step $2$ equipped with a left-invariant homogeneous metric. The assumption about homogeneity of the group can be made without loss of generality by  \cite[Proposition 2.2.17, Proposition 2.2.18]{BonfiglioliLanconelliUguzzoniBook}. 
	
	Recall that  $G$ is a Carnot group of step $r$ if $G$ is a connected and simply connected Lie group whose Lie algebra $\mathfrak{g}$ is \emph{stratified}, that is, it can be written as
	\begin{align*}
		\mathfrak{g}=V_{1}\oplus\cdots\oplus V_{r},
	\end{align*}
	where
	\begin{equation}\label{e.Stratification}
		\begin{aligned}
			&\left[V_{1}, V_{i-1}\right]=V_{i}, \hskip0.1in 2 \leqslant i \leqslant r,
			\\
			&[ V_{1}, V_{r} ]=\left\{ 0 \right\}.
		\end{aligned}
	\end{equation}
	To exclude trivial cases we assume that the dimension of $\mathfrak{g}$ is at least $3$. In addition we will use a stratification such that the center of $\mathfrak{g}$ is contained in  $V_{r}$. In particular, Carnot groups are nilpotent. We will concentrate on the case of $r=2$. 
	\begin{notation}\label{n.startification} By $\mathcal{H}:=V_{1}$ we denote the space of \emph{horizontal} vectors that generate the rest of the Lie algebra with $\mathcal{V}:=V_{2}=[\mathcal{H},\mathcal{H}], \ldots, V_r = \mathcal{H}^{(r)}$.
	\end{notation}
	For $r=2$, we therefore have the decomposition
	\begin{align}\label{eq:lie_alg_decomp}
		\mathfrak{g}=\mathcal{H} \oplus  \mathcal{V},
	\end{align}
	where for any $Z=(Z_{\mathcal H},Z_{\mathcal V} )$, $Z'=(Z'_{\mathcal H},Z'_{\mathcal V})\in \mathfrak g$ with $Z_{\mathcal H}, Z^\prime_{\mathcal H}\in\mathcal{H}$ and $Z_{\mathcal V}, Z'_{\mathcal V}\in\mathcal{V}$, we have 
	\begin{align*}
		& [Z, Z']=(0, [Z_{\mathcal H}, Z'_{\mathcal H}])
	\end{align*}
	As usual, we let
	\begin{align*}
		\exp&: \mathfrak{g} \longrightarrow G,
		\\
		\log&: G \longrightarrow \mathfrak{g}
	\end{align*}
	denote the exponential and logarithmic maps, which  are global diffeomorphisms for connected nilpotent groups, see for example \cite[Theorem 1.2.1]{CorwinGreenleafBook}. For any $Z\in\mathfrak g$, the left-invariant vector field corresponding to $Z$ is defined as 
	\begin{align*}
		Zf(g)=\left.\frac{d}{dt}\right|_{t=0}f(g\star \exp(tZ)).
	\end{align*}
	Choosing a basis $\{Z_1,\ldots, Z_{n+m}\}$ of $\mathfrak g$ with $\{Z_i: 1\le i\le n\}\subset \mathcal H$ and $\{Z_i: n+1\le i\le n+m\}\in \mathcal V$, this allows us to identify $G$ with a linear space underlying its Lie algebra, $\R^{n}\times\R^{m}$,  equipped with  a non-trivial group law. 
	This is done by using the Baker-Campbell-Dynkin-Hausdorff formula, expressing the group operation in terms of the Lie algebra.  Namely, the group operation on $G$ can be described as 
	\begin{align}\label{eq:group_law}
		& (h_{1}, v_{1} )\star (h_{2}, v_{2} )=(h_{1}+h_{2}, v_{1}+v_{2}+\frac{1}{2}\omega( h_{1}, h_{2})), 
		\\
		& (h_{1}, v_{1} ), (h_{2}, v_{2} ) \in \mathbb{R}^{n+m}\cong G, \notag
	\end{align}
	where $\omega:\R^n\times \R^n\to \R^m$ is a skew-symmetric bilinear map. In this identification, we represent any element $g\in G$ as $g=(h,v)$, where $h\in\R^n, v\in\R^m$, and we call $h$ \emph{horizontal} and $v$ \emph{vertical} components of $g$. Moreover, a Haar measure on $G$ can be chosen to be the Lebesgue measure on $\R^{n+m}$. In the exponential coordinates we just described, one can show that $Z_1,\ldots, Z_{n+m}$ can be expressed as first order differential operators as follows
	\begin{align}\label{eq:V_i}
		Z_i=\frac{\partial}{\partial h_i}+\frac{1}{2}\omega_i(h,\nabla_v), \  Z_{n+j}=\frac{\partial}{\partial v_j}, \quad 1\leqslant i\le n, 1\leqslant j\le m.
	\end{align}
	$(G,\star)$ can also be equipped with dilations $(\delta_c)_{c\in\R}\subset\operatorname{Aut}(G)$ such that for any $(h,v)\in G$,
	\begin{align*}
		\delta_{c}(h,v)=(ch, c^2 v).
	\end{align*}
	Due to the decomposition in \eqref{eq:lie_alg_decomp}, $G$ can be equipped with a sub-Riemannian structure with respect to the horizontal vector fields $\mathcal{H}$ and the metric $g$ on  $\mathcal H$ is chosen in such a way that $Z_1,\ldots, Z_n$ are orthonormal in $\mathcal H$. Then for the Carnot group  equipped with this sub-Riemannian structure, the central object of interest is the sub-Laplacian on $G$  defined as 
	\begin{align}\label{eq:sub-laplacian}
		\Delta_{\mathcal{H}}:=\sum_{i=1}^n Z^2_i.
	\end{align}
	Due to H\"ormander's theorem, $\Delta_{\mathcal H}$ is a hypoelliptic diffusion operator on $G$. By \cite[p.~428]{DriverGrossSaloff-Coste2010} it is known that $(\Delta_{\mathcal{H}}, C^\infty_c(G))$ is essentially self-adjoint in $L^2(G)$, and with abuse of notation we denote its self-adjoint extension in $L^2(G)$ by $\Delta_{\mathcal{H}}$. The corresponding \emph{horizontal heat semigroup} on $L^2(G)$ is defined as $\mathbb{Q}_t:=e^{t\Delta_{\mathcal{H}}}$.
	
	\subsection{Generalized Fourier transform for step 2 Carnot groups}\label{sec:GFT_step2}
	For step $2$ Carnot groups,  we will use an explicit description of the generalized Fourier transform (GFT), which we describe in more detail in  Appendix~\ref{sec:3}. We also refer to thePh.D. thesis of Z.~Yang \cite{YangZPhDThesis2022} for harmonic analysis on step $2$ Carnot groups and its applications. For the sake of completeness and clarity, we briefly describe the construction of irreducible representations and the GFT on step $2$ Carnot groups.
	
	To use Theorem~\ref{t.Kirillov} and Theorem~\ref{t.Plancherel}, one needs to identify
	\begin{enumerate}
		\item polarizing sub-algebras $\mathfrak{h}_l$ for each $l\in\mathfrak{g}^{\ast}$,
		\item \label{step2} the set $S$ of \emph{jump indices}, and its complement $T$ in $\{1,\ldots, n+m\}$,
		\item \label{step3} the set $\mathcal{U}^{\ast}\cap\mathfrak{g}^{\ast}_T$ and the \emph{Plancherel measure} $\Pf(l) dl$.
	\end{enumerate}
	We first describe how \eqref{step2} and \eqref{step3} can be accomplished. Since $G$ is a stratified group of step $2$, we note that $\mathcal{V}\subseteq\operatorname{rad}_{l}$ for any $l\in\mathfrak{g}^{\ast}$. Let $\{Z_1,\ldots, Z_n, Z_{n+1}, \ldots, Z_{n+m}\}$ be a basis of $\mathfrak{g}$ such that $\mathcal{V}=\Span\{Z_{n+1},\ldots, Z_{n+m}\}$. Then we can write $T=\{n+1,\ldots, n+m, n_1,\ldots, n_k\}$. Any $l\in\mathcal{U}^{\ast}\cap\mathfrak{g}^{\ast}_T$ can be written as
	\begin{align}\label{eq:ell}
		l=\sum_{i=1}^m\lambda_i Z^{\ast}_{n+i}+\sum_{j=1}^k \nu_jZ^{\ast}_{n_j}, \quad \lambda\in\R^m, \nu\in\R^k.
	\end{align}
	Since the matrix $K_l\in\mathcal{M}_{n+m}$ defined by $K_l(i,j)=l([Z_i, Z_j])$ only depends on $l\left|\right._{\mathcal{V}}$, \eqref{eq:ell} shows that any $l\in\mathcal{U}^{\ast}\cap\mathfrak{g}^{\ast}_T$ can be identified by $(\lambda,\nu)\in\mathcal{Z}\times\R^k$, where $\mathcal{Z}$ is a Zariski open subset of $\R^m$. Also, we identify $K_l$ and $\operatorname{rad}_l$ simply by $K_\lambda$, $\operatorname{rad}_\lambda$ respectively. We now describe the Plancherel measure. For any $l=(\lambda,\nu)\in \mathcal{Z}\times \R^k$, we can choose an orthonormal basis $\mathcal{B}_\lambda$ of $\mathcal H$ such that
	\begin{align}\label{eq:basis_lambda}
		\mathcal{B}_\lambda=\{X_1(\lambda), Y_1(\lambda), \ldots, X_d(\lambda), Y_d(\lambda), R_1(\lambda), \ldots, R_k(\lambda)\},
	\end{align}
	where $\operatorname{rad}_{\lambda}=\Span\{R_1(\lambda),\ldots, R_k(\lambda)\}$, and with respect to the basis \[\{X_1(\lambda), Y_1(\lambda), \ldots, X_d(\lambda), Y_d(\lambda)\}\]the matrix $\widetilde{K}_l=\widetilde{K}_\lambda$ in Theorem~\ref{t.Plancherel} reduces to 
	\begin{align}\label{eq:K}
		\begin{pmatrix}
			0 & \eta(\lambda) \\
			-\eta(\lambda) & 0 \\
		\end{pmatrix}\in\mathcal{M}_{2d}(\mathbb R),
	\end{align}
	where $\eta(\lambda)=\operatorname{diag}(\eta_1(\lambda), \ldots, \eta_d(\lambda))\in\mathcal{M}_{d}(\R)$, and each $\eta_j(\lambda)>0$ is smooth and homogeneous of degree $1$ in $\lambda=(\lambda_1,\ldots, \lambda_m)$. Therefore the Pfaffian is given by  
	\begin{align}
		\operatorname{Pf}(\lambda)=\prod_{j=1}^d\eta_j(\lambda).
	\end{align}
	Writing an element $g=\exp(X(\lambda)+Y(\lambda)+R(\lambda)+W)\in G$, where 
	\begin{equation}\label{eq:basis_rep}
		\begin{aligned}
			&
			X(\lambda)=\sum_{i=1}^d x_i(\lambda)X_i(\lambda), \quad Y(\lambda)=\sum_{i=1}^d y_i(\lambda)Y_i(\lambda), \\
			& R(\lambda)=\sum_{i=1}^k r_i(\lambda)R_i(\lambda), \quad W=\sum_{i=1}^m v_i Z_{n+i},
		\end{aligned}
	\end{equation}
	we identify $g$ with $(x(\lambda),y(\lambda),r(\lambda),v)\in\R^d\times\R^d\times\R^k\times\R^m$. For each $l=(\lambda,\nu)\in\mathcal{Z}\times \R^k$, let us write 
	\begin{align*}
		\mathcal{P}_\lambda=\Span\{X_1(\lambda),\ldots, X_d(\lambda)\}, \quad \mathcal{Q}_\lambda=\Span\{Y_1(\lambda),\ldots, Y_d(\lambda)\}.
	\end{align*}
	Then $\mathcal{H}=\mathcal{P}_\lambda\oplus\mathcal{Q}_{\lambda}\oplus\operatorname{rad}_\lambda$. Defining
	\begin{align*}
		\mathfrak{h}_l=\Span\{X_1(\lambda),\ldots, X_d(\lambda), R_1(\lambda),\ldots, R_k(\lambda), Z_{n+1},\ldots, Z_{n+m}\},
	\end{align*}
	we note that $\mathfrak{h}_l$ satisfies \eqref{pol1}-\eqref{pol3}. As a result, $\mathfrak{h}_l$ is a polarizing sub-algebra for $l$. The one-dimensional unitary irreducible representations on $\operatorname{H}_l=\exp(\mathfrak{h}_l)$ are given by 
	\begin{align*}
		\rho_l(\exp(X(\lambda)+R(\lambda)+W))=\exp(\i l(X(\lambda)+R(\lambda)+W))=\exp(\i l(R(\lambda))+\i l(W)),
	\end{align*}
	since $l(X(\lambda))=0$. Following the argument described in \cite[p.~31]{ParuiPhDThesis2005}, the induced representation $\pi_l=\Ind^{G}_{\mathbb{H}_l,\rho_l}$ is defined on $L^2(\mathcal{Q}_\lambda)$ and it is given by 
	\begin{align*}
		&\pi_l(\exp(X(\lambda),Y(\lambda),R(\lambda),W))\phi(Y^{\prime}(\lambda))
		\\
		&=e^{\i l(R(\lambda))}e^{\i l(W+[Y^{\prime}(\lambda)+\frac{1}{2}Y(\lambda), X(\lambda)-Y^{\prime}(\lambda)+R(\lambda)])}\phi(Y(\lambda)+Y^{\prime}(\lambda)).
	\end{align*}
	Writing $X(\lambda), Y(\lambda), R(\lambda), W, Y^{\prime}(\lambda)$ using the basis expansion in \eqref{eq:basis_rep}, $\pi_l=\pi_{\lambda,\nu}$ can be realized on $L^2(\R^d)$, and it is given by
	\begin{align}
		\pi_{\lambda,\nu}(g)\phi(\xi)=e^{\i\langle r(\lambda),\nu\rangle}e^{\i\langle\lambda,v\rangle} e^{\i\sum_{j=1}^d\eta_j(\lambda)(y_j(\lambda)\xi_j+\frac{1}{2}x_j(\lambda)y_j(\lambda))}\phi(\xi+x(\lambda)),
	\end{align}
	where 
	\begin{align}\label{eq:coordinates}
		g=(x_1(\lambda),y_1(\lambda),\ldots, x_d(\lambda), y_d(\lambda), r(\lambda),v), 
	\end{align}
	$\phi\in L^2(\R^d)$. Then the GFT on $G$ is defined  as
	\begin{align}\label{eq:fourier_transform}
		\mathcal{F}_{ G}(f)(\lambda,\nu)=\int_{\mathbb G} f(g)\pi_{\lambda,\nu}(g)dg.
	\end{align}	
	For each $(\lambda,\nu)\in\mathcal{Z}\times\R^k$, $\mathcal{F}_G(f)(\lambda,\nu):L^2(\R^d)\longrightarrow L^2(\R^d)$ is a Hilbert-Schmidt operator. The Plancherel formula in Theorem~\ref{t.Plancherel} reduces to the following.
	\begin{proposition}[Plancherel's theorem]\label{prop:plancherel}
		For all $f\in L^1(G)\cap L^2(G)$ we have 
		\begin{align*}
			\int_G |f(g)|^2 dg=\int_{\mathcal{Z}}\int_{\R^k}\|\mathcal{F}_G(f)(\lambda,\nu)\|^2_{\HS} \frac{\Pf(\lambda)}{(2\pi)^{d+k+m}}d\lambda d\nu.
		\end{align*}
	\end{proposition}
	\begin{notation}[Coordinates on \texorpdfstring{$G$}{G}]\label{notation:coordinates}
		Below we use coordinates  $g=(z,r,v)\in G$ in \eqref{eq:coordinates}, where $z\in \C^{2d}, r\in\R^k, v\in\R^m$. 
	\end{notation}
	We end this section with the following lemma describing the behavior of the GFT with respect to the dilation on $G$.
	
	\begin{lemma}\label{lem:dilation}
		For any $f\in L^2(G)$ and $c>0$
		\begin{align*}
			\mathcal{F}_G(\delta_c f)(\lambda,\nu)=c^{-(n+2m)}d_c\mathcal{F}_G(f)(c^{-2}\lambda, c^{-1}\nu) d_{\frac1c},
		\end{align*}
		where $(d_c)_{c>0}$ is the dilation group on $\R^d$ defined by $d_c f(x)=f(cx)$ for any function on $\R^d$.
	\end{lemma}
	\begin{proof}
		First we note that for any $c>0$ and $(z,r,w)\in G$
		\begin{align*}
			d_c \pi_{\lambda,\nu}(z,r, v)=\pi_{c^2\lambda , c\nu}(z,r,v)d_{c}
		\end{align*}
		Now \eqref{eq:fourier_transform} becomes
		\begin{align*}
			\mathcal{F}_G(\delta_c f)(\lambda,\nu)&=\int_{G} f(c z, cr, c^2 v) \pi_{\lambda,\nu}(z,r,v) dzdv\\
			&=c^{-(n+2m)}\int_{G} f(v) \pi_{\lambda,\nu}\left(\frac{z}{c},\frac{r}{c}, \frac{v}{c^2}\right) dz dv \\
			&=c^{-(n+2m)}d_{c}\left(\int_G f(v)\pi_{c^{-2}\lambda , c^{-1}\nu }(z,r, v) dzdv\right) d_{\frac 1c} \\
			&=c^{-(n+2m)}d_{c}\mathcal{F}_G(f)(c^{-2}\lambda,c^{-1}\nu) d_{\frac1c}.
		\end{align*}
	\end{proof}
	
	\subsection{Properties of the GFT on step 2 Carnot groups} \label{sec:4}
	Let $G$ be a Carnot group of step $2$. Using coordinates in Notation~\ref{notation:coordinates}, the GFT of $f\in L^1(G)$ can be written as 
	\begin{align*}
		\mathcal{F}_G(f)(\lambda,\nu)&=\int_{\C^d}\int_{\R^k}\int_{\R^m} f(z,r,v) \pi_{\lambda,\nu}(z,r,v) dz dr dv \\
		&=\int_{\C^d}\int_{\R^k}\int_{\R^m} f(z,r,v)e^{\i\langle r,\nu\rangle}e^{\i\langle\lambda,v\rangle} \pi_{\lambda,\nu}(z,0,0) dz dr dv,
	\end{align*}
	where the last identity follows from the factorization
	\begin{align*}
		\pi_{\lambda,\nu}(z,r,v)=\pi_{\lambda,\nu}(z,0,0)e^{\i\langle r,\nu\rangle}e^{\i\langle\lambda,v\rangle}
	\end{align*}
	for all $(\lambda,\nu)\in\mathcal{Z}\times\R^k$. This motivates the definition of the \emph{Weyl transform} of a function as given below.
	\begin{definition}[Weyl transform] For each $\lambda\in\mathcal{Z}$, the $\lambda$-Weyl transform of $f\in\mathcal{S}(\C^d)$ is defined by 
		\begin{align*}
			W_{\lambda}(f):=\int_{\C^d} f(z)\pi_{\lambda}(z) dz,
		\end{align*}
		where $\pi_{\lambda}(z):=\pi_{\lambda,\nu}(z,0,0)$, and $\mathcal{S}(\C^d)$ denotes the Schwartz class of functions on $\C^{d}$. 
	\end{definition}
	Note that $\pi_{\lambda,\nu}(z,0,0)$ does not depend on $\nu$, which justifies the notation.
	\begin{remark}
		This definition of the Weyl transform is slightly different from the one considered in \cite{YangZPhDThesis2022}, and it is similar to the definition introduced in \cite{ThangaveluBook1998} in the setting of Heisenberg groups.
	\end{remark}
	Let us also consider the Fourier transform with respect to the \emph{radical variables}, that is, for any $f\in L^1(G)$ we define 
	\begin{align}\label{eq:fourier_radical}
		f^{\lambda,\nu}(z):=\int_{\R^k}\int_{\R^m} f(z,r,v) e^{\i\langle r,\nu\rangle}e^{\i\langle v, \lambda\rangle} dv dr.
	\end{align}
	Then by definition of the GFT $\mathcal{F}_G$ in \eqref{eq:fourier_transform}, we have that for all $f\in L^1(G)$
	\begin{align}\label{eq:Fourier_Weyl}
		\mathcal{F}_G(f)(\lambda,\nu)=W_{\lambda}(f^{\lambda,\nu}).
	\end{align}
	\begin{proposition}
		For any $\lambda\in\mathcal{Z}$ and $f\in\mathcal{S}(\C^d)$, $W_{\lambda}(f)$ is a Hilbert-Schmidt operator on $L^2(\R^d)$, and for all $f\in\mathcal{S}(\C^d)$
		\begin{align}\label{eq:Weyl_isometry}
			\frac{\Pf(\lambda)}{(2\pi)^d}\|W_{\lambda}(f)\|^2_{\HS}= \|f\|^2_{L^2(\C^d)}.
		\end{align}
		In addition, \eqref{eq:Weyl_isometry} extends to all $f \in L^2(\C^d)$.
	\end{proposition}
	\begin{proof}
		We provide the proof for completeness, even though it is similar to the case of the Heisenberg groups. For any $f\in\mathcal{S}(\C^d)$, we note that the kernel associated to $W_\lambda$ is given by 
		\begin{align*}
			K_\lambda(\xi,\zeta)=\int_{\R^d} f(x+\i(\xi-\zeta)) e^{\frac{\i}{2}\langle \eta(\lambda)x,(\xi+\zeta)\rangle}dx,
		\end{align*}
		where $\eta(\lambda)=\operatorname{diag}(\eta_1(\lambda),\cdots,\eta_d(\lambda))$, see \eqref{eq:K}. Using a simple change of variable $(\xi-\zeta, \xi+\zeta)\longmapsto (\xi,\zeta)$, we have
		\begin{align*}
			\int_{\R^{2d}}|K_\lambda(\xi,\zeta)|^2 d\xi d\zeta=4^{-d}\int_{\R^{2d}}\left|\int_{\R^d} f(x+\i\xi) e^{\frac{\i}{2}\langle \eta(\lambda)x,\zeta\rangle}dx\right|^2 d\xi d\zeta.
		\end{align*}
		Finally, using the Plancherel theorem for the classical Fourier transform we obtain 
		\begin{align*}
			\int_{\R^d}\left|\int_{\R^d} f(x+\i\xi) e^{\frac{\i}{2}\langle \eta(\lambda)x,\zeta\rangle}dx\right|^2 d\zeta=\frac{(8\pi)^d}{\det(\eta(\lambda))}\int_{\R^d}|f(x+\i\xi)|^2dx.
		\end{align*}
		Since $\det(\eta(\lambda))=\Pf(\lambda)$, it follows that 
		\[
		\int_{\R^{2d}}|K_\lambda(\xi,\zeta)|^2d\xi d\zeta=(2\pi)^d\Pf(\lambda)^{-1}\|f\|^2_{L^2(\C^d)},
		\]
		which proves the proposition.
	\end{proof}
	
	\begin{notation}[Vertical convolutions] For $f\in L^1(G)$ and a finite measure $\mu$ on $\left(\R^m, \mathcal{B}\left(\R^m\right)\right)$, we denote by
		\begin{equation}
			\begin{aligned}
				f {\ast}_v \mu(h, v)&=\int_{\R^m} f(h,v-u)\mu(du).
			\end{aligned}
		\end{equation}
		
\end{notation}
	
	\begin{lemma}\label{lem:vertical_conv}
		For any $f\in L^1(G)$ and a finite measure $\mu$ on $\R^m$, one has $(\lambda,\nu)\in\mathcal{Z}\times\R^k$ and 
		\begin{align*}
			\mathcal{F}_G(f{\ast}_v \mu)(\lambda,\nu)&=\mathcal{F}(\mu)(\lambda)\mathcal{F}_G(f)(\lambda,\nu), \\
			\mathcal{F}_G(f {\ast}_v \overline{\mu} )(\lambda,\nu)&= \mathcal{F}(\mu)(-\lambda)\mathcal{F}_G(f)(\lambda,\nu),
		\end{align*}
		where $\overline{\mu}(B)=\mu(-B)$ for all $B\in\mathcal{B}(\R^m)$.
	\end{lemma}
	\begin{proof}
		Using the convolution formula for the classical Fourier transform we have $(f{\ast}_v\mu)^{\lambda,\nu}=\mathcal{F}(\mu)(\lambda)f^{\lambda,\nu}$ for all $(\lambda,\nu)\in\mathcal{Z}\times\R^k$. Therefore by \eqref{eq:Fourier_Weyl} we obtain $\mathcal{F}_G(f{\ast}_v\mu)(\lambda,\nu)=W_\lambda((f{\ast}_v\mu)^{\lambda,\nu})=\mathcal{F}(\mu)(\lambda)\mathcal{F}_G(f)(\lambda,\nu)$ for all $(\lambda,\nu)\in\mathcal{Z}\times\R^k$. Since $\mathcal{F}(\overline{\mu})(\lambda)=\mathcal{F}(\mu)(-\lambda)$ for all $\lambda\in\R^m$, the second identity follows similarly.
	\end{proof}
	
	\subsubsection{GFT of functions with symmetries}
	The aim of this section is to analyze the GFT of functions which are invariant under orthogonal transformations of the horizontal variables. Using coordinates in Notation~\ref{notation:coordinates}, we consider the following group action on $G$. Let $\mathbb{U}$ denote the group of all $2n\times 2n$ orthogonal matrices of the form
	\begin{align*}
		U=U_1\otimes\cdots\otimes U_d=\begin{pmatrix}
			U_1 &  &  \\
			& \ddots &  \\
			&  & U_d
		\end{pmatrix}, \quad   U_i \text{ is a } 2\times 2 \text{ orthogonal matrix},
	\end{align*}
	and let $\mathbb{U}$ act on $G$ and $\C^d$ as follows
	\begin{align*}
		U\cdot g&=(U_1(x_1,y_1),\ldots, U_d(x_d, y_d), r,v), \quad g=(x_1,y_1,\ldots, x_d, y_d,r,v), \\
		U\cdot z&=(U_1(x_1,y_1),\ldots, U_d(x_d, y_d)), \quad z=(x_1, y_1,\ldots, x_d, y_d).
	\end{align*}
	We note that this group action does not necessarily induce an automorphism on $G$. We write $L^2_{\mathbb{U}}(G)$ (resp. $L^2_{\mathbb{U}}(\C^d)$) to denote the space of all $\mathbb{U}$-invariant functions in $L^2(G)$ (resp. $L^2(\C^d)$), that is, 
	\begin{equation}\label{eq:U_invariant}
		\begin{aligned}
			L^2_{\mathbb{U}}(G)&=\{f\in L^2(G): f(U\cdot g)=f(g) \text{ for all } g\in G\}, 
			\\
			L^2_{\mathbb{U}}(\C^d)&=\{f\in L^2(\C^d): f(U\cdot z)=f(z) \text{ for all } z\in\C^d\}.
		\end{aligned}
	\end{equation}
	We aim to obtain the spectral expansion of the Hilbert-Schmidt operator $\mathcal{F}_G(f)(\lambda,\nu): L^2(\R^d) \longrightarrow L^2(\R^d)$  for any $f\in L^2_{\mathbb{U}}(G)$. For that, we introduce the Laguerre transform of functions in $L^2_{\mathbb{U}}(\C^d)$. For $\beta=(\beta_1,\ldots,\beta_d)\in\N^d_0$ and $(z_1,\ldots, z_d)\in\C^d$, let us define 
	\begin{align*}
		\varphi^\lambda_{\beta}(z_1,\ldots, z_d)=\Pf(\lambda)^{\frac12} (2\pi)^{-\frac{d}{2}}\prod_{j=1}^d L_{\beta_j}\left(\frac{1}{2}\eta_j(\lambda)|z_j|^2\right)e^{-\frac{1}{4}\eta_j(\lambda)|z_j|^2},
	\end{align*}
	where $L_k(x)=e^{x}\frac{d^k}{dx^k}(x^k e^{-x})$ are the Laguerre polynomials. Now, for any $f\in L^1_\mathbb{U}(\C^d)$ and $\beta\in\N^d_0$, we define its \emph{Laguerre transform} by
	\begin{align}
		R^\lambda_{\beta}(f)=\int_{\C^d} f(z_1,\ldots, z_d)\varphi^\lambda_{\beta}(z_1,\ldots, z_d) dz_1\cdots dz_d.
	\end{align}
	Let us also introduce the projection operators $\{E^\lambda_{\beta}: \beta\in\N^d_0\}$ on $L^2(\R^d)$ such that
	\begin{align}\label{eq:E}
		E^\lambda_{\beta} f=\langle f, \Phi^\lambda_{\beta}\rangle \Phi^\lambda_{\beta}, \quad \Phi^\lambda_{\beta}(x)=|\Pf(\lambda)|^{\frac{1}{4}}\Phi_{\beta_1}(\eta_1(\lambda)^{\frac12}x_1)\cdots\Phi_{\beta_d}(\eta_d(\lambda)^{\frac12} x_d),
	\end{align}
	where $\{\Phi_n: n\in\N_0\}$ are orthonormal Hermite functions on $\R$ defined by 
	\begin{align*}
		\Phi_n(x)=(-1)^n (2^n n!\sqrt{\pi})^{-\frac{1}{2}} e^{\frac{x^2}{2}}\frac{d^n}{dx^n}(e^{-x^2}).
	\end{align*}
	\begin{proposition}\label{prop:U_inv}
		For all $f\in L^2_\mathbb{U}(G)$ and $(\lambda,\nu)\in\mathcal{Z}\times\R^k$,
		\begin{align*}
			\mathcal{F}_G(f)(\lambda,\nu)=\frac{(2\pi)^{\frac{d}{2}}}{\Pf(\lambda)^{\frac12}}\sum_{\beta\in\N^d_0} R^\lambda_{\beta}(f^{\lambda,\nu}) E^{\lambda}_{\beta}.
		\end{align*}
	\end{proposition}
	\begin{remark}
		This proposition can be regarded as a version of \cite[Theorem~1.4.3]{ThangaveluBook1998} which provides the spectral expansion of the GFT of radially symmetric functions defined on Heisenberg groups.
	\end{remark}
	\begin{proof}
		We start with the observation that for each $\lambda\in\mathcal{Z}$, $\{\varphi^\lambda_{\beta}: \beta\in\N^d_0\}$ form an orthonormal basis for $L^2_{\mathbb U}(\C^d)$. Observe that by \eqref{eq:Fourier_Weyl} and the isometry relation for the Weyl transform in \eqref{eq:Weyl_isometry}, $f^{\lambda,\nu}\in L^2_\mathbb{U}(\C^d)$ for any $f\in L^2_\mathbb{U}(G)$, and therefore
		\begin{align*}
			f^{\lambda,\nu}=\sum_{\beta\in\N^d_0}\langle f^{\lambda,\nu}, \varphi^\lambda_{\beta}\rangle_{L^2(\C^d)}\varphi^\lambda_{\beta}.
		\end{align*}
		Using the identity $\mathcal{F}_G(f)(\lambda,\nu)=W_{\lambda}(f^{\lambda,\nu})$, we note that for any $f\in L^2_\mathbb{U}(G)$
		\begin{align*}
			\mathcal{F}_G(f)(\lambda,\nu)=\sum_{\beta\in\N^d_0} R^\lambda_\beta(f^{\lambda}) W_{\lambda}(\phi^\lambda_{\beta}).
		\end{align*}
		Finally, we use the identity 
		\begin{align}\label{eq:varphi_beta}
			W_{\lambda}(\varphi^\lambda_{\beta})=\frac{(2\pi)^{\frac d2}}{\Pf(\lambda)^{\frac12}}E^\lambda_\beta,
		\end{align}
		which follows from the proof of \cite[Theorem~3.63]{YangZPhDThesis2022}, to conclude the proof.
	\end{proof}
	As a consequence of the previous proposition, we obtain an orthogonal decomposition of the subspace $L^2_{\mathbb U}(G)$. For $\beta\in\N^d_0$, we define the following subspace of Hilbert-Schmidt operators on $L^2(\R^d)$. 
	\begin{align}\label{e.masthscrLbeta}
		\mathscr{L}_\beta:=\{(\lambda,\nu)\longmapsto \Pf(\lambda)^{-\frac12} f(\lambda,\nu) E^\lambda_\beta: f\in L^2(\mathcal{Z}\times\R^k)\}.
	\end{align}
	
	\begin{proposition}\label{prop:decomp}
		For all $\beta\in\N^d_0$, $\mathscr{L}_\beta\subset L^2\left(\mathcal{Z}\times\R^k,\operatorname{HS}(L^2(\R^d)), \frac{\operatorname{Pf}(\lambda)}{		(2\pi)^{d+k+m}}d\lambda d\nu \right)$, where $\operatorname{HS}(L^2(\R^d))$ is the space of all Hilbert-Schmidt operators on $L^2(\R^d)$. Moreover, one has the following orthogonal decomposition
		\begin{align}\label{eq:L_2_decomp}
			L^2_\mathbb{U}(G)=\bigoplus_{\beta\in\N^d_0} \mathcal{L}_\beta,
		\end{align}
		where $\mathcal{L}_\beta=\mathcal{F}_G^{-1}(\mathscr{L}_\beta)$.
	\end{proposition}
	\begin{proof}
		To check the orthogonality, for any $T\in\mathcal{L}_\beta,  S \in \mathcal{L}_{\beta'}$ with $\beta,\beta'\in\N^d_0, \beta\neq\beta'$, writing $T(\lambda,\nu)=\Pf(\lambda)^{-1/2}f(\lambda,\nu)E^\lambda_\beta$ and $S(\lambda,\nu)=\Pf(\lambda)^{-1/2}g(\lambda,\nu) E^\lambda_{\beta'}$ we have
		\begin{align*}
			&\quad \int_{\R^k}\int_{\mathcal{Z}}\langle T(\lambda,\nu),S(\lambda,\nu)\rangle_{\HS} \Pf(\lambda) d\lambda d\nu \\
			&=\int_{\R^{k}} \int_{\mathcal Z}f(\lambda,\nu)\overline{g(\lambda,\nu)}\langle E^\lambda_\beta, E^\lambda_{\beta'}\rangle_{\HS}d\lambda d\nu \\
			&=0.
		\end{align*}
		On the other hand, by Proposition~\ref{prop:U_inv} for any $f\in L^2_\mathbb{U}(G)$ we have
		\begin{align*}
			\mathcal{F}_G(f)(\lambda,\nu)=\frac{(2\pi)^{\frac d2}}{\Pf(\lambda)^{\frac12}}\sum_{\beta\in\N^d_0} R^\lambda_\beta(f^{\lambda,\nu}) E^\lambda_\beta.
		\end{align*}
		Then by Proposition~\ref{prop:plancherel} it follows that
		\begin{align*}
			&
			\frac{1}{(2\pi)^{k+m}}
			\sum_{\beta\in\N^d_0} \int_{\R^{k}}\int_{\mathcal{Z}}|R^\lambda_\beta(f^{\lambda,\nu})|^2 d\lambda d\nu \\		&=\frac{1}{(2\pi)^{d+k+m}}\int_{\R^{k}}\int_{\mathcal{Z}}\|\mathcal{F}_G(f)(\lambda,\nu)\|^2_{\HS}\Pf(\lambda) d\lambda d\nu 
			\\
			&=\|f\|^2_{L^2(G)}.
		\end{align*}
		Therefore, for each $\beta\in\N^d_0$, $(\lambda,\nu)\longmapsto \Pf(\lambda)^{-\frac12}R^\lambda_\beta(f^{\lambda,\nu})\in L^2(\mathcal{Z}\times \R^k)$. This shows that 
		\begin{align*}
			L^2_{\mathbb{U}}(G)\subseteq \bigoplus_{\beta\in\mathbb{N}^d_0}\mathcal{L}_\beta.
		\end{align*}
		To prove the reverse inclusion, let $T\in\bigoplus_{\beta\in\mathbb{N}^d_0}\mathscr{L}_\beta$. We want to show that there exists $f\in L^2_{\mathbb{U}}(G)$ such that $T(\lambda,\nu)=\mathcal{F}_G(f)(\lambda,\nu)$ for all $(\lambda,\nu)\in\mathcal{Z}\times \R^k$. Let us write 
		\begin{align*}
			T(\lambda,\nu)=\Pf(\lambda)^{-\frac12}\sum_{\beta\in\mathbb{N}^d_0} g_\beta(\lambda,\nu) E^\lambda_\beta,
		\end{align*}
		and define 
		\begin{align*}
			f_\beta(z,r,v)=\frac{1}{(2\pi)^{k+m+\frac d2}}\int_{\mathcal{Z}}\int_{\R^k}  g_\beta(\lambda,\nu) \varphi^\lambda_\beta(z) e^{-\i\langle r,\nu\rangle} e^{-\i\langle v,\lambda\rangle} d\nu d\lambda.
		\end{align*}
		By \eqref{eq:varphi_beta} we have $\mathcal{F}_G(f_\beta)(\lambda,\nu)=W_\lambda(f^{\lambda,\nu}_\beta)=\Pf(\lambda)^{-\frac12}g_\beta(\lambda,\nu) E^\lambda_\beta$, which shows that $f_\beta\in L^2_{\mathbb{U}}(G)$ for all $\beta\in\mathbb{N}^d_0$. Therefore, taking $f=\sum_{\beta\in\mathbb{N}^d_0} f_\beta$, we conclude that $\mathcal{F}_G(f)(\lambda,\nu)=T(\lambda,\nu)$ for all $(\lambda,\nu)\in\mathcal{Z}\times\R^k$. Moreover, $f\in L^2_{\mathbb{U}}(G)$, which completes the proof of the proposition.
	\end{proof}

	\section{Spectrum of the horizontal heat semigroup on $G$ by intertwining}\label{sec:5}
	We now turn to the sub-Laplacian $\Delta_{\mathcal H}$ in \eqref{eq:sub-laplacian} and the horizontal heat semigroup $\mathbb{Q}_t= e^{t\Delta_{\mathcal H}}$. Since $(\Delta_{\mathcal H}, C^\infty_c(G))$ is essentially self-adjoint in $L^2(G)$, $(\mathbb{Q}_t)_{t\geqslant 0}$ is a self-adjoint Markov semigroup on $L^2(G)$. The main result of this section is stated below.
	\begin{theorem}\label{thm:subLaplacian_spectrum}
		For all $t\geqslant 0$, $\sigma(\mathbb{Q}_t; L^2(G))=\sigma_{\rm c}(\mathbb{Q}_t; L^2(G))=[0,1]$.
	\end{theorem}
	Proof of this theorem requires several intermediate steps, including the Fourier analysis of the sub-Laplacian $\Delta_{\mathcal H}$. Since $\Delta_{\mathcal H}$ does not depend on the choice of the orthonormal basis in $\mathcal H$  by \cite[Theorem~3.6]{GordinaLaetsch2016}, we note that for any $\lambda\in\mathcal Z$
	\begin{align*}
		\Delta_{\mathcal{H}}=\sum_{i=1}^d (X^2_i(\lambda)+Y^2_i(\lambda))+\sum_{j=1}^k R^2_j(\lambda),
	\end{align*}
	where $X_i(\lambda), Y_i(\lambda), R_j(\lambda)$ are defined by \eqref{eq:basis_rep}. Using \cite[Lemma~2.4]{Ray2001}, the GFT of the operator $\Delta_{\mathcal{H}}$ applied to $f\in\mathcal{S}(G)$, the Schwartz space of functions on $G$, and for $(\lambda,\nu)\in\mathcal{Z}\times\R^k$, is given by
	\begin{align}\label{eq:sL_fourier}
		\mathcal{F}_G(\Delta_{\mathcal{H}} f)(\lambda,\nu)=\mathcal{F}_G(f)(\lambda,\nu)\circ B_{\lambda,\nu},
	\end{align}
	where $B_{\lambda,\nu}=\Delta_x-\sum_{j=1}^d \eta_j(\lambda)^2 x^2_j-|\nu|^2$ is the Schr\"{o}dinger operator on $\R^d$. As a result, for any $t\geqslant 0$ and $f\in L^2(G)$, the GFT for the heat semigroup $\mathbb{Q}_t$ is given by
	\begin{align}\label{eq:heat_fourier}
		\mathcal{F}_G(\mathbb{Q}_tf)(\lambda,\nu)=\mathcal{F}_G(f)(\lambda,\nu)\circ H^{\lambda,\nu}_t,
	\end{align}
	where $H^{\lambda,\nu}_t := e^{tB_{\lambda,\nu}}$ on $L^2(\R^d)$. It is known that $H^{\lambda,\nu}_t$ is a Hilbert-Schmidt operator with the generalized Hermite functions as its eigenfunctions, that is, for any $\beta\in\N^d_0$
	\begin{equation}\label{eq:harmonic_osc}
		\begin{aligned}
			&H^{\lambda,\nu}_t\Phi^\lambda_\beta=e^{-t\mathfrak{n}(\beta,\lambda,\nu)}\Phi^\lambda_\beta, \\
			&\mathfrak{n}(\beta,\lambda,\nu)=\sum_{j=1}^d (2\beta_j+1)\eta_j(\lambda)+|\nu|^2,
		\end{aligned}
	\end{equation}
	where $\Phi^\lambda_\beta$ is defined by \eqref{eq:E}, and $\eta(\lambda)$ is defined by \eqref{eq:K}.
	$(H^{\lambda,\nu}_t)_{t\geqslant 0}$ also satisfies the following scaling property which follows directly from the scaling property of the Laplacian on $\R^d$ and homogeneity of $\eta_j$.
	\begin{lemma}\label{lem:scaling_Harmonic}
		For all $c>0$ and $(\lambda,\nu)\in\mathcal{Z}\times\R^k$ we have
		\begin{align*}
			H^{c^2\lambda,c\nu}_t=d_c H^{\lambda,\nu}_{tc^2} d_{\frac{1}{c}}.
		\end{align*}
	\end{lemma}
	Let $\mathsf{q}_t(\cdot,\cdot)$ denote the transition kernel associated to horizontal heat semigroup $\mathbb{Q}_t$, that is, for all $f\in C^\infty_c(G)$,
	\begin{align*}
		\mathbb{Q}_tf(g)=\int_{G}\mathsf{q}_t(g,h)f(h)dh,    g, h\in G.
	\end{align*}
	Since $\mathbb{Q}_t$ is left translation invariant, one has $\mathsf{q}_t(g, h)=\mathsf{q}_t(g^{-1}\star h)$ with $\mathsf{q}_t(g)=\mathsf{q}_t(0,g)$. The next proposition provides an explicit formula for $\mathsf{q}_t$, and the proof can be found in \cite[Theorem~4.23]{YangZPhDThesis2022}, and a similar result in \cite[Example~3.1]{AsaadGordina2016}.
	
	\begin{proposition} \label{prop:heat_kernel}
		In coordinates $g=(z,r,v)\in G$  defined by \eqref{eq:coordinates}, we have that  for all $t>0$ and $(\lambda,\nu)\in\mathcal{Z}\times \R^k$ 
		\begin{equation}
			\begin{aligned}
				& (2\pi)^d\int_{\R^m}\int_{\R^k}\mathsf{q}_t(z,r,v) e^{\i\langle\nu,r\rangle} e^{\i\langle\lambda,v\rangle}drdv \\	
				&=e^{-t|\nu|^2}\prod_{j=1}^d \frac{\eta_j(\lambda)}{2\sinh(\eta_j(\lambda)t)}\exp\left(-\sum_{j=1}^d\frac{\eta_j(\lambda)|z_j|^2}{2}\coth(\eta_j(\lambda)t)\right).
			\end{aligned}
		\end{equation}
	\end{proposition}
To prove Theorem~\ref{thm:subLaplacian_spectrum}, we first find subspaces of $L^2(G)$ that are invariant under $\mathbb{Q}_t$ for all $t\geqslant 0$ as follows. 
	\begin{proposition}\label{prop:invariance}
		Recall that $\mathcal{L}_\beta, \beta\in\N^d_0$ is defined by \eqref{e.masthscrLbeta}. Then for all $t\geqslant 0$ we have $\mathbb{Q}_t(\mathcal{L}_\beta)\subseteq \mathcal{L}_\beta$ and moreover
		\begin{align*}
			\mathbb{Q}_t (L^2_{\mathbb{U}}(G))\subseteq L^2_{\mathbb{U}}(G).
		\end{align*}
	\end{proposition}	
	\begin{proof}
		For any $f\in\mathcal{L}_\beta$, we have 
		\begin{align*}
			\mathcal{F}_G(f)(\lambda,\nu)=\Pf(\lambda)^{-\frac12} g(\lambda,\nu) E^\lambda_\beta
		\end{align*}
		for some $g\in L^2(\mathcal{Z}\times\R^k)$. Using the formula in \eqref{eq:heat_fourier} for the GFT of $\mathbb{Q}_t f$ we get 
		\begin{align*}
			\mathcal{F}_G(\mathbb{Q}_t f)(\lambda,\nu)&=\Pf(\lambda)^{-\frac12} g(\lambda,\nu) E^\lambda_\beta\circ H^{\lambda,\nu}_t \\
			&=\Pf(\lambda)^{-\frac12} g(\lambda,\nu) e^{-t\mathfrak{n}(\beta,\lambda,\nu)} E^\lambda_\beta,
		\end{align*} 	
		where the last identity follows from the fact that $H^{\lambda,\nu}_t E^\lambda_\beta=e^{-t\mathfrak{n}(\beta,\lambda,\nu)} E^\lambda_\beta$ for all $(\lambda,\nu)\in\mathcal{Z}\times\R^k$. Since 
		\begin{align*}
			(\lambda,\nu)\longmapsto e^{-t\mathfrak{n}(\beta,\lambda,\nu)} g(\lambda,\nu)\in L^2(\mathcal{Z}\times \R^k),
		\end{align*}
		we conclude that $\mathbb{Q}_t(\mathcal{L}_\beta)\subseteq\mathcal{L}_\beta$. Finally, the second assertion follows by  \eqref{eq:L_2_decomp}.
	\end{proof}	
	Proposition~\ref{prop:invariance} allows us to introduce the following notation 
	\begin{equation}\label{nota:Qtrestricted}
		\begin{aligned}
			\mathbb{Q}^\beta_t&:=\left.\mathbb{Q}_t\right|_{\mathcal{L}_\beta}, \beta\in\N^d_0,
			\\
			\mathbb{Q}^{\mathbb{U}}_t&:=\left.\mathbb{Q}_t\right|_{L^2_{\mathbb{U}}(G)}.
		\end{aligned}
	\end{equation}
	
	Since $\mathbb{Q}_t$ is a self-adjoint operator on $L^2(G)$, by Proposition~\ref{prop:invariance} it follows that $\mathbb{Q}^\beta_t:\mathcal{L}_\beta\longrightarrow \mathcal{L}_\beta$ is also a self-adjoint. In what follows, we will obtain an intertwining relationship between $\mathbb{Q}^\beta_t$ and a strongly continuous contraction semigroup defined on $L^2(\R^{m+k})$.
	To this end, for each $\beta\in\N^d_0$ we define the Fourier multiplier operator $\operatorname{M}_\beta:L^2(\R^{m+k})\longrightarrow L^2(G)$ such that

\begin{align}\label{eq:M_beta}
		\mathcal{F}_G(\operatorname{M}_\beta f)(\lambda,\nu)=(2\pi)^{\frac{d}{2}}\Pf(\lambda)^{-\frac{1}{2}}\mathcal{F}(f)(\lambda,\nu)E^\lambda_\beta,
	\end{align}
	where $E^\lambda_\beta$ is defined by \eqref{eq:E}. Note that $\Range(\operatorname{M}_\beta)\subseteq\mathcal{L}_\beta$, where $\mathcal{L}_\beta$ is defined in \eqref{e.masthscrLbeta}. By Proposition~\ref{prop:plancherel}, $\operatorname{M}_\beta:L^2(\R^{m+k})\longrightarrow\mathcal{L}_\beta$ is a unitary operator. Next, for each $\beta\in\N^d_0$, let us also define the semigroup $Q^\beta_t$ on $L^2(\R^{m+k})$ such that for all $f\in L^2(\R^{m+k})$, 
	\begin{align*}
		\mathcal{F}(Q^\beta_t f)(\lambda,\nu)=e^{-t\mathfrak{n}(\beta,\lambda,\nu)}\mathcal{F}(f)(\lambda,\nu), \quad (\lambda,\nu)\in\R^{m+k},
	\end{align*}
where $\mathfrak{n}(\beta,\lambda,\nu)$ is defined in \eqref{eq:harmonic_osc}. Since $Q^\beta_t$ is unitarily equivalent to a multiplication operator on $L^2(\R^{m+k})$, it is self-adjoint with purely continuous spectrum. More precisely, for all $t>0$, 
	\begin{align*}
		\sigma(Q^\beta_t; L^2(\R^{m+k}))=\sigma_{\rm c}(Q^\beta_t; L^2(\R^{m+k}))=\{e^{-t\mathfrak{n}(\beta,\lambda,\nu)}: (\lambda,\nu)\in\R^{m+k}\}=[0,1].
	\end{align*}
	Then we have the following intertwining relationship.
	\begin{proposition}\label{thm:heat_intertwining}
		For all $t\geqslant 0$, $f\in L^2(\R^{m+k})$, and $\beta\in\mathbb{N}^d_0$ we have
		\begin{align*}
			\mathbb{Q}^\beta_t\operatorname{M}_\beta f=\operatorname{M}_\beta Q^\beta_t f,
		\end{align*}
		where $\mathrm{M}_\beta$ is defined in \eqref{eq:M_beta}.
	\end{proposition}
	\begin{remark}\label{rem:poisson} When $\mathbb G=\R^{2d}\times\R^m$ is an $H$-type group (see \cite[p.~681, Definition~18.1.1]{BonfiglioliLanconelliUguzzoniBook}), the semigroup $Q^\beta$ is generated by the fractional Laplacian $(2|\beta|+d)\Delta^{\frac12}$ on $\R^m$. When $\mathbb G=\R^{2d}\times\R$ is a non-isotropic Heisenberg group (see \cite[Definition~1.1]{GordinaLuo2022}) with symplectic form $\omega:\R^{2d}\to\R$ such that 
		$\omega((x,y), (x',y'))=\sum_{i=1}^da_i(x_iy'_i-x'_iy_i)$, then $Q^\beta$ is generated by $\sum_{j=1}^d(2\beta_j+1)a_j \Delta^{\frac12}$ on $\R$.
	\end{remark}
	
	\begin{proof}[Proof of Proposition~\ref{thm:heat_intertwining}]
		For any $t>0$ and $f\in L^2(\R^{m+k})$ we have
		\begin{align*}
			\mathcal{F}_G(\mathbb{Q}^\beta_t \operatorname{M}_\beta f)(\lambda,\nu)&=\mathcal{F}_G(\operatorname{M}_\beta f)(\lambda,\nu) H^{\lambda,\nu}_t \\
			&=\mathcal{F}(f)(\lambda,\nu) \Pf(\lambda)^{-\frac12}E^\lambda_\beta H^{\lambda,\nu}_t \\
			&=e^{-t\mathfrak{n}(\beta,\lambda,\nu)} \mathcal{F}(f)(\lambda,\nu) \Pf(\lambda)^{-\frac12}E^\lambda_\beta \\
			&=\mathcal{F}_G(\operatorname{M}_\beta Q^\beta_t f)(\lambda,\nu).
		\end{align*}
		Therefore, proof of the proposition follows by invoking the injectivity of GFT.
	\end{proof}
	\begin{proof}[Proof of Theorem~\ref{thm:subLaplacian_spectrum}] Let $L^2_{\mathbb{U}}(G)$ be defined by \eqref{eq:U_invariant}, and $\mathbb{Q}^{\mathbb{U}}_t$ by \eqref{nota:Qtrestricted}. 		We note that $\sigma(\mathbb{Q}^\mathbb{U}_t; L^2(G))\subseteq\sigma(\mathbb{Q}_t; L^2(G))\subseteq [0,1]$. Using Proposition~\ref{prop:decomp} and Proposition~\ref{thm:heat_intertwining} we have 
		\begin{gather*}
			\mathbb{Q}^\mathbb{U}_t  =\bigoplus_{\beta\in\N^d_0}^\infty  \mathbb{Q}^\beta_t,
			\\
			\sigma(\mathbb{Q}^\beta_t; \mathcal{L}_\beta) =\sigma_{\rm c}(\mathbb{Q}^{\beta}_t; \mathcal{L}_\beta)=\sigma_{\rm c}(Q^\beta_t; L^2(\R^{m+k}))=\sigma(Q^\beta_t; L^2(\R^{m+k}))=[0,1].
		\end{gather*}
		This proves that $\sigma(\mathbb{Q}_t; L^2(G))=\sigma(\mathbb{Q}^\mathbb{U}_t; L^2(G))=\sigma_{\rm c}(\mathbb{Q}^\mathbb{U}_t; L^2(G))=[0,1]$. It only remains to prove that the spectrum of $\mathbb{Q}_t$ is purely continuous. Since $\mathbb{Q}_t$ is a self-adjoint operator, its residual spectrum is empty. We need to show that $\sigma_{\rm p}(\mathbb{Q}_t; L^2(G))=\emptyset$. Let $e^{-a t}\in\sigma_{\rm p}(\mathbb{Q}_t; L^2(G))$ for some $a\geqslant 0$. Then, there exists a nonzero $f\in L^2(G)$ such that 
		\begin{align*}
			\mathbb{Q}_tf=e^{-a t}f.
		\end{align*}
		Taking the GFT on both sides of the above equation one gets that for any $(\lambda,\nu)\in\mathcal{Z}\times \R^k$
		\begin{align*}
			\mathcal{F}_G(\mathbb{Q}_t f)(\lambda,\nu)=\mathcal{F}_G(f)(\lambda,\nu)\circ H^{\lambda,\nu}_t=e^{-at}\mathcal{F}_G(f)(\lambda,\nu).
		\end{align*}
		Since, according to \eqref{eq:harmonic_osc}, $\sigma(H^{\lambda,\nu}_t; L^2(\R^{d}))\setminus\{0\}=\sigma_{\rm p}(H^{\lambda,\nu}_t; L^2(\R^d))=\{e^{-t\mathfrak{n}(\beta,\lambda,\nu)}: \beta\in\N^d_0\}$, the above equation implies that $\mathcal{F}_G(f)(\lambda,\nu)=0$ whenever $a\notin \{\mathfrak{n}(\beta,\lambda,\nu): \beta\in\N^d_0, \lambda\in\mathcal{Z}, \nu\in\R^k\}$, where $\mathfrak{n}(\beta,\lambda,\nu)$ is defined by \eqref{eq:harmonic_osc}. For any $a\geqslant 0$ and a fixed $\beta\in\N^d_0$, let us denote
		\begin{align*}
			\vartheta(a):=\operatorname{Leb}\{(\lambda,\nu)\in\R^{m+k}: \mathfrak{n}(\beta,\lambda,\nu)\le a\}.
		\end{align*}
		Since $\eta_i$ in \eqref{eq:harmonic_osc} is smooth and homogeneous of degree $1$ for each $i$, and $\nu\mapsto |\nu|^2$ is homogeneous of degree $2$, we have $\vartheta(a)=a^{m+\frac k2}\vartheta(1)$ for all $a\geqslant 0$. From continuity of $\vartheta$, it follows that for all $a\geqslant 0$, 
		\begin{align*}
			\operatorname{Leb}\{(\lambda,\nu)\in\R^{m+k}: \mathfrak{n}(\beta,\lambda,\nu)=a \text{  for some } \beta\in\N^d_0\}=0,
		\end{align*}
		which implies that $\mathcal{F}_G(f)(\lambda,\nu)=0$ on a set of full Lebesgue measure. This forces that $f=0$, and therefore a contradiction. Hence, $\sigma(\mathbb{Q}_t; L^2(G))=\sigma_{\rm c}(\mathbb{Q}_t; L^2(G))$.
	\end{proof}

	\section{Spectrum of some L\'evy semigroups on \texorpdfstring{$G$}{G}} \label{sec:6}
	\subsection{L\'evy processes on Lie groups and Euclidean spaces} 
	\begin{definition}
		A stochastic process $X=(X_t)_{t\ge 0}$ on $G$ is called a \emph{L\'evy process} if 
		\begin{enumerate}
			\item $X_t$ has stationary and independent left-increment, that is, for any $0\le t_1<\cdots<t_n$, $X_{t_1}, X^{-1}_{t_1}\star X_{t_2}, \ldots, X^{-1}_{t_{n-1}} \star X_{t_n}$ are independent random variables, and 
			\begin{align*}
				X^{-1}_{s}\star X_{t+s}\overset{d}{=} X_t \text{ for all } s, t>0.
			\end{align*}
			\item $t\mapsto X_t$ is stochastically continuous, that is, $\lim_{h\to 0}X^{-1}_h\star X_{t+h}=X_t$ in probability for any $t>0$.
		\end{enumerate}
	\end{definition}
	From the definition, it is clear that the semigroup $\widetilde{\mathbb{Q}}_t$ associated to $X$ is  invariant under left translation on $G$. In other words, a L\'evy semigroup $(\widetilde{\mathbb{Q}}_t)_{t\geqslant 0}$ on $C_0(G)$ is a left-invariant Feller Markov semigroup, that is, for any $g\in G$ and $t\geqslant 0$,
	\begin{align*}
		\tau_g \widetilde{\mathbb{Q}}_t=\widetilde{\mathbb{Q}}_t\tau_g,
	\end{align*}
	where $\tau_g f(h):=f(g\star h)$ is the left translation operator on $C_0(G)$. For example, the horizontal heat semigroup on $G$ is a L\'evy semigroup. The following result is due to G.A. Hunt \cite{Hunt1956a} which characterizes all L\'evy semigroups on $C_0(G)$ through their generators. We provide a slightly different version of Hunt's original result which can be found in \cite{LiaoMingBookLevyProcessesinLieGroups}.
	\begin{theorem}[Theorem~1.1 in \cite{LiaoMingBookLevyProcessesinLieGroups}]\label{thm:Hunt} Let $L$ be the generator of a left-invariant Feller semigroup of probability kernels on a Lie group $G$. Then its domain $\mathcal{D}(L)$ contains $C^{2,l}_0(G)$, and for all $f\in C^{2,l}_0(G)$ and $g\in G$, 
		\begin{equation}\label{eq:Levy-generator}
			\begin{aligned}
				L f(g)=\frac{1}{2}&\sum_{j,k=1}^d a_{jk} X^l_j X^l_kf(g)+\sum_{i=1}^dc_iX^l_i f(g) \\
				&+\int_G \left[ f(g\star h)-f(g)-\sum_{i=1}^d x_i(h) X^l_i f(g)\right]\Pi(dh),
			\end{aligned}
		\end{equation}
		where $C^{2,l}_0(G)$ denotes the space of all functions twice differentiable with respect to the left-invariant vector fields and with vanishing derivatives at infinity, $(a_{ij})$ and $c_i$ are constants with $\{a_{ij}\}$ being nonnegative definite symmetric matrix, $(X^l_1,\ldots, X^l_d)$ is a basis of left invariant vector fields in $\mathfrak{g}$, the Lie algebra of $G$ and $x_1, \ldots, x_d$ are coordinate functions satisfying $x_i(e)=0$ and $X_ix_j(e)=\delta_{ij}$, $e$ being the identity element in $G$. Moreover, $\Pi$ is a measure on $G$ that satisfies 
		\begin{align*}
			\Pi(\{e\})=0, \quad \sum_{i=1}^d \int_U x^2_i(g)\Pi(dg)<\infty, \quad \Pi(U^c)<\infty
		\end{align*}
		for any open neighborhood $U$ of $e$.
	\end{theorem}
It is worth noting that the above representation of the generator of L\'evy semigroups does not require any metric structure on $\mathfrak{g}$. Although the above theorem is a very powerful result, it does not provide information about the core of the generator. Moreover, it is in general difficult to understand the spectrum of an operator in the form \eqref{eq:Levy-generator} as the GFT of such operators are not tractable objects. In what follows, we will consider L\'evy semigroups on $G$ which are obtained by perturbing the sub-Laplacian $\Delta_{\mathcal H}$ along the directions of vertical variables. To this end, we recall that the generator of a L\'evy semigroup on $\R^m$ is a second order integro-differential operators with the following representation.
	\begin{equation}\label{eq:Levy}
		\begin{aligned}
			Af(v)&=\trace(\sigma\nabla^2 f(v))-\langle b, \nabla f(v)\rangle \\
			&+\int_{\mathbb{R}^m}\left(f(v+v^{\prime})-f(v)-\langle \nabla f(v), v^{\prime}\rangle\mathbbm{1}_{\{|v^{\prime}|\le 1\}}\right)\kappa(dv^{\prime}),
		\end{aligned}
	\end{equation}
	where $\sigma$ is a nonnegative definite matrix, $b\in\R^m$, and $\kappa$ is a L\'evy measure on $\mathcal{B}(\R^m)$, that is, 
	\begin{align*}
		\kappa(\{0\})=0 \quad \text{and} \quad \int_{\mathbb{R}^m} (1\wedge |v|^2)\kappa(dv)<\infty.
	\end{align*}
	A direct computation shows that $A$ is a Fourier multiplier operator with symbol $\psi$, that is, for all $f\in C^\infty_c(\R^m)$ and $\lambda\in\R^d$,
	\begin{align*}
		\mathcal{F}(Af)(\lambda)=\psi(\lambda)\mathcal{F}(f)(\lambda),
	\end{align*}
	where 
	\begin{align}\label{eq:psi}
		\psi(\lambda):=-\langle \sigma\lambda,\lambda\rangle+\i\langle b,\lambda\rangle+\int_{\R^m}(e^{\i\langle v, \lambda\rangle}-1-\i\langle v,\lambda\rangle\mathbbm{1}_{\{|v|\le 1\}})\kappa(dv).
	\end{align}
	$\psi$ is called the \emph{L\'evy-Khintchine exponent} associated to the semigroup generated by $A$. 
	\begin{notation}\label{notation:N}
		We denote by $\mathcal{N}$ the set of all L\'evy-Khintchine exponents $\psi$ that are of the form \eqref{eq:psi}. We note that any $\psi\in\mathcal{N}$ can be uniquely identified with the triple $(\sigma,b,\kappa)$, and we often write $\psi=(\sigma,b,\kappa)$. 
	\end{notation}
	From now on, we denote the operator $A$ in \eqref{eq:Levy} by $A^\psi$ to indicate the dependence on the L\'evy-Khintchine exponent. It is known by \cite[Theorem~31.5]{SatoBook1999} that $(A^\psi, C^\infty_c(\R^m))$ generates a Feller Markov semigroup $(Q^\psi_t)_{t\geqslant 0}$ on $C_0(\R^m)$ and for all $f\in L^1(\R^m)\cap C_0(\R^m)$ and $\lambda\in\R^m$, 
	\begin{align}\label{eq:Levy-semigroup}
		\mathcal{F}(Q^\psi_t f)(\lambda)= e^{t\psi(\lambda)}\mathcal{F}(f)(\lambda).
	\end{align}
	$(Q^\psi_t)_{t\geqslant 0}$ is called the \emph{L\'evy semigroup} corresponding to the L\'evy-Khintchine exponent $\psi$. When $\psi(\lambda)=-|\lambda|^2$, $(Q^\psi_t)_{t\geqslant 0}$ is the classical heat semigroup on $\R^m$. We denote the transition kernel of $Q^\psi_t$ by $\mu^\psi_t$, that is, for all $\lambda\in\R^m$,
	\begin{equation}\label{eq:mu_alpha}
		\begin{aligned}
			e^{t\psi(\lambda)}&=\int_{\R^m} e^{\i\langle\lambda, v\rangle}\mu^\psi_t(dv).
		\end{aligned}
	\end{equation}

	\subsection{L\'evy perturbations of a sub-Laplacian}
	For any $\psi\in\mathcal N$, we consider perturbations of $\Delta_{\mathcal{H}}$ defined by 
	\begin{align*}
		\Delta^{\!\psi}_{G}:=\Delta_{\mathcal{H}}+A^\psi_{\mathcal{V}},
	\end{align*}
	where $A^\psi_{\mathcal{V}}$ denotes $A^\psi$ in the vertical direction of $G$. Note that when $\psi(\lambda)=-|\lambda|^2$, $\Delta^{\!\psi}_{G}$ is the full Laplacian on $G$. By \cite[Theorem~1.1]{LiaoMingBookLevyProcessesinLieGroups}, it follows that $\Delta^{\!\psi}_{G}$ coincides with the generator of a L\'evy semigroup on $G$ at least on the space of all twice differentiable functions with respect to the left-invariant vector fields. The following theorem describes a core for the generator and also gives an explicit formula for the semigroup in terms of the horizontal heat semigroup $(\mathbb{Q}_t)_{t\geqslant 0}$.
	
	\begin{theorem}\label{thm:heat_perurbation} For any $\psi\in\mathcal N$, the closure of $(\Delta^{\!\psi}_{G}, C^\infty_c(G))$ generates a L\'evy semigroup $(\mathbb{Q}^\psi_t)_{t\geqslant 0}$ on $C_0(G)$ such that for any $f\in C_0(G)$
		\begin{align}\label{eq:Q_psi}
			\mathbb{Q}^\psi_t f(h,v)=\int_{\R^m}\mathbb{Q}_t f(h,v-u) \mu^\psi_t(du).
		\end{align}
		The heat kernel for $\mathbb{Q}^\psi_t$, $\mathsf{q}^\psi_t(g):=\mathsf{q}^\psi_t(0,g)$ satisfies
		\begin{equation}\label{eq:heat_kernel_perturbed}
			\begin{aligned}
				& (2\pi)^{d}\int_{\R^m}\int_{\R^k}\mathsf{q}^\psi_t(z,r,v) e^{\i\langle\nu,r\rangle} e^{\i\langle\lambda,v\rangle}drdv \\ &=e^{-t|\nu|^2}e^{t\psi(\lambda)}\prod_{j=1}^d\frac{\eta_j(\lambda)}{2\sinh(\eta_j(\lambda)t)}\exp\left(-\sum_{j=1}^d\frac{\eta_j(\lambda)|z_j|^2}{2}\coth(\eta_j(\lambda)t)\right),
			\end{aligned}
		\end{equation}
		where $g=(z,r,v)$ are the coordinates on $G$ given by \eqref{eq:coordinates}.
	\end{theorem}
	
	\begin{remark}
		The Markov process $(X^\psi(t))_{t\geqslant 0}$ associated to the semigroup $\mathbb{Q}^\psi_t$ has an explicit pathwise representation. Let $B(t)=(B_1(t),\ldots, B_n(t))_{t\geqslant 0}$ be $n$-dimensional standard Brownian motion and $(Y^\psi(t))_{t\geqslant 0}$ denote the L\'evy process on $\R^m$ associated to $\psi\in\mathcal N$, independent of $(B(t))_{t\geqslant 0}$. Then for all $t\geqslant 0$,
		\begin{align*}
			X^\psi(t)=\left(B(t), Y^\psi(t)+\int_0^t \omega(B(t), dB(t))\right).
		\end{align*}
	\end{remark}
	\begin{remark}\label{remark:2}Since $\mathbb{Q}^\psi=(\mathbb{Q}^\psi_t)_{t\geqslant 0}$ is a convolution semigroup on $G$, the Haar measure on $G$ is invariant with respect to $\mathbb{Q}^\psi$, and therefore $\mathbb{Q}^\psi$ extends to a strongly continuous Markov semigroup on $L^2(G)$. By a routine approximation argument using mollifiers, it can be shown that $C^\infty_c(G)$ is a core for the $L^2(G)$-generator of $\mathbb{Q}^\psi$.
	\end{remark}
	The above representation for $X^\psi$ indicates that the distribution of the horizontal components of $X^\psi$ does not depend on $\psi$. Alternatively, if $\Pi: C_0(\R^n)\longrightarrow C_0(G)$ denotes the \emph{lifting} operator defined by 
\begin{align}\label{eq:def_Pi}
		\Pi f(h,v)=f(h), \quad (h,v)\in G,
	\end{align}
	then for all $f\in C_0(\R^n)$ and $\psi\in\mathcal N$,
	\begin{align}
		\mathbb{Q}^\psi_t\Pi f=\Pi Q_t f \text{ for any } t\geqslant 0,
	\end{align}
	where $(Q_t)_{t\geqslant 0}$ is the classical heat semigroup on $\R^n$.
	In the next theorem, we provide a characterization for the class of the Markov semigroups $\{\mathbb{Q}^\psi: \psi\in\mathcal N\}$.
	\begin{theorem}\label{thm:characterization}
		Let $(\widetilde{\mathbb{Q}}_t)_{t\geqslant 0}$ be a L\'evy semigroup on $C_0(G)$ such that for all $f\in C_0(\R^n)$ and $t\geqslant 0$
		\begin{align}
			\widetilde{\mathbb{Q}}_t \Pi f=\Pi Q_t f.
		\end{align}
		Then there exists a $\psi\in\mathcal N$ such that $\widetilde{\mathbb{Q}}_t=\mathbb{Q}^\psi_t$ for all $t\geqslant 0$.
	\end{theorem}
	We need the following lemma to prove this theorem.
	\begin{lemma}\label{lem:core_Q}
		Let $\mathcal{C}\subset C_0(G)$ denote the class of functions such that 
		\begin{align}\label{eq:C(G)}
			\mathcal{C}=\left\{f\in C^\infty_0(G): \lim_{h,v\to\infty}|h|^{2} \frac{\partial f}{\partial h^{\beta_1}\partial v^{\beta_2}}=0,  \beta_1+\beta_2\leqslant 2\right\},
		\end{align}
		Then $\mathcal{C}\subset \mathcal{D}(\Delta_{\mathcal{H}})$ and $\mathbb{Q}_t (\mathcal{C})\subseteq \mathcal{C}$ for all $t\geqslant 0$. As a consequence, $C^\infty_c(G)$ is a core for the operator $\Delta_{\mathcal{H}}$.
	\end{lemma}
	\begin{proof} Writing the skew-symmetric bilinear map $\omega$ in \eqref{eq:group_law} as 
		\begin{align*}
			\omega(h,h')=(\langle A_1h, h'\rangle, \ldots, \langle A_m h, h'\rangle),
		\end{align*}
		where $A_1,\ldots, A_m$ are real skew-symmetric matrices, we observe that the sub-Laplacian $\Delta_{\mathcal{H}}$ on $G$ can be written as
		\begin{align*}
			\Delta_{\mathcal{H}}=\Delta_{h}+\frac{1}{4}\sum_{l_1,l_2=1}^m\langle A_{l_1} h, A_{l_2} h\rangle\frac{\partial^2}{\partial v_{l_1} \partial v_{l_2}} 
			+\sum_{l=1}^m \langle B_lh, \nabla_h\rangle\frac{\partial}{\partial v_l}.
		\end{align*}
		Thus for any $f\in\mathcal{C}$, $\Delta_{\mathcal{H}}f\in C_0(G)$. We need to show that $\mathcal{C}\subset\mathcal{D}(\Delta_{\mathcal{H}})$. We recall that the Brownian motion on $G$ can be described by the stochastic differential equations
		\begin{align*}
			dX_i(t)&=dB_i(t) \text{ for any } 1\le i\le n, \\
			(dY_{1}(t),\ldots, dY_{m}(t))&=\omega(X_1(t),\ldots, X_n(t), dB_1(t), \ldots, dB_n(t)).
		\end{align*}
		Therefore by It\^{o}'s formula for any $f\in C^2(G)$
		\begin{align*}
			M_f(t)=f(X(t))-f(X(0))-\int_0^t \Delta_{\mathcal{H}} f(X(s)) ds
		\end{align*}
		is a $\operatorname{P}_{\!g}$-martingale for every $g\in G$. As $M_f(0)=0$, we have $\operatorname E_g(M_f(t))=0$ for all $t\geqslant 0$ and $g\in G$, which shows that
		\begin{align*}
			\lim_{t\searrow 0} \frac{\mathbb{Q}_tf(g)-f(g)}{t}=\Delta_{\mathcal{H}}f(g) \text{ for any } g\in G.
		\end{align*}
		Therefore, $f\in\mathcal{D}(\Delta_{\mathcal{H}})$ for any $f\in\mathcal{C}$.
		We also note that the heat kernel $\mathsf{q}_t$ of $\mathbb{Q}_t$ has the following space-time scaling property
		\begin{align*}
			\mathsf{q}_t(g)=t^{-\frac{n+2m}{2}}\mathsf{q}_1\circ \delta_{\frac{1}{\sqrt{t}}}(g) \text{ for any } g\in G,
		\end{align*}
		and therefore for all $f\in\mathcal{C}$ one has
		\begin{align*}
			\mathbb{Q}_tf(g)=\int_G f(g\star\delta_{\sqrt{t}}(\widetilde{g}^{-1}))\mathsf{q}_1(\widetilde{g})d\widetilde{g}.
		\end{align*}
		Writing $g=(h,v), \widetilde{g}=(\widetilde{h}, \widetilde{g})$ with $h, \widetilde{h}\in\R^n$ and $v, \widetilde{v}\in\R^m$, the above expression becomes
		\begin{align*}
			\mathbb{Q}_tf(h,v)=\int_{G}f(h-\sqrt{t}\widetilde{h},v-t\widetilde{v}-\sqrt{t}\omega(h, \widetilde{h}))\mathsf{q}_1(\widetilde{h}, \widetilde{v})d\widetilde{h}d\widetilde{v}.
		\end{align*}
		Since both $f$ and $\mathsf{q}_t$ are smooth and $\mathsf{q}_t$ has moments of all order, the above expression also implies that $\mathbb{Q}_t(\mathcal C)\subseteq \mathcal{C}$ for all $t\geqslant 0$. By \cite[Lemma~31.6]{SatoBook1999} we see that $\mathcal{C}$ is a core for $\Delta_{\mathcal{H}}$. Let $\varphi\in C^\infty_c(\R)$ be such that $\varphi\equiv 1$ on $[-1,1]$ and vanishes outside $[-2,2]$. Consider $f_n(g)=f(g)\varphi(|g|^2/n^2)$. Clearly, for all $\beta_1+\beta_2\leqslant 2$ 
		\begin{align*}
			&|h|^{2}\frac{\partial f_n}{\partial h^{\beta_1} \partial v^{\beta_2}}\xrightarrow[n \to \infty]{} |h|^{2} \frac{\partial f}{\partial h^{\beta_1} \partial v^{\beta_2}} \text{ in } C_0(G).
		\end{align*}
		As a result, for any $f\in\mathcal{C}$, both $f_n$ and $\Delta_{\mathcal{H}} f_n$ converge to $f, \Delta_{\mathcal{H}} f$ respectively in $C_0(G)$. This proves that $C^\infty_c(G)$ is also a core for $\Delta_{\mathcal{H}}$, and completes the proof of the lemma.
	\end{proof}
	
	\begin{proof}[Proof of Theorem~\ref{thm:heat_perurbation}]First, we show that $\mathbb{Q}^\psi_t$ is indeed a Markov semigroup on $C_0(G)$. Since $\mathbb{Q}_t$ is a Markov operator, \eqref{eq:Q_psi} shows that $\mathbb{Q}^\psi_t$ is a Markov operator as well. To prove the semigroup property, for any $t,s>0$, $f\in \mathcal{S}(G)$ and $(\lambda,\nu)\in\mathcal{Z}\times \R^k$ we have
		\begin{align*}
			\mathcal{F}_G(\mathbb{Q}^\psi_t \mathbb{Q}^\psi_sf)(\lambda,\nu)&=e^{t\psi(\lambda)}\mathcal{F}_G(\mathbb{Q}^\psi_sf)(\lambda,\nu) H^{\lambda,\nu}_t \\
			&= e^{t\psi(\lambda)}e^{s\psi(\lambda)}\mathcal{F}_G(f)(\lambda,\nu) H^{\lambda,\nu}_s H^{\lambda,\nu}_t \\
			&= e^{(t+s)\psi(\lambda)} \mathcal{F}_G(f)(\lambda,\nu) H^{\lambda,\nu}_{t+s} \\
			&=\mathcal{F}_G(\mathbb{Q}^\psi_{t+s}f)(\lambda,\nu).
		\end{align*}
		By injectivity of GFT, we have $\mathbb{Q}^\psi_{t+s}f=\mathbb{Q}^\psi_t\mathbb{Q}^\psi_sf$ for all $f\in\mathcal{S}(G)$. Finally, using the density of $\mathcal{S}(G)$ in $C_0(G)$, the semigroup property follows. Next, we prove that $\Delta^{\!\psi}_{G}$ is the generator of $\mathbb{Q}^\psi$ with $C^\infty_c(G)$ as its core. We note that for any $f\in\mathcal C$, $\Delta^{\!\psi}_{G}f\in C_0(G)$. Next, we proceed to show that $\mathcal{C}\subset\mathcal{D}(\Delta^{\!\psi}_{G})$. For any $t\geqslant 0$, 
		we can write 
		\begin{equation}
			\begin{aligned}
				\mathbb{Q}^{\psi}_tf(h,v)-f(h,v)&=\int_{\R^m}\left(\mathbb{Q}_tf(h,v-v')-f(h,v-v')\right)\mu^{\psi}_t(dv') \\
				&+\int_{\R^m} f(h, v-v')\mu^{\psi}_t(dv')-f(h,v).
			\end{aligned}
		\end{equation}
		Since $\mu^\psi_t$ is the transition kernel of a L\'evy process with the L\'evy-Khintchine exponent $\psi$, and $\mathcal{C}\subset C^\infty_0(G)$, \cite[Theorem~31.5]{SatoBook1999} implies that 
		\begin{align*}
			t^{-1}\left(\int_{\R^m} f(h,v-v') \mu^\psi_t(dv')-f(h,v)\right)=A^\psi_{\mathcal{V}}f(h,v).
		\end{align*}
		On the other hand, since $\mathcal{C}\subset\mathcal{D}(\Delta_{\mathcal{H}})$ by Lemma~\ref{lem:core_Q} we have
		\begin{align*}
			\lim_{t\searrow 0}\frac{\mathbb{Q}_tf-f}{t}= \Delta_{\mathcal{H}} f \text{ \ in } \ C_0(G).
		\end{align*}
		We also note that $\mu^{\psi}_t\to\delta_0$ weakly as $t\searrow 0$. As a result, given any $(h,v)\in G$, $w'\in\R^m$ and $\epsilon>0$, one can choose small enough $t$ such that
		\begin{align*}
			\left\|\frac{\mathbb{Q}_tf-f}{t}-\Delta_{\mathcal{H}} f\right\|_\infty<\frac{\epsilon}{2}  \text{ and } \left|\int_{\R^m} \Delta_{\mathcal{H}}f(h,v-v')\mu^{\psi}_t(dv')-\Delta_{\mathcal{H}}f(h,v)\right|<\frac{\epsilon}{2}, 
		\end{align*}
		which implies that for each $(h,v)\in G$,
		\begin{align*}
			\lim_{t\searrow 0}t^{-1}\int_{\R^m}\left(\mathbb{Q}_tf(h,v-v')-f(h,v-v')\right)\mu^{\psi}_t(dv')=\Delta_{\mathcal{H}}f(h,v).
		\end{align*}
		Using \cite[Lemma~31.7]{SatoBook1999}, we have $\mathcal{C}\subset \mathcal{D}(\Delta^{\!\psi}_{G})$. Now, for any $\beta_1, \beta_2\in\mathbb{N}_0$ with $\beta_1+\beta_2\le 2$, due to the smoothness of $f$  we have 
		\begin{align*}
			&\lim_{h,v\to\infty}|h|^{2} \frac{\partial}{\partial h^{\beta_1} \partial v^{\beta_2}}\mathbb{Q}^\psi_tf(h,v) \\
			&=\int_{\R^m}\lim_{h,v\to\infty}|h|^{2}\frac{\partial}{\partial h^{\beta_1} \partial v^{\beta_2}}\mathbb{Q}_tf(h,v-v') \mu^\psi_t(dv').
		\end{align*} 
		By Lemma~\ref{lem:core_Q} we have 
		\begin{align*}
			\lim_{h,v\to\infty}|h|^{2} \frac{\partial}{\partial h^{\beta_1} \partial v^{\beta_2}}\mathbb{Q}_tf(h,v-v')=0
		\end{align*}
		as well as 
		\begin{align*}
			\sup_{h,v,v'}|h|^{2}\left| \frac{\partial}{\partial h^{\beta_1} \partial v^{\beta_2}}\mathbb{Q}_tf(h,v-v')\right|<\infty.
		\end{align*}
		 Thus by the dominated convergence theorem, we conclude that $\mathbb{Q}^\psi_t(\mathcal{C})\subseteq \mathcal{C}$. Using the approximation approach as in the proof of Lemma~\ref{lem:core_Q}, one can show that $C^\infty_c(G)$ is also a core for $\Delta^{\!\psi}_{G}$. Finally, for any $g\in G$, let $\tau_g$ denote the left translation operator on $C_0(G)$. Since $\tau_g$ commutes with $\mathbb{Q}_t$, \eqref{eq:Q_psi} implies that for all $(h,v)\in G$
		
		\begin{align*}
			\tau_g \mathbb{Q}^\psi_tf(h,v)=\int_{\R^m} \tau_g\tau_{0, -v'}\mathbb{Q}_tf(h,v) \mu^{\psi_t}(dv').
		\end{align*}
		Since $\tau_g$ commutes with $\tau_{0,-v'}$ and $\mathbb{Q}_t$, the above equation implies that $\tau_g\mathbb{Q}^\psi_t=\mathbb{Q}^\psi_t \tau_g$. This proves that $\mathbb{Q}^\psi_t$ is translation-invariant and the proof  is complete. 
	\end{proof}
	\begin{proof}[Proof of Theorem~\ref{thm:characterization}] Since $\widetilde{\mathbb{Q}}$ is a left invariant Markov semigroup on $C_{0}(G)$, by Theorem~\ref{thm:Hunt}  the generator of $\widetilde{\mathbb{Q}}$, denoted by $\widetilde{\Delta}_G$, has the following representation.
		\begin{align*}
			\widetilde{\Delta}_Gf(g)&=\sum_{i,j=1}^{n+m} a_{ij} Z_i Z_jf(g)+\sum_{i=1}^{n+m} c_i Z_if(g)\\
			&+\int_{G} \left[f(g\star g')-f(g)-\mathbbm{1}_{\{g'\in U\}}\sum_{i=1}^{n+m} \operatorname{x}_i(g') Z_if(g)\right]\kappa(dg'), \nonumber
		\end{align*}
		where $U$ is a relatively compact open neighborhood of $0$ in $G$, 
		$\operatorname{x}_1,\ldots, \operatorname{x}_{n+m}\in C^\infty(G)$ are local coordinate functions with respect to $Z_1,\ldots, Z_{n+m}$ such that $\operatorname{x}_i(0)=0$, $Z_i(\operatorname{x}_j)(0)=\delta_{ij}$, $a_{ij}, c_i$ being constants such that $(a_{ij})_{1\le i,j\le n+m}$ is a symmetric nonnegative definite matrix, and $\kappa$ is a measure on $G$ satisfying 
		\begin{align*}
			\kappa(\{0\})=0, \quad \sum_{i=1}^{n+m}\int_{U}\operatorname{x}_i(g)^2\kappa(dg)<\infty,  \text{ and } \kappa(U^c)<\infty.
		\end{align*}
		Since $G$ is a homogeneous Carnot group, one can take $\operatorname{x}_i(g)=g_i$ for each $i=1,\ldots, n+m$. Recalling the identity $\widetilde{\Delta}_G\Pi=\Pi\Delta$, the above representation of $\widetilde{\Delta}_G$ forces that $a_{ij}=\delta_{ij}$ for $i,j\in\{1,\ldots, n\}$. Moreover, it also implies that $c_i=0$ for all $i=1,\ldots, n$. Let us now write 
		\begin{align}\label{eq:jump}
			\mathbb{J} f(g)=\int_{G} \left[f(g\star g')-f(g)-\mathbbm{1}_{\{g'\in U\}}\sum_{i=1}^{n+m} g'_i Z_if(g)\right]\kappa(dg').
		\end{align}
		Then one has $\mathbb{J}\Pi f=0$ for all $f\in C^\infty_c(\R^n)$. Now, writing $g=(h, v), h=(h', v')$, where $h,h'\in\R^n, v,v'\in\R^m$, we obtain
		\begin{align*}
			\mathbb{J}\Pi f(h,v)=\int_{\R^m}\int_{\R^n}\left[ f(h+h')-f(h)-\mathbbm{1}_{\{(h',v')\in U\}}\sum_{i=1}^{n} h'_i \frac{\partial}{\partial h_i} f(h)\right]\kappa(dh', dv')=0.
		\end{align*}
		We claim that the above equality implies that $\kappa$ must be supported on $\{0\}\times \R^m$. To see why, taking the Fourier transform with respect to $v$ one gets that for all $\zeta\in \R^n$,
		\begin{align*}
			&\int_{\R^n}\mathbb{J}\Pi f(h,v) e^{\i\langle h,\zeta\rangle} dh \\
			&=\mathcal{F}(f)(\zeta)\int_{G}\left[ e^{\i\langle h',\zeta\rangle}-1-\i\langle h',\zeta\rangle\mathbbm{1}_{\{(h',v')\in U\}}\right]\kappa(dh', dv') \\
			&=0.
		\end{align*}
		Since $f\in C^\infty_c(\R^n)$ is arbitrary, the above equality implies that 
		\begin{align}\label{eq:fourier=0}
			\int_{G}\left[ e^{\i\langle h',\zeta\rangle}-1-\i\langle h',\zeta\rangle\mathbbm{1}_{\{(h',v')\in U\}}\right]\kappa(dh', dv')=0 \text{ for all $\zeta\in\R^n$}.
		\end{align}
		Using the facts that 
		\begin{gather*}
			\lim_{|\zeta_i|\to\infty}\frac{e^{\i\langle h',\zeta\rangle}-1-\i\langle h',\zeta\rangle\mathbbm{1}_{\{(h',v')\in U\}}}{\zeta^2_i}=-\frac{1}{2}h'^2_i\mathbbm{1}_{\{(h',v')\in U\}}  \text{ and } \\
			\frac{e^{\i\langle h',\zeta\rangle}-1-\i\langle h',\zeta\rangle\mathbbm{1}_{\{(h',v')\in U\}}}{\zeta^2_i \frac{1}{2} h'^2_i \mathbbm{1}_{\{(h',v')\in U\}}}= \operatorname{O}_{\zeta_i}(1),
		\end{gather*}
		\eqref{eq:fourier=0} implies that $\int_{U} h'^2_i\kappa(dh', dv')=0$ for each $i=1,\ldots, n$. On the other hand, \eqref{eq:fourier=0} also implies that 
		\begin{align*}
			\int_{U^c}(1-\cos(\langle h', \zeta\rangle))\kappa(dh', dv')=0 \text{ for all } \zeta\in \R^n.
		\end{align*} 
		Since $(1-\cos(\langle h',\zeta\rangle))$ is strictly positive almost surely, we conclude that $\kappa$ is supported on $\{0\}\times\R^m$, that is, $\kappa=\delta_0\otimes\kappa_1$ where $\delta_0$ is the Dirac measure at $0\in\R^n$ and $\kappa_1$ is a measure on $\R^m$. As a result, from \eqref{eq:jump} it follows that 
		\begin{align*}
			\mathbb{J}f(h,v)=\int_{\R^m}\left[f(h,v+v')-f(h,v)-\langle v',\nabla_vf(h,v)\rangle\mathbbm{1}_{\{v'\in U_0\}}\right]\kappa_1(dv'),
		\end{align*}
		where $U_0=\{v\in\R^m: (0,v)\in U\}$. Moreover, 
		\begin{align*}
			\int_{U_0} |v|^2\kappa_1(dv)=\int_U\sum_{i=1}^{n+m} g^2_i \kappa(dg)<\infty, \quad \kappa_1(U^c_0)=\kappa(U^c)<\infty,
		\end{align*}
		that is, $\kappa_1$ is a L\'evy measure on $\R^m$. Therefore, $\widetilde{\Delta}_G=\Delta^\psi_G$ where $\psi=(\sigma,b,\kappa_1)$ with $\sigma=(a_{ij})_{n+1\le i,j\le n+m}$, $b=(c_{n+1},\ldots, c_{n+m})$. This completes the proof of the theorem. 
	\end{proof}
	
	\subsection{Spectrum of L\'evy perturbations of a sub-Laplacian} 
	Recall that by Remark~\ref{remark:2} $\mathbb{Q}^\psi$ extends to a strongly continuous Markov semigroup on $L^2(G)$ and with an abuse of notation, we denote its $L^2(G)$-generator also by $\Delta^{\!\psi}_{G}$. In this section, we analyze how the spectrum of $\Delta^{\!\psi}_{G}$ depends on $\psi$. In the case of Euclidean spaces, if $A^\psi$ denotes the generator of the L\'evy semigroup on $\R^m$ with the L\'evy-Khintchine exponent $\psi$, then by \eqref{eq:Levy-semigroup} it follows that $L^2(\R^m)$-spectrum of $\Delta^{\!\psi}_{G}$ is given by 
	\begin{align*}
		\sigma(A^\psi; L^2(\R^m))=\sigma_{\rm c}(A^\psi; L^2(\R^m))=\Range(\psi).
	\end{align*}
	Likewise, in the next theorem, we prove that the spectrum of $\Delta^{\!\psi}_{G}$ is also pure continuous and depends on the $\Range(\psi)$.  
	\begin{theorem}\label{thm:spec_dGp}
		For any $\psi\in\mathcal N$, 
		\begin{align}\label{eq:spec_delta_psi}
			S_\psi\subseteq \sigma(\Delta^{\!\psi}_{G}; L^2(G))=\sigma_{\rm c}(\Delta^{\!\psi}_{G}; L^2(G))\subseteq (-\infty,0]+\Range(\psi),
		\end{align}
		where 
		\begin{align}\label{eq:S_psi}
			S_\psi=\{\psi(\lambda)-\mathfrak{n}(\beta,\lambda,\nu): \beta\in\N^d_0, \lambda\in\R^m, \nu\in\R^k\},
		\end{align}
		and $\mathfrak{n}(\beta,\lambda,\nu)$ is defined in \eqref{eq:harmonic_osc}.
	\end{theorem}
	The inclusions $S_\psi\subseteq \sigma(\Delta^{\!\psi}_{G}; L^2(G))\subseteq (-\infty, 0]+\Range(\psi)$ follow from an argument closely related to the proof of Theorem~\ref{thm:subLaplacian_spectrum}. Proving the equality $\sigma(\Delta^{\!\psi}_{G}; L^2(G))=\sigma_{\rm c}(\Delta^{\!\psi}_{G}; L^2(G))$ requires more delicate arguments where the fact that L\'evy-Khintchine exponents are negative-definite plays a crucial role.
	
	\begin{corollary}\label{cor:spec_delta_psi}
		For any $\psi\in\mathcal N$, $\Im(\sigma(\Delta^{\!\psi}_{G}; L^2(G)))=\Range(\Im(\psi))$. If $\psi$ is real-valued then $\sigma(\Delta^{\!\psi}_{G}; L^2(G))=\sigma_{\rm c}(\Delta^{\!\psi}_{G}; L^2(G))=(-\infty,\psi(0)]$.
	\end{corollary}
	\begin{remark} In particular, when $\psi(\lambda)=-|\lambda|^2$, $\Delta^{\!\psi}_{G}$ is the full Laplacian on $G$, denoted by $\Delta_G$, and Corollary~\ref{cor:spec_delta_psi} implies that $\sigma(\Delta_G; L^2(G))=\sigma_{\rm c}(\Delta_G; L^2(G))=(-\infty, 0]$. This strengthens \cite[Theorem~1]{FurutaniSagamiOtsuki1993} in which the authors prove that the spectrum of the full Laplacian on any nilpotent Lie group of step two is $[0,\infty)$. They also proved that the spectrum is purely continuous in the case of the Heisenberg group, see \cite[Theorem~2]{FurutaniSagamiOtsuki1993}.
	\end{remark}
	\begin{proof}[Proof of Theorem~\ref{thm:spec_dGp}] 
		With an abuse of notation, if $Q^\psi$ denotes the semigroup generated by $A^\psi_{\mathcal{V}}$ on $G$, then we note that for all $t\geqslant 0$,
		\begin{align*}
			\mathbb{Q}^\psi_t=\mathbb{Q}_t Q^\psi_t=Q^\psi_t \mathbb{Q}_t \text{ on  } L^2(G).
		\end{align*}	
		As a result, $\sigma(\mathbb{Q}^\psi_t; L^2(G))\subseteq \sigma(\mathbb{Q}_t; L^2(G))\sigma(Q^\psi_t; L^2(\R^m))$ for all $t\geqslant 0$. Since $Q^\psi_t$ is a Fourier multiplier operator on $L^2(\R^m)$ with the multiplier $\lambda\mapsto e^{t\psi(\lambda)}$, it follows that
		\begin{align*}
			\sigma(Q^\psi_t; L^2(\R^m))=\sigma_{\rm c}(Q^\psi_t; L^2(\R^m))=\{e^{t\psi(\lambda)}: \lambda\in\R\}.
		\end{align*}	
		Therefore, $\sigma(\mathbb{Q}^\psi_t; L^2(G))\subseteq \{e^{-t\lambda_1+t\psi(\lambda_2)}: \lambda_1\geqslant 0, \lambda_2\in\R\}$. By spectral mapping theorem, we have $\sigma(\Delta^{\!\psi}_{G})\subseteq [0,\infty)+\Range(\psi)$. In order to show the other inclusion, using similar argument as in the proof of Proposition~\ref{thm:heat_intertwining}, one can show that $\mathbb{Q}^\psi_t(L^2_{\mathbb{U}}(G))\subseteq L^2_{\mathbb{U}}(G)$ for any $\psi\in\mathcal N$. Moreover, for any $\beta\in\N^d_0$ and $t\geqslant 0$, $\mathbb{Q}^\psi_t(\mathcal{L}_\beta)\subseteq\mathcal{L}_\beta$ and
		\begin{align*}
			\mathbb{Q}^\psi_t \operatorname{M}_\beta=\operatorname{M}_\beta \widetilde{Q}^\beta_t,
		\end{align*}
		where $\operatorname{M}_\beta$ is defined in \eqref{eq:M_beta} and $\widetilde{Q}^\beta_t$ is a Fourier multiplier operator defined by 
		\begin{align*}
			\mathcal{F}(\widetilde{Q}^\beta_t f)=e^{-t\mathfrak{n}(\beta,\lambda,\nu)+t\psi(\lambda)}\mathcal{F}(f)(\lambda), \quad f\in L^2(\R^m).
		\end{align*}
		As a consequence, for all $t\geqslant 0$,
		\begin{align*}
			\{e^{-t\mathfrak{n}(\beta,\lambda,\nu)+t\psi(\lambda)}: \lambda\in\R^m, \nu\in\R^k,\beta\in \N^d_0\}\subseteq \sigma(\mathbb{Q}^\psi_t),
		\end{align*}
		which is also equivalent to the inclusion $S_\psi\subseteq \sigma(\Delta^{\!\psi}_{G}; L^2(G))$. It remains to show that $\sigma(\Delta^{\!\psi}_{G}; L^2(G))$ is purely continuous. First, we prove that any $a\in\sigma(\Delta^{\!\psi}_{G}; L^2(G))$ with $\Re(a)=0$, one must have $a\in\sigma_{\rm c}(\Delta^{\!\psi}_{G}; L^2(G))$. It suffices to show that $a\notin\sigma_{\rm p}(\Delta^{\!\psi}_{G}; L^2(G))$ and $a\notin\sigma_{\rm r}(\Delta^{\!\psi}_{G}; L^2(G))$. Let $f\in L^2(G)$ be such that $\Delta^{\!\psi}_{G}f=af$ and $f\neq 0$. Then, for all $t\geqslant 0$, $\mathbb{Q}^\psi_tf= e^{at} f$. Taking the GFT of the last identity we get
		\begin{align*}
			\mathcal{F}_G(f)(\lambda,\nu)\circ H^{\lambda,\nu}_t=e^{-t(\psi(\lambda)-a)}\mathcal{F}_G(f)(\lambda,\nu) \text{ for all } \lambda\in\mathcal{Z}, \nu\in\R^k.
		\end{align*}
		Since $H^{\lambda,\nu}_t$ is a self-adjoint operator with eigenvalues $e^{-t\mathfrak{n}(\beta,\lambda,\nu)}, \beta\in\N^d_0$, the above identity forces that $\mathcal{F}_G(f)(\lambda,\nu)$ vanishes outside the following set 
		\begin{align*}
			\bigcup_{\beta\in\N^d_0}\left\{(\lambda,\nu)\in\R^{m+k}: -a+\psi(\lambda)= \sum_{j=1}^d (2\beta_j+1)\eta_j(\lambda)+|\nu|^2 \right\}
		\end{align*}
		However, for any $\beta\in\N^d_0$, we note that for any $(\lambda,\nu)\in\R^{m+k}\setminus\{0\}$, 
		\begin{align*}
			\sum_{j=1}^d (2\beta_j+1)\eta_j(\lambda)+|\nu|^2-\Re(\psi(\lambda))>0,
		\end{align*}
		which implies that $\mathcal{F}_G(f)(\lambda,\nu)=0$ for all $(\lambda,\nu)\in\mathcal Z^\ast\times \R^k$, that is, $f=0$, which contradicts our assumption. Therefore, $a\notin\sigma_{\rm p}(\Delta^{\!\psi}_{G}; L^2(G))$. Using a similar argument and the fact that $\sigma_{\rm r}(\Delta^{\!\psi}_{G}; L^2(G))\subseteq\sigma_{\rm p}(\widehat{\Delta}^{\!\psi}_{G}; L^2(G))$, we conclude that $a\notin\sigma_{\rm r}(\Delta^{\!\psi}_{G}; L^2(G))$. Hence, $a\in\sigma_{\rm c}(\Delta^{\!\psi}_{G}; L^2(G))$. Next, for the other values in the spectrum with strictly negative real part, we split the rest of the proof into two cases: $k\neq 0$ and $k= 0$ respectively.
		
		\textbf{Case 1: $k\neq 0$.} Note that this equivalent to $\mathrm{rad}_\lambda\neq\{0\}$ for any $\lambda\in\mathcal{Z}$. Let $a\in\sigma_{\rm p}(\Delta^{\!\psi}_{G}; L^2(G))$ for some $a\in\mathbb C$ with $\Re(a)\le 0$. Then, we must have $\mathbb{Q}^\psi_tf=e^{at} f$ for some nonzero $f\in L^2(G)$. Taking GFT of the last identity, a similar computation as before implies that $\mathcal{F}_G(f)(\lambda,\nu)=0$ whenever
		\begin{align}\label{eq:eigenvalue_eq}
			-a+\psi(\lambda)\neq\sum_{j=1}^d(2\beta_j+1)\eta_j(\lambda)+|\nu|^2 \text{ for any } \beta\in\N^d_0.
		\end{align}
		As in the proof of Theorem~\ref{thm:subLaplacian_spectrum}, for each $\beta\in\N^d_0$ and $s\geqslant 0$, let us denote 
		\begin{align*}
			\vartheta_\beta(s)=\operatorname{Leb}\left\{(\lambda,\nu)\in\R^{m+k}: \sum_{j=1}^d(2\beta_j+1)\eta_j(\lambda)+|\nu|^2-\Re(\psi(\lambda))\le s\right\}.
		\end{align*}
		Then, writing $\mathfrak{m}_\beta(\lambda)=\sum_{j=1}^d(2\beta_j+1)\eta_j(\lambda)-\Re(\psi(\lambda))$, we have
		\begin{align*}
			\vartheta_\beta(s)=\int_{\R^m} C_k(s-\mathfrak{m}_\beta(\lambda))^k_+ d\lambda,
		\end{align*}
		where $C_k$ denotes the volume of the unit ball in $\R^k$. This shows that $\vartheta_\beta(s)$ is continuous in $s$ and hence for any $a\in\C$ with $\Re(a)<0$,
		\begin{align*}
			\operatorname{Leb}\left\{(\lambda,\nu)\in\R^{m+k}:\sum_{j=1}^d(2\beta_j+1)\eta_j(\lambda)+|\nu|^2-\Re(\psi(\lambda))= -\Re(a) \right\}=0.
		\end{align*}
		So, $\mathcal{F}_G(f)(\lambda,\nu)=0$ on a set of full Lebesgue measure, which implies that $f=0$. Hence, $\sigma_{\rm p}(\Delta^{\!\psi}_{G}; L^2(G))=\emptyset$. By a similar argument, one can show that $\sigma_{\rm r}(\Delta^{\!\psi}_{G}; L^2(G))=\emptyset$, which proves that $\sigma(\Delta^{\!\psi}_{G}; L^2(G))=\sigma_{\rm c}(\Delta^{\!\psi}_{G}; L^2(G))$.
		
		\textbf{Case 2: $k=0$.} Let $a\in\sigma(\Delta^{\!\psi}_{G}; L^2(G))$ be such that $\Re(a)<0$ and $a\in\sigma_{\rm p}(\Delta^{\!\psi}_{G}; L^2(G))$. If $f\in L^2(G)$ is such that $\mathbb{Q}^\psi_tf=e^{at}f$, using the same argument as before, we get that $\mathcal{F}_G(f)(\lambda)$ vanishes on the complement of the following set 
		\begin{align}\label{eq:vanishing_set}
			\bigcup_{\beta\in\N^d_0}\left\{\lambda\in\R^m: \sum_{j=1}^d(2\beta_j+1)\eta_j(\lambda)-\psi(\lambda)= -a\right\}.
		\end{align}
		As before, for any $\beta\in\N^d_0$ and $s>0$ let denote 
		\begin{align*}
			\varsigma_\beta(s)=\operatorname{Leb}\left\{\lambda\in\R^m: \sum_{j=1}^d(2\beta_j+1)\eta_j(\lambda)-\Re(\psi(\lambda))=s\right\}.
		\end{align*}
		First, we claim that for any $\lambda_0\in S^{m-1}$, where $S^{m-1}$ denotes the $m$-dimensional sphere, the ray $\{r\lambda_0: r>0\}$ contains at most one solution of the equation 
		\begin{align*}
			\sum_{j=1}^d(2\beta_j+1)\eta_j(\lambda)-\Re(\psi(\lambda))=s.
		\end{align*}
		Indeed, any solution on the ray satisfies 
		\begin{align}\label{eq:eigen_equation}
			r\sum_{j=1}^d(2\beta_j+1)\eta_j(\lambda_0)-\Re(\psi(r\lambda_0))=\delta r-\Re(\psi(r\lambda_0))=s,
		\end{align}
		where $\delta:=\sum_{j=1}^d(2\beta_j+1)\eta_j(\lambda_0)>0$. Suppose that there at least two solutions $r_1, r_2$ to \eqref{eq:eigen_equation}. Since $x\mapsto -\Re(\psi(x\lambda_0))$ is a negative definite function, the above equation implies that for any $c_1, c_2\in\R$ with $c_1+c_2=0$, we should have 
		\begin{align*}
			\sum_{i=1}^2\sum_{j=1}^2 (s-\delta(r_i-r_j))c_ic_j\le 0,
		\end{align*}
		see \cite[Proposition~7.5]{BergForstBook1975}. However, choosing $c_1=1, c_2=-1$, we get $\sum_{i=1}^2\sum_{j=1}^2 (s-\delta(r_i-r_j))c_ic_j=2s>0$, which is a contradiction. This proves our claim. Now using Fubini's theorem with respect to the spherical coordinates, we conclude that $\varsigma_\beta(s)=0$ for any $s>0$ and $\beta\in\N^d_0$. As a result, the set defined in \eqref{eq:vanishing_set} has zero Lebesgue measure, which further implies that $\mathcal{F}_G(f)(\lambda)$ vanishes on a set of full Lebesgue measure. This contradicts the assumption that $a\in\sigma_{\rm p}(\Delta^{\!\psi}_{G}; L^2(G))$. A similar argument and the inclusion $\sigma_{\rm r}(\Delta^{\!\psi}_{G}; L^2(G))\subseteq\sigma_{\rm p}(\Delta^{\!\psi}_{G}; L^2(G))$ will imply that $a\notin\sigma_{\rm}(\Delta^{\!\psi}_{G}; L^2(G))$. Therefore, $\sigma(\Delta^{\!\psi}_{G}; L^2(G))=\sigma_{\rm c}(\Delta^{\!\psi}_{G}; L^2(G))$, which completes the proof of theorem.
	\end{proof}
	\begin{proof}[Proof of Corollary~\ref{cor:spec_delta_psi}]
		From the definition of $S_\psi$ in \eqref{eq:S_psi}, it is evident that $\Im(S_\psi)=\Range(\Im(\psi))$. Thus, the first statement follows directly from \eqref{eq:spec_delta_psi}. When $\psi$ is real valued, $\psi(\lambda)\le 0$ for all $\lambda\in\R^m$. On the other hand, since $\eta_j$ for any $j$ is a nonnegative homogeneous continuous function on $\R^m$, we conclude that for every $\beta\in\N^d_0$,
		\begin{align*}
			\{\psi(\lambda)-\mathfrak{n}(\beta,\lambda,\nu): \lambda\in\R^m, \nu\in\R^k\}=(-\infty,\psi(0)),
		\end{align*}
		which implies that $S_\psi=(-\infty,\psi(0))$. Finally, the second statement of the theorem follows after noting that $(-\infty,0]+\Range(\psi)=(-\infty, \psi(0)]$ whenever $\psi$ is real-valued. 
	\end{proof}

	\section{L\'evy-Ornstein-Uhlenbeck semigroup  and intertwinings} \label{sec:7}
	
	\subsection{Ornstein-Uhlenbeck semigroup and its perturbations} For the Euclidean space $\R^n$, the classical Ornstein-Uhlenbeck (OU) operator is defined as the diffusion operator 
	\begin{align*}
		L=\Delta-\langle x, \nabla\rangle.
	\end{align*}
	Then $L$ generates an ergodic Markov semigroup known as the OU semigroup on $\R^n$, and the unique invariant distribution is standard Gaussian. We can alternatively view $L$ as a perturbation of the Laplacian $\Delta$ by $-\langle x,\nabla\rangle$, which is also the generator of the one parameter dilation semigroup $(d_{e^{-t}})_{t\geqslant 0}$ on $C_0(\R^n)$. Similarly  the Ornstein-Uhlenbeck (OU) operator on $G$ is defined as the second order differential operator 
	\begin{align*}
		\mathbb{L}= \Delta_{\mathcal{H}}+\mathbb{D},
	\end{align*}
	where $\Delta_{\mathcal{H}}$ is the sub-Laplacian on $G$, and $\mathbb{D}$ denotes the generator of the dilation semigroup $(\delta_{e^{-t}})_{t\in\R}$ on $C_0(G)$.
	In Lemma~\ref{cor:gen_OU_core} below it is shown that $(\mathbb{L},C^\infty_c(G))$ generates a Feller Markov semigroup $(\mathbb{P}_t)_{t\geqslant 0}$ on $C_0(G)$, which we call the \emph{Ornstein-Uhlenbeck semigroup} on $G$. Using self-similarity of the horizontal heat semigroup $(\mathbb{Q}_t)_{t\geqslant 0}$ generated by $\Delta_{\mathcal{H}}$, one can show that 
	\begin{align}\label{eq:OU_semigroup}
		\mathbb{P}_t = \delta_{e^{-t}}\mathbb{Q}_{\frac{1-e^{-2t}}{2}}.
	\end{align}
	This identity also implies that $(\mathbb{P}_t)_{t\geqslant 0}$ is ergodic with respect to the unique invariant distribution $\mathsf{q}_{1/2}$, where $\mathsf{q}_t$ denotes the heat kernel of the horizontal heat semigroup at time $t$. We refer to \cite{Lust-Piquard2010} for a detailed spectral analysis of OU semigroups on any stratified group, and to \cite{BaudoinHairerTeichmann2008} for self-adjoint OU semigroups on Lie groups. By \cite[Theorem~3]{Lust-Piquard2010}, it is known that $\mathbb{P}$ is non-self-adjoint on $L^2(G,\mathsf{q}_{1/2})$. 
	
	The goal of this section is to study a class of operators obtained by some perturbations of $\mathbb{L}$ along the directions of vertical vector fields, where the perturbation operators are the same as in Section~\ref{sec:6}. We begin with a description of L\'evy-type perturbations of OU operators on $\R^m$. For each $\psi\in\mathcal{N}$, where $\mathcal{N}$ is defined in \eqref{notation:N}, the L\'evy-OU operator $L_\psi$ on $\R^m$ is defined by 
	\begin{align*}
		L_\psi=A^\psi-2\sum_{j=1}^m v_j\frac{\partial}{\partial v_j},
	\end{align*}
	where $A^\psi$ is the generator of L\'evy semigroup defined in \eqref{eq:Levy}.
	By \cite{SatoBook1999} it follows that $L_\psi$ is the generator of a strongly continuous Markov semigroup on $C_0(\R^m)$, which is known as the L\'evy-OU semigroup on $C_0(\R^m)$. Let us denote 
	\begin{align}\label{eq:OU_levy}
		P^\psi_t=\text{ semigroup generated by $L_\psi$ on $C_0(\R^m)$}.
	\end{align}
	We note that $L_\psi$ is a generalization of the diffusion OU operator on $\R^m$. Moreover, the Markov process associated to the semigroup $P^\psi$ satisfies the stochastic differential equation
	\begin{align}\label{eq:P_psi_sde}
		dX(t)=-2X(t) dt+ dY^\psi(t),
	\end{align}
	where $Y^\psi(t)$ is the L\'evy process with generator $A^\psi$.
	
By \cite[Theorem~17.5]{SatoBook1999}, if the L\'evy measure of $\psi$ has finite logarithmic moment away from $1$, that is, $\psi=(\sigma,b,\kappa)$ with 
	\begin{align*}
		\int_{|v|>1}\log|v| \kappa(dv)<\infty,
	\end{align*} 
	$P^\psi$ is ergodic with the unique invariant distribution $\mu_\psi$ such that 
	\begin{align}\label{eq:p_psi}
		\int_{\R^m} e^{\i\lambda v}\mu_\psi(dv)= \exp\left({\int_0^\infty \psi(e^{-2s}\lambda)ds}\right).
	\end{align}
	Let us denote
	\begin{align}\label{eq:N_log}
		\mathcal{N}_{\log}:=\left\{\psi=(\sigma, b, \kappa)\in\mathcal N: \int_{|v|>1} \log|v|\kappa(dv)<\infty \right\}.
	\end{align}
	For Carnot groups, we consider L\'evy-type perturbations of $\mathbb{L}$ defined by 
	\begin{align*}
		\mathbb{L}_{\psi}=\Delta^{\!\psi}_{G} +\mathbb{D}=\mathbb{L}+A^\psi_{\mathcal{V}},
	\end{align*}
	where $A^\psi_{\mathcal{V}}$ denotes $A^\psi$ in the vertical direction of $G$. In Theorem~\ref{prop:P_gamma}, we prove that $\mathbb{L}_\psi$ generates a Markov semigroup on $G$. Moreover, when $\psi\in\mathcal{N}_{\log}$, $\mathbb{L}_\psi$ generates an ergodic Markov semigroup. Before stating the main result, let us introduce a key notation which will be required in the statement and throughout the rest of the paper. For any $t\in\R$, we define 
	\begin{equation}\label{eq:psi_t}
		\begin{aligned}
			\psi_t(\lambda):=\int_0^t \psi(e^{2s}\lambda) ds
			=&-\frac{e^{4t}-1}{4}\langle\sigma \lambda,\lambda\rangle+\i\frac{e^{2t}-1}{2}\langle b,\lambda\rangle \\
			&+\int_0^t \int_{\R^m}(\exp(\i e^{2s}\langle\lambda,v\rangle)-1-\i\langle\lambda,e^{2s}v\rangle)\kappa(dv).
		\end{aligned}
	\end{equation}
	Then for all $t\in [0,\infty)$, $\psi_t, -\psi_{-t}\in\mathcal N$. We are now ready to state the main result of this section.
	\begin{theorem}\label{prop:P_gamma}
		For any $\psi\in\mathcal{N}$, we have the following.
		\begin{enumerate}[leftmargin=*]
			\item  \label{it:3} The closure of $(\mathbb{L}_{\psi}, C^\infty_c(G))$ 
			generates a Feller Markov semigroup $(\mathbb{P}^{\psi}_t)_{t\geqslant 0}$ on $C_0(G)$ such that
			\begin{align}\label{eq:P_alpha}
				\mathbb{P}^{\psi}_tf(h,v)=\int_{\R^m}\mathbb{P}_tf(h,v-v')\mu^{\psi_t}(dv'),
			\end{align}
			where $\psi_t$ is defined in \eqref{eq:psi_t}, and $\mu^{\psi_t}$ is such that 
			\begin{align}\label{eq:q_psi_t}
				\int_{\R^m} e^{\i\langle\lambda, v\rangle} \mu^{\psi_t}(dv)=e^{\psi_t(\lambda)} \text{  for all } \lambda\in\R^m.
			\end{align}
			\item \label{it:4}  $\mathbb{P}^\psi$ has a unique invariant density $\mathsf{p}_\psi\in C^\infty_0(G)$ with respect to the bi-invariant Haar measure on $G$ whenever $\psi\in\mathcal{N}_{\log}$, where $\mathcal{N}_{\log}$ is defined in \eqref{eq:N_log}. Moreover, using Notation~\ref{notation:coordinates} for the coordinates of $G$,
			\begin{equation}\label{eq:p_gamma}
				\begin{aligned}
					& (2\pi)^d\int_{\R^m}\int_{\R^k} \mathsf{p}_\psi(z,r,v) e^{\i\langle \nu,r\rangle} e^{\i\langle \lambda, v\rangle} drdv \\
					&=\exp\left(-\frac{|\nu|^2}{2}+\int_0^\infty \psi(-e^{-2s}\lambda) ds\right) \\
					&\ \ \ \ \ \times \prod_{j=1}^d\frac{\eta_j(\lambda)}{2\sinh(\eta_j(\lambda)/2)}\exp\left(-\frac{\eta_j(\lambda)|z_j|^2}{2}\coth(\eta_j(\lambda)/2)\right).
				\end{aligned}
			\end{equation}
			In fact, $\mathsf{p}_\psi=\mathsf{q}_{\frac12}{\ast}_v \overline{\mu}_\psi$, where $\mu_\psi$ is defined in \eqref{eq:p_psi} and $\overline{\mu}_\psi(B)=\mu_\psi(-B)$ for all $B\in\mathcal{B}(\R^m)$, and $\mathsf{q}_{\frac12}$ is the horizontal heat kernel at $t=1/2$.
			As a result, $\mathbb{P}^\psi$ can be extended to a strongly continuous contraction Markov semigroup on $L^{p}(G, \mathsf{p}_\psi)$ for $p\geqslant 1$.
		\end{enumerate}
	\end{theorem}
	\begin{remark}
		In addition to the above theorem, using a standard approximation method with mollifiers one can prove that $C^\infty_c(G)$ is a core for $L^p(G,\mathsf{p}_\psi)$-generator of $\mathbb{P}^\psi$ for any $p\geqslant 1$.
	\end{remark}
	
	In the next proposition, we derive an explicit expression of the GFT of $\mathbb{P}^\psi$. This will play an important role in the proof of the above proposition and also the subsequent results.
	\begin{proposition}\label{prop:fourier_transform}
		Let $\mathbb{P}^{\psi}_t$ be defined as in \eqref{eq:P_alpha}. Then, for any $f\in\mathcal{S}(G)$ and $(\lambda,\nu)\in\mathcal{Z}\times\R^k$, 
		\begin{align}\label{eq:P}
			\mathcal{F}_G(\mathbb{P}^\psi_tf)(\lambda,\nu)=e^{(n+2m)t} e^{\psi_t(\lambda)}d_{e^{-t}}\mathcal{F}_G(f)(e^{2t}\lambda ,e^t\nu )d_{e^t} H^{\lambda,\nu}_{\frac{e^{2t}-1}{2}}.
		\end{align}
		If $\widehat{A}$ denotes the $L^2(G)$-adjoint of $A$, then 
		\begin{align}\label{eq:P_hat}
			\mathcal{F}_G(\widehat{\mathbb{P}}^{\psi}_tf)(\lambda,\nu)=e^{\psi_{t}(-e^{-2t}\lambda)}d_{e^t} \mathcal{F}_G({f})(e^{-2t}\lambda , e^{-t}\nu ) d_{e^{-t}} H^{\lambda,\nu}_{\frac{1-e^{-2t}}{2}}.
		\end{align}
	\end{proposition}
	\begin{proof}
		We note that \eqref{eq:OU_semigroup} and Lemma~\ref{lem:dilation} imply that for any $t>0$ and $f\in\mathcal{S}(G)$, 
		\begin{align*}
			\mathcal{F}_G(\mathbb{P}_tf)(\lambda,\nu)=e^{(n+2m)t} d_{e^{-t}}\mathcal{F}_G(f)(e^{2t}\lambda , e^{t}\nu )d_{e^t} H^{\lambda,\nu}_{\frac{e^{2t}-1}{2}}.
		\end{align*}
		If $\mathbb{P}^{\psi}_t$ is defined according to \eqref{eq:P_alpha}, we note that $\mathbb{P}^{\psi}_tf\in L^1(G)$ for any $f\in\mathcal{S}(G)$. Therefore, using the representation $\mathbb{P}^{\psi}_t=\mathbb{P}_tf{\ast}_v\mu^{\psi_t}$, we conclude the proof of \eqref{eq:P}. Next, we observe that 
		\begin{align*}
			{\widehat{\mathbb{P}}_t=e^{(n+2m)t}\mathbb{Q}_{\frac{1-e^{-2t}}{2}}\delta_{e^{t}}.}
		\end{align*}
		Now, for any $\widetilde{f}\in\mathcal{S}(G)$, we note that 
		\begin{align*}
			\int_{G}{\mathbb{P}}^\psi_tf(h,v)\widetilde{f}(h,v) dh dv=\int_{G}\mathbb{P}_tf(h,v)\left(\int_{\R} \widetilde{f}(h,v+v')\mu^{\psi_t}(dv')\right) dh dv,
		\end{align*}
		which shows that $\widehat{\mathbb{P}}^\psi_t \widetilde{f}=\widehat{\mathbb{P}}_t (\widetilde{f}{\ast}_v \overline{\mu}^{\psi_t})$, where $\overline{\mu}^{\psi_t}(B)=\overline{\mu}^{\psi_t}(-B)$ for all $B\in\mathcal{B}(\R^m)$. Consequently, applying Lemma~\ref{lem:vertical_conv}, for all $f\in\mathcal{S}(G)$ we have 
		\begin{align*}
			\mathcal{F}_G(\widehat{\mathbb{P}}^\psi_tf)(\lambda,\nu)&=\mathcal{F}_G(\widehat{\mathbb{P}}_t ({f}{\ast}_v \overline{\mu}^{\psi_t}))(\lambda,\nu) \\
			&=\mathcal{F}_G({f}{\ast}_v\overline{\mu}^{\psi_t})(e^{-2t}\lambda ,e^{-t}\nu ) d_{e^{-t}} H^{\lambda,\nu}_{\frac{1-e^{-2t}}{2}} \\
			&=\mathcal{F}_G({f})(\lambda e^{-2t},\nu e^{-t})\mathcal{F}(\mu^{\psi_t})(-e^{-2t}\lambda)d_{e^{-t}} H^{\lambda,\nu}_{\frac{1-e^{-2t}}{2}}  \\
			&=e^{\psi_{t}(-e^{-2t}\lambda)} d_{e^t} \mathcal{F}_G(f)(e^{-2t}\lambda ,e^{-t}\nu ) d_{e^{-t}} H^{\lambda,\nu}_{\frac{1-e^{-2t}}{2}}.
		\end{align*}
		This completes the proof of the proposition.
	\end{proof}
	In the next lemma, we provide  core for the generator $\mathbb{L}$ that is also invariant under the action of the semigroup $\mathbb{P}_t$.
	
	\begin{lemma}\label{cor:gen_OU_core}
		Let $\mathcal{C}'\subset C_0(G)$ be defined by
		\begin{align}
			\mathcal{C}'=\left\{f\in C^\infty_0(G): \lim_{h,v\to\infty}|h|^{2} \frac{\partial f}{\partial h^{\beta_1} \partial v^{\beta_2}}=0, \ \lim_{h,v\to\infty}|v|\frac{\partial f}{\partial v}=0, \ \beta_1+\beta_2\le 2 \right\},
		\end{align}
		Then, $\mathcal{C}'\subset\mathcal{D}(\mathbb{L})$ and $\mathbb{P}_t (\mathcal{C}')\subseteq \mathcal{C}'$ for all $t\geqslant 0$. Moreover, $C^\infty_c(G)$ is a core for the generator $\mathbb{L}$.
		\begin{proof}
			From the definition of $\mathcal{C}'$, it is evident that $\mathbb{L}(\mathcal C')\subset C_0(G)$. Let $f\in\mathcal C'$. Then, for any fixed $g\in G$, 
			\begin{align*}
				\frac{\mathbb{P}_tf(g) -f(g)}{t}&=\frac{\mathbb{Q}_{\frac{1-e^{-2t}}{2}}f(\delta_{e^{-t}}g)-f(g)}{t}\\
				&=\frac{\mathbb{Q}_{\frac{1-e^{-2t}}{2}}f(\delta_{e^{-t}}g)-f(\delta_{e^{-t}}g)+f(\delta_{e^{-t}}g)-f(g)}{t}.
			\end{align*}
			Since $(\mathbb{Q}_{\frac{1-e^{-2t}}{2}}f-f)/t\to \Delta_{\mathcal{H}} f$ uniformly and $(f(\delta_{e^{-t}}g)-f(g))/t\to-\mathbb{D}f(g)$ point-wise as $t\searrow 0$, we have 
			\begin{align*}
				\lim_{t\searrow 0}\frac{\mathbb{P}_tf -f}{t}= \Delta_{\mathcal{H}} f-\mathbb{D}f=\mathbb{L}f
			\end{align*}
			point-wise. This proves that $f\in\mathcal{D}(\mathbb{L})$.
			Also, using a very similar idea as in the proof of Lemma~\ref{lem:core_Q}, one can show that $\mathbb{Q}_t(\mathcal C')\subseteq \mathcal C'$ for all $t\geqslant 0$. Therefore, for all $t\geqslant 0$,
			\begin{align*}
				\mathbb{P}_t(\mathcal C')=\delta_{e^{-t}}\mathbb{Q}_{\frac{1-e^{-2t}}{2}}(\mathcal C')\subseteq \delta_{e^{-t}}(\mathcal C')\subseteq \mathcal C'.
			\end{align*}
			Using the same approximation technique as in the proof of Lemma~\ref{lem:core_Q}, it follows that $C^\infty_c(G)$ is a core for $\mathbb{L}$.
		\end{proof}
	\end{lemma}
	\begin{proof}[Proof of Theorem~\ref{prop:P_gamma}]
		For \eqref{it:3}, we first need to show that $\mathbb{P}^\psi$ is a semigroup on $C_0(G)$. By \eqref{eq:P} we have that for any $s, t>0$, $(\lambda,\nu)\in\mathcal{Z}\times \R^k$, and $f\in\mathcal{S}(G)$  
		\begin{align*}
			&\mathcal{F}_G(\mathbb{P}^{\psi}_t\mathbb{P}^{\psi}_sf)(\lambda,\nu) 
			\\
			&=e^{-(n+2m)t} e^{\psi_t(\lambda)} d_{e^{-t}}\mathcal{F}_G(\mathbb{P}^{\psi}_sf)(e^{2t}\lambda ,e^t\nu )d_{e^t} H^{\lambda,\nu}_{\frac{e^{2t}-1}{2}} 
			\\
			&=e^{-(n+2m)(s+t)}e^{\psi_t(\lambda)+\psi_s(e^{2t}\lambda)}  d_{e^{-(t+s)}}\mathcal{F}_G(f)(e^{2(t+s)}\lambda,e^{t+s}\nu )d_{e^{s}} H^{e^{2t}\lambda, e^t\nu }_{\frac{e^{2s}-1}{2}} d_{e^{t}} H^{\lambda,\nu}_{\frac{e^{2t}-1}{2}} 
			\\
			&=e^{-(n+2m)(s+t)}e^{\psi_{t+s}(\lambda)} d_{e^{-(t+s)}}\mathcal{F}_G(f)(e^{2(t+s)}\lambda, e^{t+s}\nu)d_{e^{(s+t)}} H^{\lambda,\nu}_{\frac{e^{2(s+t)}-e^{2t}}{2}} H^{\lambda,\nu}_{\frac{e^{2t}-1}{2}} 
			\\
			&=\mathcal{F}_G(\mathbb{P}^{\psi}_{t+s} f)(\lambda,\nu),
		\end{align*}
		where the second to the last equality follows from Lemma~\ref{lem:scaling_Harmonic} and the observation $\psi_{t+s}(\lambda)=\psi_t(\lambda)+\psi_s(e^{2t}\lambda)$. By injectivity of the GFT we conclude that $\mathbb{P}^{\psi}_t \mathbb{P}^{\psi}_s=\mathbb{P}^{\psi}_{t+s}$ on $\mathcal{S}(G)$. The previous identity extends to $C_0(G)$ by using the density of $\mathcal{S}(G)$ in $C_0(G)$. Next, to find the generator of $\mathbb{P}^{\psi}$, we proceed similarly to the proof of Theorem~\ref{thm:heat_perurbation}. For any $f\in \mathcal{C}'$, where $\mathcal C'$ has been introduced in Lemma~\ref{cor:gen_OU_core}, we note that $\mathbb{L}_\psi f\in C_0(G)$. Also, for any $t\geqslant 0$ we have 
		\begin{equation}\label{eq:limit_vertical}
			\begin{aligned}
				\mathbb{P}^{\psi}_tf(h,v)-f(h,v)&=\int_{\R^m}\left(\mathbb{P}_tf(h,v-v')-f(h,v-v')\right)\mu^{\psi_t}(dv') \\
				&+\int_{\R^m} f(h, v-v')\mu^{\psi_t}(dv')-f(h,v).
			\end{aligned}
		\end{equation}
		We first claim that
		\begin{align}\label{eq:limit1}
			\lim_{t\to 0} t^{-1}\left(\int_{\R^m} f(h, v-v')\mu^{\psi_t}(dv')-f(h,v)\right)=A^\psi_{\mathcal{V}} f(h, v)
		\end{align}
		for each $(h,v)$. Since the function $v\mapsto f(h,v)\in C^\infty_0(\R^m)$,
		\eqref{eq:limit_vertical} follows exactly in the same manner as in the proof of \cite[Theorem~31.5]{SatoBook1999} after observing that 
		\begin{align*}
			\exp\left(t^{-1} (e^{\psi_t(\lambda)}-1)\right)=\psi(\lambda) \quad \text{for all} \ \lambda\in\R^m.
		\end{align*}
		On the other hand, since $f\in \mathcal{C}'\subset \mathcal{D}(\mathbb{L})$, one has 
		\begin{align*}
			\lim_{t\searrow 0}\frac{\mathbb{P}_tf-f}{t}= \mathbb{L} f \  \text{ in $C_0(G)$}.
		\end{align*}
		As $\lim_{t\searrow 0}\psi_t(\lambda)=1$, using the weak convergence of $\mu^{\psi_t}$ to $\delta_0$ as $t\searrow 0$, we conclude that
		\begin{align*}
			&t^{-1}\int_{\R^m}\left(\mathbb{P}_tf(h,v-v')-f(h,v-v')\right)\mu^{\psi_t}(dv') \\
			&= \int_{\R^m}\mathbb{L}f(h, v-v')\delta_0(dv') \\
			&=\mathbb{L}f(h,v).
		\end{align*}
		Consequently, using \cite[Lemma~31.7]{SatoBook1999}, we have $\mathcal{C}'\subset \mathcal{D}(\mathbb{L}_{\psi})$. Let us first assume that the L\'evy measure $\kappa$ has finite mean. Then by \cite[Theorem~25.3]{SatoBook1999} it follows that $\mu^{\psi_t}$ has finite mean for each $t\geqslant 0$. Therefore, for any $t\geqslant 0$ and $f\in\mathcal{C}'$, 
		\begin{align*}
			&|v|\left|\frac{\partial}{\partial v}\int_{\R^m}\mathbb{P}_tf(h,v-v')\mu^{\psi_t}(dv')\right| \\
			&\le\int_{\R^m}|v-v'|\left|\frac{\partial}{\partial v}\mathbb{P}_tf(h,v-v')\right|\mu^{\psi_t}(dv')+\int_{\R^m}\left|\frac{\partial}{\partial v}\mathbb{P}_tf(h,v-v')\right||v'|\mu^{\psi_t}(dv').
		\end{align*}
		Since 
		\begin{align*}
		\lim_{|v|\to\infty}|v-v'|\frac{\partial}{\partial v}\mathbb{P}_tf(h,v-v')&= 0, \\
		\lim_{|v|\to\infty}\frac{\partial}{\partial v}\mathbb{P}_tf(h,v-v')&= 0
		\end{align*}
		uniformly, see Corollary~\ref{cor:gen_OU_core}, using dominated convergence theorem we conclude that $\mathbb{P}^\psi_t(\mathcal{C}')\subseteq \mathcal{C}'$ for all $t\geqslant 0$.
		Thus, applying \cite[Lemma~31.6]{SatoBook1999}, it follows that $\mathcal{C}'$ is a core for the generator $\mathbb{L}_\psi$ when $\kappa$ has finite mean. Finally, using the approximation technique in the proof of Lemma~\ref{lem:core_Q}, it follows that $C^\infty_c(G)$ is a core for $\mathbb{L}_\psi$ when $\kappa$ has finite mean. This is also equivalent to the density of $(\alpha I-\mathbb{L}_\psi)(C^\infty_c(G))$ in $C_0(G)$ for some $\alpha>0$, see e.g. \cite[Exercise~1.15]{EngelNagelBook2006}. Now, when $\kappa$ does not have a finite mean, let us denote $\kappa_n=\kappa(\cdot\cap B_n)$, where $B_n$ is the ball of radius $n$ in $\R^m$. Also, let $\psi_n\in\mathcal N$ be associated to the triplet $(\sigma,b,\kappa_n)$. In that case, for all $f\in C^\infty_c(G)$ we have 
		\begin{align*}
			|\mathbb{L}_{\psi_n}f(h,v)-\mathbb{L}_\psi f(h,v)|\le \int_{v'\in B^c_n} |f(h,v+v')-f(h,v)|\kappa(dv)\le \|f\|_\infty\kappa(B^c_n),
		\end{align*}
		which shows that $\mathbb{L}_{\psi_n}f \to \mathbb{L}_\psi f$ uniformly as $n\to\infty$. Since $(\alpha I-\mathbb{L}_{\psi_n})(C^\infty_c(G))$ is dense in $C_0(G)$ for all $n\geqslant 1$, the above approximation shows that $(\alpha I-\mathbb{L}_\psi)(C^\infty_c(G))$ is dense in $C_0(G)$ as well. Thus, $C^\infty_c(G)$ is a core for $\mathbb{L}_\psi$ for any $\psi\in\mathcal N$. This completes the proof of \eqref{it:3}. 
		
		For the existence of the limiting distribution, we note that for any $f\in C_0(G)$ and $(h,v)\in G$, we can use \eqref{eq:P_alpha} combined with \eqref{eq:OU_semigroup}, and  a simple change of variable yields 
		\begin{align}
			\mathbb{P}^{\psi}_tf(h,v)&=\int_{\R^m}\delta_{e^{-t}}\mathbb{Q}_{\frac{1-e^{-2t}}{2}}f(h,v-v')\mu^{\psi_t}(dv') \nonumber \\
			&=\int_{\R^m}\left(\delta_{e^{-t}}\mathbb{Q}_{\frac{1-e^{-2t}}{2}}f(h,v-v')-\delta_{e^{-t}}\mathbb{Q}_{\frac12}f(h,v-v')\right)\mu^{\psi_t}(dv') \label{eq:eq1} \\
			&+\int_{\R^m}\mathbb{Q}_{\frac12}f(e^{-t}h,e^{-2t}v-v') \mu^{\psi_t}\circ d_{e^{2t}}(dv'). \nonumber
		\end{align}
		The strong continuity of the semigroup $\mathbb{Q}_t$ implies that $\mathbb{Q}_{(1-e^{-2t})/2}f\to\mathbb{Q}_{1/2}f$ in the uniform topology of $C_0(G)$ as $t\to\infty$. As a result, the contraction property of the dilation group $\delta_{e^{-t}}$ implies that for each $(h,v)\in G$, the first term in \eqref{eq:eq1} converges to $0$ as $t\to\infty$.
		To conclude the proof, we need the following lemma.
		\begin{lemma}\label{lem:weak_conv}
			For any $\psi\in\mathcal{N}_{\log}$, $\mu^{\psi_t}\circ d_{e^{2t}}\to \mu_\psi$ weakly as $t\to\infty$, where $\mu^{\psi_t}$ and $\mu_\psi$ are defined by \eqref{eq:q_psi_t} and \eqref{eq:p_psi} respectively.
		\end{lemma}
		\begin{proof}
			For any $\lambda\in\R^m$ and $t>0$, we note that 
			\begin{align*}
				\int_{\R^m} e^{\i\langle\lambda,v\rangle}\mu^{\psi_t}\circ d_{e^{2t}}(dv)=e^{\psi_t(e^{-2t}\lambda)}.
			\end{align*}
			Now, for all $t\ge 0$, we have
			\begin{align*}
				\psi_t(e^{-2t}\lambda)=\int_0^t \psi(e^{-2(t-s)}\lambda)ds=\int_0^t \psi(e^{-2s}\lambda)ds,
			\end{align*}	
			which shows that $e^{\psi_t(e^{-2t}\lambda)}\overset{t\to\infty}{\longrightarrow} e^{-\psi_{-\infty}(\lambda)}$ point-wise as $\psi\in\mathcal{N}_{\log}$, see \cite[Theorem~17.5]{SatoBook1999} for details. This proves the lemma.
		\end{proof}
		Note that for any fixed $(h,v)\in G$, 
		\begin{align*}
			\lim_{t\to\infty}\mathbb{Q}_{\frac12}f(e^{-t}h, e^{-2t}v-v')=\mathbb{Q}_{\frac12}f(0,-v')
		\end{align*}	
		uniformly with respect to $v'$. This fact follows from the uniform continuity of the function $\mathbb{Q}_{\frac12}f$. Thus, Lemma~\ref{lem:weak_conv} yields
		\begin{align*}
			\lim_{t\to\infty}&\int_{\R^m}\mathbb{Q}_{\frac12}f(e^{-t}h,e^{-2t}v-v') \mu^{\psi_t}\circ d_{e^{2t}}(dv') \\
			&=\int_{\R^m}\mathbb{Q}_{\frac{1}{2}}f(0,-v')\mu_\psi(dv')\\
			&=\int_{G}f(h,v)\left(\int_{\R^m}\mathsf{q}_{\frac12}(h,-v'-v)\mu_\psi(dv')\right)dhdv\\
			&=\int_{G}f(h,v)\left(\int_{\R^m}\mathsf{q}_{\frac12}(h,v+v')\mu_\psi(dv')\right)dhdv,
		\end{align*}
		where the last line follows as the function $v\mapsto\mathsf{q}_{\frac12}(h,v)$ is even.
		This shows that for any $g\in G$, the kernel $\mathsf{p}^{\psi}_t$ of $\mathbb{P}^{\psi}_t$ satisfies
		\begin{align*}
			\lim_{t\to\infty} \mathsf{p}^{\psi}_t(g,\cdot)=\mathsf{q}_{\frac12}{\ast}_v \overline{\mu}_\psi   \text{ weakly},
		\end{align*}
		where $\overline{\mu}_\psi(B)=\mu_\psi(-B)$ for all $B\in\mathcal{B}(\R^m)$.
		Therefore, $\mathsf{p}_\psi=\mathsf{q}_{\frac12}{\ast}_v \overline{\mu}_\psi$.
		Finally, the proof of \eqref{it:4} is concluded after observing that 
		\begin{align*}
			&(\mathsf{q}_{\frac12}{\ast}_v \overline{\mu}_\psi)^{\lambda,\nu}(z)\\
			&=e^{-\frac{|\nu|^2}{2}}e^{-\psi_{-\infty}(-\lambda)}\prod_{j=1}^d\frac{\eta_j(\lambda)}{2\sinh(\eta_j(\lambda)/2)}\exp\left(-\sum_{j=1}^d\frac{\eta_j(\lambda)|z_j|^2}{2}\coth(\eta_j(\lambda)/2)\right),
		\end{align*} 
		where for any $f\in L^1(G)$, $f^{\lambda,\nu}$ is defined by \eqref{eq:fourier_radical}.
		Since $\mathsf{q}_{\frac12}$ is smooth, the limiting distribution $\mathsf{p}_\psi$ is absolutely continuous with respect to the Haar measure on $G$, and the density, still denoted by $\mathsf{p}_\psi$, is also smooth.
	\end{proof}
	
	Since we will be considering the Markov semigroup $\mathbb{P}^\psi$ on the weighted space $L^p(G,\mathsf{p}_\psi)$, $p\in [1,\infty]$, we restrict ourselves to the subclass $\mathcal{N}_{\log}$ throughout the rest of the article.
	
	\subsection{Intertwining with OU semigroups on Euclidean spaces}
	In the next result, we show that the horizontal projections of the semigroups $\mathbb{P}^{\psi}_t$ coincide with the OU semigroup on $\R^{n}$. To this end, let us introduce the operator $\Pi:L^p(\R^{n},\widetilde{\mu})\longrightarrow L^p(G,\mathsf{p}_\psi)$ defined by 
	\begin{align}
		\Pi f(h,v)=f(h),
	\end{align}
	where $\widetilde{\mu}(dh)=(2\pi)^{-\frac{n}{2}} e^{-|h|^2/2} dh$ is the Gaussian measure on $\R^n$.
	Then, we have the following result.
	\begin{proposition}\label{prop:horizontal_intertwining}
		For any $\psi\in\mathcal{N}_{\log}$ (see \eqref{eq:N_log} for definition) and $p\geqslant 1$, $\Pi:L^p(\R^{n},\widetilde{\mu})\longrightarrow L^p(G,\mathsf{p}_\psi)$ is a isometry. Moreover, for any $f\in L^p(\R^{n},\widetilde{\mu})$, one has
		\begin{align*}
			\mathbb{P}^\psi_t\Pi f=\Pi \widetilde{P}_t f,
		\end{align*}
		where $(\widetilde{P}_t)_{t\geqslant 0}$ is the OU semigroup on $L^p(\R^{n},\mu)$ generated by $\widetilde{L}=\Delta-\langle h, \nabla\rangle$.
	\end{proposition}
	\begin{proof}
		Let us define the projection map
		\begin{equation}
			\begin{aligned}
				\pi:G&\longrightarrow\R^{n} \\
				(h,v)&\longmapsto h.
			\end{aligned}
		\end{equation}
		With this notation, one has $\Pi f=f\circ\pi$. Denoting the heat kernel associated to $\mathbb{Q}^\psi_t$ by $\mathsf{q}^\psi_t$, we claim that $\widetilde{\mu}_t=\mathsf{q}^\psi_t\circ \pi^{-1}$, where $\widetilde{\mu}_t$ is the heat kernel associated with the Laplacian on $\R^{n}$. Indeed, putting $\lambda=0$ in \eqref{eq:heat_kernel_perturbed} one gets
		\begin{align}\label{eq:marginal}
			\int_{\R^m} \mathsf{q}^\psi_t(h,v) dv=\widetilde{\mu}_t(h),
		\end{align}
		which proves our claim. Since $\widetilde{\mu}=\widetilde{\mu}_{1/2}$, we have $\widetilde{\mu}=\mathsf{p}_\psi\circ\pi^{-1}$. This shows that $\Pi: L^p(\R^{n},\widetilde{\mu})\longrightarrow L^p(G,\mathsf{p}_\psi)$ is a isometry. Since $\pi$ also commutes with the dilations, that is,
		\begin{align*}
			\pi\delta_c=d_c\pi \text{  for all } c>0,
		\end{align*}
		\eqref{eq:marginal} implies that $\mathbb{P}_t\Pi=\Pi\widetilde{P}_t$ on $L^p(\R^n,\widetilde{\mu})$ for all $t>0$. Finally, for any $f\in C_0(\R^n)$, \eqref{eq:P_alpha} shows that 
		\begin{align*}
			\mathbb{P}^\psi_t \Pi f(h,v)&=\int_{\R^m}\mathbb{P}_t \Pi f(h,v-v')\mu^{\psi_t}(dv') \\
			&=\int_{\R^m}\Pi \widetilde{P}_t f(h,v-v')\mu^{\psi_t}(dv') \\
			&=\widetilde{P}_t f(h)=\Pi\widetilde{P}_t f(h,v).
		\end{align*}
		The above equality extends to $L^p(\R^n,\widetilde{\mu})$ due to the density of $C_0(\R^n)$ in $L^p(\R^n,\widetilde{\mu})$, which completes the proof.
	\end{proof}
	\begin{corollary}\label{cor:non-normal}
		For any $\psi\in\mathcal N$ and $t>0$, $\mathbb{P}^\psi_t$ is not a normal operator on $L^2(G,\mathsf{p}_\psi)$.
	\end{corollary}
	\begin{proof}
		Denoting the $L^2(G,\mathsf{p}_\psi)$-adjoint of $\mathbb{P}^\psi_t$ by $\mathbb{P}^{\psi*}_t$, we have 
		\begin{align*}
			\mathbb{P}^{\psi*}_tf=\frac{\widehat{\mathbb{P}}^\psi_t (\mathsf{p}_\psi f)}{\mathsf{p}_\psi} \text{ for all } f\in C_b(G).
		\end{align*}
		Now assume that $\mathbb{P}^\psi_t$ is a normal operator. Then by the Fuglede-Putnam-Rosenblum theorem (see \cite[Theorem~12.16, p.~ 315]{GrandpaRudinBook}), Proposition~\ref{prop:horizontal_intertwining} implies that 
		\begin{align*}
			\mathbb{P}^{\psi*}_t \Pi =\Pi \widetilde{P}_t \text{ on } L^2(\R^n,\mu).
		\end{align*}
		Therefore for all $f\in C_b(\R^n)$ one must have
		\begin{align*}
			\widehat{\mathbb{P}}^\psi_t(\mathsf{p}_\psi \Pi f)=\mathsf{p}_\psi \Pi \widetilde{P}_t f.
		\end{align*}
		Taking the derivative of the above equality at $t=0$ implies that at least for all $f\in C^\infty_c(\R^n)$, 
		\begin{align*}
			(\Delta_{\mathcal{H}}+\mathbb{D}+A^{\overline{\psi}}_{\mathcal V}+n+2m)(\mathsf{p}_\psi\Pi f)=\mathsf{p}_\psi\Pi\widetilde{L} f,
		\end{align*}
		which clearly does not hold for all $f\in C^\infty_c(\R^n)$. This leads to a contradiction, which completes the proof.
	\end{proof}
	
	\subsection{Intertwining between OU semigroups} We recall the semigroups $\mathbb{P}$, $P^\psi$, and $\mathbb{P}^\psi$ defined by \eqref{eq:OU_semigroup}, \eqref{eq:OU_levy}, and Theorem~\ref{prop:P_gamma}\eqref{it:3} respectively.
	Consider the Markov kernels $\Lambda:C_0(G)\longrightarrow C_0(\R^m)$ and $\Gamma_\psi: C_0(G)\longrightarrow C_0(G)$ such that 
	\begin{equation}\label{eq:Lambda}
		\begin{aligned}
			&\Lambda f(v)=\int_{\R^n}\int_{\R^m}f(h,v+v')\mathsf{p}(h,v')dv' dh, \\
			&\Gamma_\psi f(h,v)=\int_{\R^m}f(h, v+v')\mu_\psi(dv'),
		\end{aligned}
	\end{equation}
	where $\mathsf{p}$ is the invariant distribution of $\mathbb{P}$ that also coincides with $\mathsf{q}_{1/2}$, the horizontal heat kernel on $G$ at $t=1/2$, and $\mu_\psi$ defined in \eqref{eq:p_psi} is the invariant distribution of $P^\psi$.
	The next lemma shows that both $\Lambda$ and $\Gamma_\psi$ can be extended as bounded operators on weighted $L^p$-spaces.
	\begin{lemma}
		For any $\psi\in\mathcal{N}_{\log}$ we have 
		\begin{align*}
			\mu_\psi\Lambda&=\mathsf{p}_\psi, \\
			\mathsf{p}\Gamma_\psi&=\mathsf{p}_\psi.
		\end{align*}
		As a result, $\Lambda$ (resp. $\Gamma_{\psi}$) extends as a bounded operator between $L^p(G, \mathsf{p}_\psi)$ and $L^p(\R^m, \mu_\psi)$ (resp. $L^p(G, \mathsf{p}_\psi)$ and $L^p(G,\mathsf{p})$) for any $p\geqslant 1$.
	\end{lemma}
	\begin{proof}
		For any $f\in C_b(G)$ by a change of variable and using Fubini's theorem we note that 
		\begin{align*}
			\mu_\psi\Lambda f&=\int_{\R^m}\left(\int_{\R^n}\int_{\R^m}\mathsf{p}(h,v'-v)f(h,v') dv' dh\right)\mu_\psi(dv) \\
			&=\int_{\R^n}\int_{\R^m}\left(\int_{\R^m}\mathsf{p}(h,v'-v)\mu_\psi(dv)\right)f(h,v') dv' dh \\
			&=\int_{\R^n}\int_{\R^m}\mathsf{p}_\psi(h,v')f(h,v') dv' dh,
		\end{align*}
		where we used fact that $\mathsf{p}{\ast}_v \mu_\psi=\mathsf{p}_\psi$.
		Similarly, we have 
		\begin{align*}
			\mathsf{p}\Gamma_\psi f&=\int_{\R^n}\int_{\R^m}\left(\int_{\R^{m}} f(h, v+v')\mu_\psi(dv')\right)\mathsf{p}(h,v) dv dh \\
			&=\int_{\R^n}\int_{\R^m}\left(\int_{\R^m} \mathsf{p}(h,v-v')p_{\psi}(dv')\right) f(h, v)dv dh \\
			&=\int_{\R^n} \int_{\R^m}f(h,v)\mathsf{p}_\psi(h, v) dh dv.
		\end{align*}
		This completes the proof of the first statement. Next, for any $f\in L^p(G, \mathsf{p}_\psi)\cap C_b(G)$ with $p\geqslant 1$ one has 
		\begin{align*}
			\int_{\R^m} (\Lambda f)^p(v)p_{\psi}(dv)
			&\le \int_{\R^m} \Lambda (|f|^p)(v) p_{\psi}(dv) dv \\
			&= \int_{G}|f(h,v)|^p \mathsf{p}_\psi(h,v) dh dv.
		\end{align*}
		The above inequality extends to $L^p(G, \mathsf{p}_\psi)$ by a density argument. The extension of $\Gamma_\psi$ between the $L^p$-spaces can be proved similarly as above.
	\end{proof}

	We are now ready to state the main result of this section.
	\begin{theorem}\label{thm:OU_intertwining}
		For any $\psi\in\mathcal{N}_{\log}$ and $p\geqslant 1$, the following intertwining relationships hold.
		\begin{align}
			P^{\psi}_t \Lambda& = \Lambda \mathbb{P}^{\psi}_t \text{ on } L^p(G, \mathsf{p}_\psi), \label{eq:Q_P}
			\\
			\mathbb{P}_t \Gamma_\psi & =  \Gamma_\psi \mathbb{P}^{\psi}_t \text{ on } L^p(G, \mathsf{p}_\psi), \label{eq:P_P}
		\end{align}
		where $\mathbb{P}$ is defined in \eqref{eq:OU_semigroup}, and $P^\psi$ is the semigroup corresponding to the stochastic differential equation \eqref{eq:P_psi_sde}.
	\end{theorem}
	We now consider a different representation of the operators $\Lambda$ and $\Gamma_\psi$ as Fourier multiplier operators, which will be useful in the proof of the above theorem. We start by defining two operators $\mathcal{I}:L^2(\R^m)\longrightarrow L^2(G)$ and $\widetilde{\mathcal{I}}:L^2(G)\longrightarrow L^2(G)$ such that for all $f\in L^2(\R^m),\widetilde{f}\in L^2(G)$, and $(\lambda,\nu)\in\mathcal{Z}\times\R^m$
	\begin{equation}\label{eq:V_fourier}
		\begin{aligned}
			\mathcal{F}_G(\mathcal{I}f)(\lambda,\nu)&=\mathcal{F}(f)(\lambda)H^{\lambda,\nu}_{\frac{1}{2}}, 
			\\
			\mathcal{F}_G(\widetilde{\mathcal I}\widetilde{f})(\lambda,\nu)&=\exp(-\psi_{-\infty}(-\lambda))\mathcal{F}_G(\widetilde{f})(\lambda,\nu),
		\end{aligned}
	\end{equation}
	where $H^{\lambda,\nu}_{\frac12}$ is defined in \eqref{eq:harmonic_osc}.
	\begin{lemma}\label{lem:V_fourier}
		$\mathcal{I}:L^2(\R^m)\longrightarrow L^2(G)$ and $\widetilde{\mathcal I}:L^2(G)\longrightarrow L^2(G)$ are bounded and injective. Moreover, for any $f\in L^2(\R^m)\cap C_b(\R^m)$ and $\widetilde{f}\in L^2(G)\cap C_b(G)$, 
		\begin{align}
			\mathcal If(h,v)&=\int_{\R^m} \mathsf{p}(h, v-v')f(v') dv' \label{eq:V}, \\
			\widetilde{\mathcal I}\widetilde{f}(h,v)&=\Gamma_\psi \widetilde{f}(h,v)\label{eq:V_tilde},
		\end{align} 
		where $\Gamma_{\psi}$ is defined in \eqref{eq:Lambda}.
	\end{lemma}
	\begin{proof}
		Since $H^{\lambda,\nu}_{\frac{1}{2}}$ is a Hilbert-Schmidt operator, $\mathcal I$ is well defined on its domain. Now, let us verify the boundedness of $\mathcal I$. Using the Plancherel theorem for $\mathcal{F}_G$, for all $f\in L^2(\R^m)$ we have,
		\begin{align*}
			\|\mathcal If\|^2_{L^2(G)}=\frac{1}{(2\pi)^{d+k+m}}\int_{\R^k}\int_{\R^m} \|\mathcal{F}_G(\mathcal I f)(\lambda,\nu)\|^2_{\HS} \Pf(\lambda) d\lambda d\nu.
		\end{align*}
		Now, for each $t>0$, and $\beta\in\N^d_0$ we have 
		\begin{align*}
			H^{\lambda,\nu}_t \Phi^\lambda_\beta= e^{-t\mathfrak{n}(\beta,\lambda,\nu)} \Phi^\lambda_{\beta},
		\end{align*}
		where $\{\Phi^\lambda_{\beta}: \beta\in\N^d_0\}$ are defined in \eqref{eq:E}. As a result, 
		\begin{align*}
			\|H^{\lambda,\nu}_{\frac{1}{2}}\|^2_{\HS}&=\sum_{\beta\in\N^d_0} e^{-\mathfrak{n}(\beta,\lambda,\nu)}\\
			&=\sum_{\beta\in\N^d_0}\exp\left(-|\nu|^2-\sum_{j=1}^d (2\beta_j+1)\eta_j(\lambda)\right) \\
			&=e^{-|\nu|^2}\prod_{j=1}^d\sum_{\beta_j=0}^\infty e^{-(2\beta_j+1)\eta_j(\lambda)} = e^{-|\nu|^2}\prod_{j=1}^d \frac{e^{-\eta_j(\lambda)}}{1-e^{-2\eta_j(\lambda)}}.
		\end{align*}
		Since $\eta_j$ are smooth and homogeneous of degree $1$, one has $\lim_{\lambda\to 0}\eta_j(\lambda)=0$ for each $j=1,\ldots, d$. Therefore, the function defined by
		\begin{align*}
			u(\lambda):=\Pf(\lambda)\prod_{j=1}^d\frac{e^{-\eta_j(\lambda)}}{1-e^{-2\eta_j(\lambda)}}=\prod_{j=1}^d\frac{\eta_j(\lambda)e^{-\eta_j(\lambda)}}{1-e^{-2\eta_j(\lambda)}}
		\end{align*}
		is bounded on $\R^m$. Consequently, for all $f\in L^2(\R^m)$,
		\begin{align}
			\|\mathcal I f\|^2_{L^2(G)}&=\frac{1}{(2\pi)^{d+k+m}}\int_{\R^{m+k}} \|\mathcal{F}_G(\mathcal I f)(\lambda,\nu)\|^2_{\HS}\Pf(\lambda) d\lambda d\nu \nonumber \\
			&\le \frac{c}{(2\pi)^m}\int_{\R^{m}}|\mathcal{F}(f)(\lambda)|^2 d\lambda=c\|f\|^2_{L^2(\R^m)}, \label{eq:plancherel}
		\end{align}
		where 
		\begin{align*}
			c=\frac{1}{(2\pi)^{d+k}}\|u\|_\infty\int_{\R^k} e^{-|\nu|^2}d\nu,
		\end{align*}
		and \eqref{eq:plancherel} follows by the Plancherel theorem on $L^2(\R^m)$. To show the injectivity, let $f\in L^2(\R^m)$ be such that $\mathcal If=0$. Then, for all $(\lambda,\nu)\in\mathcal{Z}\times\R^k$, 
		\begin{align*}
			\mathcal{F}(f)(\lambda,\nu)H^{\lambda,\nu}_{\frac{1}{2}}=0.
		\end{align*}
		Since $H^{\lambda,\nu}_{\frac{1}{2}}\neq 0$ for any $(\lambda,\nu)\in\mathcal{Z}\times\R^k$, it implies  that $\mathcal{F}(f)(\lambda,\nu)=0$, which shows that $f=0$ a.e.
		Injectivity of $\widetilde{\mathcal I}$ follows by a similar argument and the boundedness follows directly from the Plancherel theorem for $\mathcal{F}_G$. Now, for any $f\in C^\infty_c(\R^m)$, we note that 
		\begin{align*}
			\mathcal{F}_G(\mathsf{p}{\ast}_v {f})(\lambda,\nu)=\mathcal{F}(f)(\lambda)\mathcal{F}_G(\mathsf{p})(\lambda,\nu)=\mathcal{F}(f)(\lambda)H^{\lambda,\nu}_{\frac12},
		\end{align*}
		which shows that $\mathcal{I}f=\mathsf{p}{\ast}_v{f}$ for any $f\in C^\infty_c(\R^m)$. Using the density of $C^\infty_c(\R^m)$ in $L^2(\R^m)\cap C_b(\R^m)$, the proof of \eqref{eq:V} follows. Proof of \eqref{eq:V_tilde} follows by a similar argument.
	\end{proof}
	
	In the next lemma, we establish the link between the operators $\mathcal{I}$ and $\Lambda$. 
	\begin{lemma}\label{lem:V=Lambda*}
		Let $\widehat{\mathcal I}: L^2(G)\longrightarrow L^2(\R^m)$ denote the adjoint of $\mathcal I$. Then 
		\begin{align*}
			\Lambda=\left. \widehat{\mathcal I}\right|_{L^2(G)\cap C_b(G)} \\
		\end{align*}
	\end{lemma}
	\begin{proof}
		For any $f\in L^2(\R^m)\cap C_b(\R^m)$ and $\widetilde{f}\in L^2(G)\cap C_b(G)$ we note that 
		\begin{align*}
			\langle\mathcal{I}f, \widetilde{f}\rangle_{L^2(G)}&=\int_{\R^n}\int_{\R^m}\left(\int_{\R^m} \mathsf{p}(h,v-v') f(v') dv'\right) \widetilde{f}(h,v)dvdh \\
			&=\int_{\R^m}\left(\int_{\R^n}\int_{\R^m} \mathsf{p}(h,v)\widetilde{f}(h,v+v')dv dh\right) f(v')dv' \\
			&= \langle f, \Lambda \widetilde{f}\rangle_{L^2(\R^m)},
		\end{align*}
		where the second to the last equality is justified by Fubini's theorem. 
	\end{proof}
	\begin{lemma}\label{lem:psi_t}
		For any $t\in [-\infty,\infty)$, let $\psi_t$ be defined by \eqref{eq:psi_t}. Then for all $t\in [-\infty,\infty)$ and $\lambda\in\R^m$ we have
		\begin{align*}
			-\psi_{-\infty}(e^{2t}\lambda)=-\psi_{-\infty}(\lambda)+\psi_t(\lambda).
		\end{align*}
	\end{lemma}
	\begin{proof}
		The proof follows directly from the definition of $\psi_t$ in \eqref{eq:psi_t}.
	\end{proof}
	\begin{proof}[Proof of Theorem~\ref{thm:OU_intertwining}]
		For \eqref{eq:Q_P}, we will prove the adjoint of the intertwining relationship, that is, for all $f\in L^2(\R^m)$,
		\begin{align}\label{eq:intertwining_adj}
			\widehat{\mathbb{P}}^{\psi}_t \mathcal{I} f =\mathcal{I}\widehat{P}^{\psi}_t f,
		\end{align}
		where $\mathcal{I}$ is defined in \eqref{eq:V_fourier}. Using the result of Proposition~\ref{prop:fourier_transform}, for any $f\in L^2(\R^m)$ we have 
		\begin{align}
			\mathcal{F}_G(\widehat{\mathbb P}^{\psi}_t \mathcal{I} f)(\lambda)&= e^{\psi_t(-e^{-2t}\lambda) } d_{e^t}\mathcal{F}_G(\mathcal I f)(e^{-2t}\lambda, e^{-t}\nu) d_{e^{-t}}H^{\lambda,\nu}_{\frac{1-e^{-2t}}{2}} \nonumber \\
			&=e^{\psi_t(-e^{-2t}\lambda) }\mathcal{F}(f)(e^{-2t}\lambda)d_{e^t} H^{e^{-2t}\lambda, e^{-t}\nu}_{\frac{1}{2}} d_{e^{-t}} H^{\lambda,\nu}_{\frac{1-e^{-2t}}{2}} \nonumber \\
			&=e^{\psi_t(-e^{-2t}\lambda) }\mathcal{F}(f)(e^{-2t}\lambda) H^{\lambda,\nu}_{\frac{1}{2}}, \label{eq:1}
		\end{align}
		where the last line follows from Lemma~\ref{lem:scaling_Harmonic} and the semigroup property of $(H^{\lambda,\nu}_t)_{t\geqslant 0}$. On the other hand, 
		\begin{align*}
			\mathcal{F}_G(\mathcal I \widehat{P}^{\psi}_t f)(\lambda,\nu)&=\mathcal{F}(\widehat{P}^{\psi}_t f)(\lambda)H^{\lambda,\nu}_{\frac{1}{2}} \\
			&= e^{\psi_t(-e^{-2t}\lambda) }\mathcal{F}(f)(e^{-2t}\lambda)H^{\lambda,\nu}_{\frac{1}{2}} \\
			&=\mathcal{F}_G(\widehat{\mathbb P}^{\psi}_t \mathcal{I}f)(\lambda,\nu),
		\end{align*}
		where the last identity follows from \eqref{eq:1}. Finally using Lemma~\ref{lem:V=Lambda*} and taking the adjoint of the identity in \eqref{eq:intertwining_adj}, \eqref{eq:Q_P} follows. Next, to prove \eqref{eq:P_P}, we note that for any $f\in C^\infty_c(G)$, \eqref{eq:P} with $\psi\equiv 0$ yields
		\begin{align*}
			\mathcal{F}_G(\mathbb{P}_t \widetilde{\mathcal{I}} f)(\lambda,\nu)&=e^{-(n+2m)t} d_{e^{-t}}\mathcal{F}_G(\widetilde{\mathcal{I}}f)(e^{2t}\lambda, e^t\nu)d_{e^t} H^{\lambda,\nu}_{\frac{e^{2t}-1}{2}}\\
			&=e^{-(n+2m)t}e^{-\psi_{-\infty}(e^{2t}\lambda)}d_{e^{-t}}\mathcal{F}_G(f)(e^{2t}\lambda, e^t\nu)d_{e^t} H^{\lambda,\nu}_{\frac{e^{2t}-1}{2}} \\
			&=e^{-\psi_{-\infty}(\lambda)}e^{-(n+2m)t} e^{\psi_t(\lambda)}d_{e^{-t}}\mathcal{F}_G(f)(e^{2t}\lambda, e^t\nu)d_{e^t} H^{\lambda,\nu}_{\frac{e^{2t}-1}{2}}\\
			&=\mathcal{F}_G(\widetilde{\mathcal{I}} \mathbb{P}^{\psi}_tf)(\lambda,\nu),
		\end{align*}
		where we repeatedly used Proposition~\ref{prop:fourier_transform} along with Lemma~\ref{lem:psi_t}. Thus by injectivity of GFT and the density of $C^\infty_c(G)$ in $L^2(G)$, we have 
		\begin{align*}
			\mathbb{P}_t\widetilde{\mathcal{I}} f=\widetilde{\mathcal{I}}\mathbb{P}^{\psi}_tf
		\end{align*}
		for all $f\in L^2(G)$.
		The proof is concluded by using \eqref{eq:V_tilde} in Lemma~\ref{lem:V_fourier}.
	\end{proof}
	\subsection{Compactness and spectrum of perturbed OU semigroup} In this section, we consider a further sub-class of the L\'evy-Khintchine exponents $\mathcal{N}$ that satisfies the finite exponential moment condition. Let us define 
	\begin{align}\label{eq:exponential_moment}
		\mathcal{N}_{\exp}=\left\{\psi=(\sigma, b,\kappa)\in\mathcal{N}: \int_{|v|>1} e^{a |v|}\kappa(dv)<\infty \text{ for all } a>0 \right\}.
	\end{align}
	We note that $\mathcal{N}_{\exp}\subset\mathcal{N}_{\log}$, where $\mathcal{N}_{\log}$ is defined by \eqref{eq:N_log}, and therefore according to Theorem~\ref{prop:P_gamma}, $\mathbb{P}^\psi$ is ergodic with the unique invariant density $\mathsf{p}_\psi$ whenever $\psi\in\mathcal{N}_{\log}$. We need to restrict ourselves to $\mathcal{N}_{\exp}$ to ensure that the space of polynomials in $G$ is included in $L^p(G,\mathsf{p}_\psi)$ for all $p\geqslant 1$.
	If $(S_t)_{t\geqslant 0}$ denotes the L\'evy process associated to $\psi$, the above condition is also equivalent to $\mathbb{E}(e^{\xi|S_t|})<\infty$ for all $\xi>0$ and $t>0$, see \cite[Theorem~25.3]{SatoBook1999}. In the next proposition we prove compactness of the semigroup $\mathbb{P}^\psi$.
	\begin{proposition}\label{prop:compactness}
		For any $\psi\in\mathcal{N}_{\exp}$, where $\mathcal{N}_{\exp}$ is defined by \eqref{eq:exponential_moment}, $\mathbb{P}^\psi_t$ is a Hilbert-Schmidt operator on $L^2(G,\mathsf{p}_\psi)$.
	\end{proposition}
	\begin{proof}
		We adapt an idea very similar to the proof of \cite[Proposition~6(a)]{Lust-Piquard2010}. Writing 
		\begin{align*}
			\mathbb{P}^{\psi}_tf(g,g')=\int_{G}f(g')\frac{\mathsf{p}^\psi_t(g,g')}{\mathsf{p}_\psi(g')} \mathsf{p}_\psi(g')dg',
		\end{align*}
		it suffices to show that 
		\begin{align}
			I(t):=\int\limits_{G\times G} \mathsf{p}^\psi_t(g,g')^2 \frac{\mathsf{p}_\psi(g)}{\mathsf{p}_\psi(g')} dg dg'<\infty.
		\end{align}
		Now, for any $g=(h,v)$ and $g'=(h',v')$,
		\begin{align*}
			\mathsf{p}^\psi_t(g,g')=\int_{\R^m}\mathsf{p}_t((h,v-\widetilde{v}), (h',v')) \mu^{\psi_t}(d\widetilde{v}),
		\end{align*}
		where $\mathsf{p}_t$ is the kernel associated with the semigroup $\mathbb{P}_t$. Let $\rho$ denote the Carnot-Carath\'eodory distance on $G$. Then, due to \cite[Proposition~5.1.4]{BonfiglioliLanconelliUguzzoniBook}, there exists a positive constant $c$ such that 
		\begin{align}
			c^{-1}|g-g'|_{G}\le \rho(g,g')\le c |g-g'|_G,
		\end{align}
		where $|g|_G$ is the homogeneous norm on $G$ defined as $|(h,v)|_G=\sqrt{|h|^2+|v|_1}$, where $|\cdot|_1$ denotes the $l^1$ norm on $\R^m$.
		Therefore, using the triangle inequality one gets
		\begin{align*}
			\rho(g,g')^2-c|\widetilde{v}|_1\le \rho((h,v-\widetilde{v}), (h',v'))^2\le \rho(g,g')^2+c^{-1}|\widetilde{v}|_1.
		\end{align*} 
		Also, using the estimates of the heat kernel $\mathsf{q}_t$ in the proof of \cite[Proposition~ 6(a)]{Lust-Piquard2010} we obtain that for any $0<\epsilon<1$, there exists a constant $C_\epsilon$ such that
		\begin{align*}
			\mathsf{p}^\psi_t(g,g')^2\le C_\epsilon \exp\left(-\frac{\rho(\delta_{(e^{2t}-1)^{-1/2}}g^{-1}\delta_{(1-e^{-2t})^{-1/2}}g')^2}{1+\epsilon}\right)\\
			\times\int_{\R^m}\exp\left(\frac{c|\widetilde{v}|_1}{(1+\epsilon)(e^{2t}-1)}\right)\mu^{\psi_t}(d\widetilde{v}),
		\end{align*}
		where $\mu^{\psi_t}$ is defined in \eqref{eq:q_psi_t}, and we also used the Jensen's inequality with respect to the probability measure $\mu^{\psi_t}$. Since $\psi_t\in\mathcal{N}_{\exp}$ whenever $\psi\in\mathcal{N}_{\exp}$, the integral above on the right hand side is convergent for all $t>0$. This shows that for any $\epsilon\in (0,1)$,
		\begin{align*}
			\mathsf{p}^\psi_t(g,g')^2\le C_{\epsilon} \exp\left(-\frac{\rho(\delta_{(e^{2t}-1)^{-1/2}}g^{-1}\delta_{(1-e^{-2t})^{-1/2}}g')^2}{1+\epsilon}\right) \text{ for all } t>0
		\end{align*}
		for some constant $C_{\epsilon}>0$. A very similar argument also leads to the following estimates:
		\begin{equation}\label{eq:p_psi_estimate}
			\begin{aligned}
				\mathsf{p}_\psi(g)&\le C_\epsilon\exp\left(-\frac{\rho(g)^2}{2+2\epsilon}\right)\int_{\R^m} e^{\frac{c|\widetilde{v}|_1}{2+2\epsilon}} \mu_\psi(d\widetilde v) \\
				\mathsf{p}_\psi(g')& \geqslant K_\epsilon\exp\left(-\frac{\rho(g')^2}{2-2\epsilon}\right)\int_{\R^m} e^{-\frac{c^{-1}|\widetilde{v}|_1}{2-2\epsilon}} \mu_\psi(d\widetilde v)
			\end{aligned}
		\end{equation}
		for some positive constants $C_\epsilon, K_\epsilon$. As a result, using the same argument as in the proof of \cite[Proposition~6(a)]{Lust-Piquard2010}, one can show that $I(t)<\infty$ for all $t>0$, which proves the proposition.
	\end{proof}
	\begin{corollary}
		Let $\psi\in\mathcal{N}_{\exp}$ (see \eqref{eq:exponential_moment} for definition). Then, for all $1<p<\infty$, $\mathbb{P}^\psi_t:L^p(G,\mathsf{p}_\psi)\longrightarrow L^p(G,\mathsf{p}_\psi)$ is a compact operator for all $t>0$.
	\end{corollary}
	\begin{proof}
		Since $\mathbb{P}^\psi_t$ is compact on $L^2(G,\mathsf{p}_\psi)$ and bounded on $L^\infty(G, \mathsf{p}_\psi)$ and $L^1(G,\mathsf{p}_\psi)$, by interpolation, $\mathbb{P}^\psi_t$ is compact on $L^p(G,\mathsf{p}_\psi)$ for any $1<p<\infty$.
	\end{proof}
	We denote the set of all polynomials on $G$ by $\mathcal P$ and the set of homogeneous polynomials in $\mathcal P$ with degree $n$ by $\mathcal{P}_n$. Equivalently, $\mathcal{P}_n$ is the eigenspace of $\mathbb D$ on $\mathcal{P}$ associated to the eigenvalue $n$. We also define $\mathcal{B}_n$ to be the space of all polynomials on $G$ with degree at most $n$. Note that 
	$\mathcal{B}_n=\mathcal{P}_0+\cdots +\mathcal{P}_n$ for each $n\geqslant 1$. Thus, $\mathcal{P}=\cup_{n\geqslant 0}\mathcal{B}_n$. Also, both $\mathbb{Q}_{1/2}$ and $\Gamma_\psi$ (see \eqref{eq:Lambda} for definition) are invertible on $\mathcal B_n$, hence on $\mathcal P$ whenever $\psi\in\mathcal{N}_{\exp}$. Moreover, from the estimate in \eqref{eq:p_psi_estimate} it follows that 
	\begin{align*}
		\int_{G} e^{a |g|_1}\mathsf{p}_\psi(dg)<\infty
	\end{align*}
	for all $a>0$, where $|g|_1$ denotes the $l^1$-norm of $g$ in $\R^{n+m}$. As a result, by \cite{AkhiezerBook2021}, $\mathcal{P}$ is dense in $L^p(G,\mathsf{p}_\psi)$ for all $p\ge 1$. In the next theorem, we show that for any $\psi\in\mathcal N_{\exp}$, the spectrum of $\mathbb{P}^\psi$ does not depend on $\psi$. This phenomenon is very different from what we have observed for the spectrum of $\mathbb{Q}^\psi_t$ (equivalently, $\Delta^{\!\psi}_{G}$), which clearly depends on $\psi$, according to Corollary~\ref{cor:spec_delta_psi}.
	
	\begin{theorem}\label{thm:spectrum} Let $\psi\in\mathcal{N}$. Then, for all $t>0$ we have the following:
		\begin{enumerate}[leftmargin=*]
			\item \label{it:spec1}For all $\psi\in\mathcal{N}_{\log}$ and $1<p<\infty$, $e^{-t\N_0}\subseteq \sigma_{\rm p}(\mathbb{P}^\psi_t; L^p(G,\mathsf{p}_\psi))$.
			\item \label{it:spec2} If $\psi\in\mathcal{N}_{\exp}$, the spectrum of $\mathbb{P}^\psi_t$ in $L^p(G, \mathsf{p}_\psi)$ does not depend on $p\in (1,\infty)$. Moreover,  
			\begin{align*}
				\sigma(\mathbb{P}^\psi_t; L^p(G,\mathsf{p}_\psi))\setminus \{0\}=\sigma_{\rm p}(\mathbb{P}^\psi_t; L^p(G, \mathsf{p}_\psi))=e^{-t\N_0}.
			\end{align*}
			\item \label{it:spec_3}For any $1<p<\infty$, the eigenspace of $\mathbb{P}^\psi_t$ in $L^p(G, \mathsf{p}_\psi)$ associated to $e^{-nt}$ is given by $\mathcal{E}^\psi_n=(\mathbb{Q}_{\frac12}\Gamma_\psi)^{-1}(\mathcal P_n)$, where $\Gamma_\psi$ is defined in \eqref{eq:Lambda}.
			
		\end{enumerate}
	\end{theorem}
	
	For any closed operator $A:\mathcal{D}(A)\subseteq E\longrightarrow E$, where $E$ is a Banach space, the algebraic multiplicity of an eigenvalue $\theta$ of $A$ is defined as 
	\begin{align*}
		\mathtt{M}_a(\theta, A)=\dim\left(\cup_{n=1}^\infty \ker (A-\theta I)^n\right),
	\end{align*}
	and the geometric multiplicity is defined as $\mathtt{M}_g(\theta,A)=\dim(\ker(A-\theta I))$. We note that both algebraic and geometric multiplicities can be infinite. The next proposition shows that the eigenvalues of $\mathbb{L}$ and $\mathbb{L}_\psi$ have the same multiplicities whenever $\psi\in\mathcal{N}_{\exp}$.
	\begin{theorem}\label{prop:multiplicity} 
		Let $\mathbb{L}_\psi$ (resp.~$\mathbb{L})$ denote the generator of the semigroup $\mathbb{P}^\psi_t:L^p(G,\mathsf{p}_\psi)\longrightarrow L^p(G,\mathsf{p}_\psi)$ (resp.~$\mathbb{P}_t:L^p(G,\mathsf{p})\longrightarrow L^p(G,\mathsf{p})$). Then for any $1<p<\infty$, $\psi\in\mathcal{N}_{\exp}$ and $\beta\in\mathbb N_0$ we have 
		\begin{align*}
			\mathtt{M}_a(-\beta, \mathbb{L}_\psi)&=\mathtt{M}_a(-\beta,\mathbb{L})<\infty, \\
			\mathtt{M}_g(-\beta, \mathbb{L}_\psi)&=\mathtt{M}_g(-\beta,\mathbb{L})<\infty.
		\end{align*}
	\end{theorem}
	In particular, choosing $\psi(\lambda)=-\epsilon |\lambda|^2$ for $\epsilon>0$, we have the following corollary.
	\begin{corollary}
		For any $\epsilon>0$, let us denote 
		\begin{align*}
			\mathbb{L}_\epsilon=\mathbb{L}+\epsilon\Delta_{\mathcal V}.
		\end{align*}
		Then, $\mathbb{L}_\epsilon$ generates a perturbed OU semigroup on $G$, denoted by $\mathbb{P}^\epsilon$. Denoting its invariant density by $\mathsf{p}_\epsilon$, for all $1<p<\infty$ we have 
		\begin{align*}
			\sigma(\mathbb{L}_\epsilon; L^p(G, \mathsf{p}_\epsilon))=\sigma_{\rm p}(\mathbb{L}_\epsilon; L^p(G, \mathsf{p}_\epsilon))=-\mathbb{N}_0.
		\end{align*}
		Also, for any $\beta\in\mathbb{N}_0$ and for all $\epsilon>0$ we have
		\begin{align*}
			\mathtt{M}_a(-\beta, \mathbb{L}_\epsilon)&=\mathtt{M}_a(-\beta,\mathbb{L})<\infty, \\
			\mathtt{M}_g(-\beta, \mathbb{L}_\epsilon)&=\mathtt{M}_g(-\beta,\mathbb{L})<\infty.
		\end{align*}
	\end{corollary}
	In order to prove the above results, we need the following two lemmas. 
	\begin{lemma}\label{lem:LP_intertwining}
		For all $t\geqslant 0$, we have 
		\begin{align*}
			\mathbb{Q}_{\frac12}\Gamma_\psi \mathbb{P}^\psi_{t}=\delta_{e^{-t}} \mathbb{Q}_{\frac12}\Gamma_\psi \text{ on } \mathcal P.
		\end{align*}
	\end{lemma}
	
	\begin{proof}
		By \cite[Lemma~5(b)]{Lust-Piquard2010} it is known that $\mathbb{Q}_{\frac12}\cos^{\mathbb{L}}\theta=\delta_{\cos\theta}\mathbb{Q}_{\frac12}$. Writing $\cos\theta=e^{-t}$ for $t\geqslant 0$ and using the intertwining relationship \eqref{eq:P_P} concludes the  proof of the lemma.
	\end{proof}
	The next lemma is about the spectrum and eigenspaces of a general compact operator and its proof can be found in the proof of \cite[Theorem~7]{Lust-Piquard2010}.
	\begin{lemma}\label{lem:compact_operator}
		Let $T:E\longrightarrow E$ be a compact operator on a separable Banach space $E$. Let $\Theta$ be a set of eigenvalues of $T$ and let $E_\theta, \theta\in\Theta$ be eigensubspaces of $T$ whose union is total in $E$, i.e., $\Span\{\cup_{\theta\in\Theta} E_\theta\}$ is dense in $E$. Then 
		\begin{enumerate}[leftmargin=*]
			\item The spectrum of $T$ is $\Theta\cup\{0\}$.
			\item $E_\theta$ is the whole eigenspace of $T$ associated to $\theta\in\Theta$.
		\end{enumerate}
	\end{lemma}
	\begin{proof}[Proof of Theorem~\ref{thm:spectrum}] From Proposition~\ref{prop:horizontal_intertwining} we have $\mathbb{P}^\psi_t\Pi=\Pi \widetilde{P}_t$ on $L^p(\R^n,\mu)$ for all $t\geqslant 0$ and $p>1$. Since $\Pi$ is injective and $\sigma_{\rm p}(\widetilde{P}_t; L^p(\R^n,\mu))=e^{-t\N_0}$, \eqref{it:spec1} follows immediately. 
		
		Next, we observe that $\mathbb{Q}_{\frac12}\Gamma_\psi:\mathcal{P}_n\longrightarrow \mathcal{P}_n$ is bijective for each $n\geqslant 1$. Therefore, Lemma~\ref{lem:LP_intertwining} implies that for all $t\geqslant 0$,
		\begin{align*}
			\mathbb{P}^\psi_t (\mathbb{Q}_{\frac12}\Gamma_\psi)^{-1}=(\mathbb{Q}_{\frac12}\Gamma_\psi)^{-1} \delta_{e^{-t}} \text{ on } \mathcal P.
		\end{align*}
		Since $\mathcal{P}_n$ is an eigenspace of $\delta_{e^{-t}}$ associated to $e^{-nt}$, the above intertwining relation implies that $\mathcal{E}^\psi_n=(\mathbb{Q}_{1/2}\Gamma_\psi)^{-1}(\mathcal P_n)$ is an eigenspace of $\mathbb{P}^\psi_t$ associated to $e^{-nt}$. Also, from the equality $(\mathbb{Q}_{1/2}\Gamma_\psi)^{-1}(\mathcal P)=\mathcal P$ we get that
		\begin{align*}
			\bigcup_{n=0}^\infty\mathcal{E}^\psi_n=\bigcup_{n=0}^\infty (\mathbb{Q}_{1/2}\Gamma_\psi)^{-1}(\mathcal P_n)=\mathcal{P},
		\end{align*}
		and therefore, $\Span\{\cup_{n=0}^\infty\mathcal{E}^\psi_n\}$ is dense in $L^p(G,\mathsf{p}_\psi)$. Also, by Proposition~\ref{prop:compactness}, there exists $T_\psi>0$ such that $\mathbb{P}^\psi_t$ is compact on $L^p(G,\mathsf{p}_\psi)$ for all $t>T_\psi$. Thus by Lemma~\ref{lem:compact_operator} for all $t>T_\psi$ we have 
		\begin{align}\label{eq:P_spect}
			\sigma(\mathbb{P}^\psi_t; L^p(G, \mathsf{p}_\psi))\setminus\{0\}=\sigma_{\rm p}(\mathbb{P}^\psi_t; L^p(G, \mathsf{p}_\psi))=\{e^{-nt}: n\in\N_0\},
		\end{align}
		and $\mathcal{E}^\psi_n$ is the eigenspace of $\mathbb{P}^\psi_t$ corresponding to the eigenvalue $e^{-nt}$. By the spectral mapping theorem for semigroups, e.~g. \cite[p.~180, Equation (2.7)]{EngelNagelBook2006}, it follows that \eqref{eq:P_spect} holds for all $t>0$.
		It remains to show that for any $0<t<T_\psi$, the eigenspace of $\mathbb{P}^\psi_t$ corresponding to the eigenvalue $e^{-nt}$ is still given by $\mathcal{E}^\psi_n$. Note that Lemma~\ref{lem:LP_intertwining} implies that $\mathcal{E}^\psi_n$ is included in the eigenspace of $\mathbb{P}^\psi_t$ associated to the eigenvalue $e^{-nt}$ for any $t>0$. For $t<T_\psi$, let $\mathcal{E}^\psi_{n,t}$ denote the eigenspace of $\mathbb{P}^\psi_t$ in $L^p(G,\mathsf{p}_\psi)$ associated to $e^{-nt}$. Also, let $k\in\mathbb{N}$ be such that $kt>T_\psi$. Using the semigroup property of $\mathbb{P}^\psi$, for any $f\in\mathcal{E}^\psi_{n,t}$ we have 
		\begin{align*}
			e^{-knt} f= \underbrace{\mathbb{P}^\psi_t \cdots \mathbb{P}^\psi_t}_{k \text{ times }} f=\mathbb{P}^\psi_{kt} f.
		\end{align*}
		Since $kt>T_\psi$, from the above argument it follows that $f\in\mathcal{E}^\psi_n$. Therefore, $\mathcal{E}^\psi_{n,t}\subseteq\mathcal{E}^\psi_n$, which implies that $\mathcal{E}^\psi_{n,t}=\mathcal{E}^\psi_n$ for all $0<t<T_\psi$. This proves \eqref{it:spec_3}, and the proof of theorem is complete. \end{proof}
	
	\subsection{Regularity properties of perturbed OU semigroups}\label{sec:regularity}
	In order to prove Proposition~\ref{prop:multiplicity}, we will first prove that any generalized eigenfunction of $\mathbb{L}_\psi$ must be a polynomial. To achieve this, let us denote the  \emph{right-invariant} vector field corresponding to $X\in\mathfrak{g}$ by $X^r$. Then, the following is known from \cite{Melcher2008}. 
	\begin{lemma}[Lemma~3.2 in \cite{Melcher2008}]\label{lem:Melcher}
		Let $P_t$ be the left-invariant horizontal heat semigroup on a nilpotent Lie group $G$. Then, for any $X\in\mathfrak{g}$ and $f\in C^\infty_c(G)$
		\begin{align*}
			X^rP_t f=P_t X^r f.
		\end{align*}
	\end{lemma}
	For a multi-index $\alpha\in\{1,\ldots, n\}^k$, let denote $X^r_\alpha=X^r_{\alpha_1}\cdots X^r_{\alpha_k}$. We also denote the weighted Sobolev spaces with respect to the right-invariant vector fields by $W^{k,p}(G, \mathsf{p}_\psi)$, that is, 
	\begin{align*}
		W^{k,p}(G,\mathsf{p}_\psi)=\left\{f\in L^p(G,\mathsf{p}_\psi): X^r_\alpha f\in L^p(G,\mathsf{p}_\psi) \text{ for each } \alpha\in\{1, \ldots, n\}^k \right\},
	\end{align*}
	where $X^r_\alpha f$ is defined in the weak sense. We note that for all $1<p<\infty$ and $\psi\in\mathcal{N}_{\log}$, the set of compactly supported smooth functions $C^\infty_c(G)$ is dense in $W^{k,p}(G,\mathsf{p}_\psi)$. Indeed, a simple truncation argument shows that the set of all $W^{k,p}(G,\mathsf{p}_\psi)$-functions with compact support is dense. On the other hand, given any $f\in W^{k,p}(G,\mathsf{p}_\psi)$ with compact support, it follows from \cite[Proposition~3.3]{BrunoPelosoTabaccoVallarino2019} with $\chi=1$ (see also \cite[Proposition~19]{CoulhonRussTardivel-Nachef2001}) that $f$ can be approximated by functions in $C^\infty_c(G)$. The next two lemmas provide some regularity properties of the semigroup $\mathbb{P}^\psi$. 
	\begin{lemma}\label{lem:regularity}
		For all $t>0$, $\mathbb{P}^\psi_t$ maps $L^p(G,\mathsf{p}_\psi)$ to $W^{k,p}(G,\mathsf{p}_\psi)$ continuously for all $1<p<\infty$ and $k\geqslant 1$.
	\end{lemma}
	\begin{proof}
		First, we show that for all $t>0$, $1<q<\infty$, $i\in\{1,\ldots, n\}$ and $f\in C^\infty_c(G)$, 
		\begin{align}\label{eq:reverse-poincare}
			|X^r_i \mathbb{Q}_tf|\le \frac{C}{\sqrt t} \mathbb{Q}_t(|f|^p)^{\frac 1p} 
		\end{align}
		for some positive constant $C$. Indeed, using the same technique from the proof of \cite[Theorem~3.1]{BakryBaudoinBonnefontChafai2008}, we obtain that 
		\begin{align*}
			X^r_i \mathbb{Q}_1f=\mathbb{Q}_1 X^r_if=-\langle f, X^r_i\mathsf{q}_1\rangle,
		\end{align*}
		where $\mathsf{q}_1$ is the horizontal heat kernel at $t=1$ and the last identity follows from integration by parts. Noting that $X^r_i \mathsf{q}_1=X^r_i (\log \mathsf{q}_1) \mathsf{q}_1$, using H\"older's inequality we obtain 
		\begin{align*}
			|X^r_i\mathbb{Q}_1f|\le \mathbb{Q}_1(|f|^p)^{\frac1p} \left(\int_G |X^r_i(\log \mathsf{q}_1)|^{q}\right)^{\frac{1}{q}}=C\mathbb{Q}_1(|f|^p)^{\frac1p}.
		\end{align*}
		Since $c\delta_c X^r_i= X^r_i \delta_c$ and $\delta_{\frac{1}{\sqrt t}} \mathbb{Q}_{1}\delta_{\sqrt t}=\mathbb{Q}_t$ for all $t, c>0$, the above inequality yields \eqref{eq:reverse-poincare}. Recall that $\mathbb{P}_t=\delta_{e^{-t}}\mathbb{Q}_{\frac{1-e^{-2t}}{2}}$, therefore \eqref{eq:reverse-poincare} implies that for all $f\in C^\infty_c(G)$ and $t>0$, 
		\begin{align*}
			|X^r_i \mathbb{P}_tf|\le \frac{C}{\sqrt{e^{2t}-1}} \mathbb{P}_t( |f|^p|)^{\frac1p}.
		\end{align*}
		Now, for any $\psi\in\mathcal{N}_{\log}$ we can use \eqref{eq:P_alpha}, and then for any $f\in C^\infty_c(G)$ we have 
		\begin{align*}
			& |X^r_i \mathbb{P}^\psi_tf|^p \\
			&\le \left|\int_G X^r_i \mathbb{P}_tf(h, v-v') \mu^{\psi_t}(dv')\right|^p \\
			&\le \int_G |X^r_i \mathbb{P}_tf(h, v-v')|^p \mu^{\psi_t}(dv') \\
			&\le \frac{C^p}{(e^{2t}-1)^{\frac p2}}\int_G \mathbb{P}_t(|f|^p)(h,v-v') \mu^{\psi_t}(dv') \\
			&=\frac{C^p}{(e^{2t}-1)^{\frac p2}} \mathbb{P}^\psi_t (|f|^p).
		\end{align*}
		Integrating the above inequality with respect to the invariant distribution $\mathsf{p}_\psi$, we get 
		\begin{align}\label{eq:sobolev_contraction}
			\|X^r_i \mathbb{P}^\psi_tf\|_{L^p(G, \mathsf{p}_\psi)} \le \frac{C}{\sqrt{e^{2t}-1}} \|f\|_{L^p(G, \mathsf{p}_\psi)}.
		\end{align}
		Using the density of $C^\infty_c(G)$ in $W^{1,p}(G, \mathsf{p}_\psi)$, we conclude that \eqref{eq:sobolev_contraction} holds for all $f\in W^{1,p}(G,\mathsf{p}_\psi)$, which proves the statement of the lemma for $k=1$. For $k\geqslant 2$, the proof follows in the same way with the additional observation 
		\begin{align*}
			X^r_\alpha\mathbb{P}^\psi_tf=e^{-kt}\mathbb{P}^\psi_tf
		\end{align*}
		for all $f\in W^{k,p}(G,\mathsf{p}_\psi)$ with $|\alpha|=k$. The last identity holds for all $f\in C^\infty_c(G)$ and extends to $W^{k,p}(G,\mathsf{p}_\psi)$ by density.
	\end{proof}
	\begin{lemma}\label{lem:sobolev_bound}
		For any $f\in W^{k,p}(G,\mathsf{p}_\psi)$, $\alpha\in\{1,\ldots, n\}^k$, and $\psi\in\mathcal N_{\log}$, we have 
		\begin{align*}
			\|X^r_\alpha \mathbb{P}^\psi_tf\|_{L^p(G,\mathsf{p}_\psi)}\le e^{-kt}\|X^r_\alpha f\|_{L^p(G,\mathsf{p}_\psi)}.
		\end{align*}
	\end{lemma}
	\begin{proof}
		It is enough to prove the above lemma for $f\in C^\infty_c(G)$. The rest will follow from the density of $C^\infty_c(G)$ in $W^{k,q}(G,\mathsf{p}_\psi)$. Using Lemma~\ref{lem:Melcher}, for any $f\in C^\infty_c(G)$ and $\alpha\in\{1,\ldots, n\}^k$ we have 
		\begin{align*}
			X^r_\alpha \mathbb{P}_tf=e^{-kt} \mathbb{P}_t X^r_\alpha f,
		\end{align*}
		where we used the fact that $\mathbb{P}_t=\delta_{e^{-t}}\mathbb{Q}_{\frac{1-e^{-2t}}{2}}$ and $X^r\delta_{e^{-t}}=e^{-t}\delta_{e^{-t}} X^r$ for any $X\in\mathfrak{g}$. Now, for any $\psi\in\mathcal N_{\log}$ and $f\in C^\infty_c(G)$ we have
		\begin{align*}
			X^r_\alpha \mathbb{P}^\psi_tf=X^r_\alpha \int_{\R^m} \mathbb{P}_tf(v,w-w') \mu^{\psi_t}(dw').
		\end{align*}
		Since $X^r_\alpha$ commutes with the vertical vector fields and $f, \mathbb{P}_tf$ are smooth, we have
		\begin{align*}
			X^r_\alpha \int_{\R^m} \mathbb{P}_tf(v,w-w') \mu^{\psi_t}(dw')&=e^{-kt}\int_{\R^m} X^r_\alpha\mathbb{P}_tf(v,w-w')\mu^{\psi_t}(dw')\\
			&=e^{-kt} X^r_\alpha\mathbb{P}^\psi_tf.
		\end{align*}
		Finally, using the contraction property of $\mathbb{P}^\psi_t$ on $L^p(G,\mathsf{p}_\psi)$ we conclude the proof of the lemma.
	\end{proof}
	Next, we call $f$ to be a generalized eigenfunction of $\mathbb{L}_\psi$ corresponding to the eigenvalue $\theta$ if there exists $s\in\N$ such that $f\in\mathcal{D}(\mathbb{L}^s_\psi)$ and $(\mathbb{L}_\psi-\theta I)^sf=0$.
	\begin{lemma}\label{lem:gen_eigen}
		Let $f$ be a generalized eigenfunction of $\mathbb{L}_\psi$ corresponding to the eigenvalue $-\beta$ satisfying $(\mathbb{L}_\psi+\beta I)^sf=0$. Then $f$ is a polynomial of degree at most $\beta$.
	\end{lemma}
	\begin{proof}
		We will prove the lemma using an argument similar to the proof of \cite[Proposition~3.1]{MetafunePallaraPriola2002}. For $s=1$, the lemma holds true due to Theorem~\ref{thm:spectrum}\eqref{it:spec_3}. Now, we proceed by induction. Suppose that the statement of the lemma holds for some $s\geqslant 1$. Let $f\in\mathcal{D}(\mathbb{L}_\psi^{s+1})\subset L^p(G,\mathsf{p}_\psi)$ be such that $(\mathbb{L}_\psi+\beta I)^{s+1}f=0$. Then, we must have 
		\begin{align*}
			\mathbb{P}^\psi_tf=e^{-\beta t}f+e^{-\beta t}\sum_{j=1}^r\frac{(\mathbb{L}_\psi+\beta I)^jf}{j!}.
		\end{align*}
		Also, by induction hypothesis, for each $1\le j\le r$, $(\mathbb{L}_\psi+\beta I)^jf$ is a polynomial of degree at most $\beta$. Since by Lemma~\ref{lem:regularity} $P_tf\in W^{k,p}(G,\mathsf{p}_\psi)$, we deduce from the above equation that $f\in W^{k,p}(G,\mathsf{p}_\psi)$. Therefore, for any $\alpha\in\{1,\ldots,n\}^k$, $X^r_\alpha\mathbb{P}^\psi_tf= e^{-\beta t}X^r_\alpha f$ for all $k\geqslant \beta+1$. Using Lemma~\ref{lem:sobolev_bound} we obtain that for all $k\ge\beta+1$ and $\alpha\in\{1,\ldots, n\}^k$,
		\begin{align*}
			e^{-\beta t}\|X^r_\alpha f\|_{L^p(G,\mathsf{p}_\psi)}\le e^{-kt}\|X^r_\alpha f\|_{L^p(G,\mathsf{p}_\psi)},
		\end{align*}
		which forces that $X^r_\alpha f=0$ whenever $\alpha\in\{1,\ldots, n\}^k$ and $k\geqslant \beta+1$. Therefore, $f$ is a polynomial of degree at most $\beta$.
	\end{proof}
	
	\begin{proof}[Proof of Theorem~\ref{prop:multiplicity}] We note that Lemma~\ref{lem:gen_eigen} implies that the algebraic multiplicity of the eigenvalue $\beta$ of $\mathbb{L}_\psi$ is finite. Moreover, Theorem~\ref{thm:OU_intertwining} also implies that for all $\beta\in\N_0$ and $s\geqslant 1$,
		\begin{align*}
			(\mathbb{L}+\beta I)^s\Gamma_\psi=\Gamma_\psi (\mathbb{L}_\psi+\beta I)^s \text{ on } \mathcal P.
		\end{align*}
		Since $\Gamma_\psi$ is invertible on $\mathcal P$, we conclude that $\dim(\ker(\mathbb{L}+\beta I)^s)=\dim(\ker(\mathbb{L}_\psi+\beta I)^s)$ for all $s\geqslant 1$, which concludes the proof.
	\end{proof}

	\subsection{Co-eigenfunctions of perturbed OU semigroup} Note that we only used the intertwining with $\mathbb{P}$ to determine the spectrum and the eigenspaces of $\mathbb{P}^\psi$. On the other hand, the intertwining with $P^\psi$ provides some information about the co-eigenfunctions of $\mathbb{P}^\psi$ for even eigenvalues. For any $\psi\in\mathcal{N}_{\log}$ (see \eqref{eq:N_log} for definition) and $\beta=(\beta_1,\ldots, \beta_m)\in\N^m_0$, let us define 
	\begin{align}\label{eq:G-psi}
		J^\psi_\beta(h,v)=(-1)^{|\beta|}\frac{\partial_v^{\beta}\mathsf{p}_\psi(h,v)}{\mathsf{p}_\psi(h,v)},
	\end{align}
	where $\partial^{\beta}_v=\partial^{\beta_1}_{v_1}\cdots\partial^{\beta_m}_{v_m}$. For any $1<p<\infty$, let us denote the adjoint of $\mathbb{P}^\psi_t: L^p(G,\mathsf{p}_\psi)\longrightarrow L^p(G,\mathsf{p}_\psi)$ by $\mathbb{P}^{\psi*}_t: L^{q}(G,\mathsf{p}_\psi)\longrightarrow L^{q}(G,\mathsf{p}_\psi)$, where $q=\frac{p}{p-1}$. Then we have the following result.
	\begin{theorem}\label{thm:coeignefunction}
		Let $\psi\in\mathcal{N}_{\log}$. Then, for all $\beta\in\N^m_0$ and $1<p<\infty$, $J^\psi_\beta\in L^{q}(G,\mathsf{p}_\psi)$. Moreover, for all $t\geqslant 0$, 
		\begin{align}\label{eq:coeigenfunction}
			\mathbb{P}^{\psi*}_t J^\psi_\beta= e^{-2|\beta|t}J^\psi_\beta \text{ on }  L^{q}(G,\mathsf{p}_\psi).
		\end{align}
	\end{theorem}
	\begin{corollary}
		For any $\beta\in\N^m_0$, let us define 
		\begin{align*}
			J(h,v)=(-1)^{|\beta|}\frac{\partial^\beta_v \mathsf{p}(h,v)}{\mathsf{p}(h,v)}.
		\end{align*}
		Then, for all $t\geqslant 0$ and $1<p<\infty$, $\mathbb{P}^{\ast}_t J_\beta= e^{-2|\beta| t} J_\beta$ on $ L^{q}(G, \mathsf{p})$.
	\end{corollary}
	We need some auxiliary results to prove the above theorem. To this end, for any nonnegative definite matrix $\sigma$ of order $m\times m$, let $P^\sigma$ denote the OU semigroup on $\R^m$ generated by 
	\begin{align*}
		L_\sigma f(v)=\trace(\sigma\nabla^2 f(v))-2\langle v, \nabla f(v)\rangle.
	\end{align*}
	\begin{proposition}\label{prop:eigen_adjoint}
		If $\psi=(\sigma,b,\kappa)\in\mathcal{N}_{\log}$ with $\sigma$ non-singular, for all $p\geqslant 1$ there exists a Markov operator $\Lambda_\psi: L^p(G,\mathsf{p}_\psi)\longrightarrow L^p(\R^m,\mu_\varrho)$ with dense range satisfying
		\begin{align}\label{eq:intertwining_0}
			P^{(\varrho)}_t\Lambda_\psi=\Lambda_\psi\mathbb{P}^\psi_t \text{ on  }  L^p(G,\mathsf{p}_\psi),
		\end{align}\label{eq:coeigen}
		where $P^{(\varrho)}=P^{\varrho I_{m\times m}}$, $\varrho$ being the smallest eigenvalue of $\frac{\sigma}{2}$, and $\mu_\varrho$ is the invariant distribution of $P^{(\varrho)}$ that is given by
		\begin{align*}
			\mu_\varrho(dv)=(\varrho\pi)^{-\frac d2} e^{-\frac{|v|^2}{\varrho}}dv.
		\end{align*}
	\end{proposition}
	We need the following lemmas to prove the above proposition.
	\begin{lemma}\label{lem:injectivity} Let $\Lambda$ be defined as in \eqref{eq:Lambda}. Then for any $p\geqslant 1$, $\Lambda: L^p(G, \mathsf{p}_\psi)\longrightarrow L^p(\R^m,\mu_\psi)$ has dense range.
	\end{lemma}
	\begin{remark}
		We note that $\Lambda$ is not injective. Choosing $f(x,y,r,v)=x$, we have $f\in L^2(G,\mathsf{p}_\psi)$. Also, due the symmetry of $\mathsf{p}$ with respect to $x$, we must have 
		\begin{align*}
			\int_{\R^m}\int_{\R^k} \int_{\C^d}x\mathsf{p}(x,y,r,v+v')dx dy drdv'=0.
		\end{align*}
	\end{remark}
	\begin{proof}
		Let us write $e_{\lambda}(v)=e^{\i\langle \lambda, v\rangle}$. Clearly, $e_\lambda\in L^2(G, \mathsf{p}_\psi)$ for any $\lambda\in\R^m$. Also, 
		\begin{align*}
			\Lambda e_{\lambda}(v)=e_{\lambda}(v)\int_{\C^d}\int_{\R^k}\int_{\R^m} \mathsf{p}(z,r,v')e^{-\i\langle\lambda,v'\rangle} dv' dr dz,
		\end{align*}
		which shows that $e_{\lambda}\in\Range(\Lambda)$ for all $\lambda\in\R^m$. Therefore, the proof of the lemma follows after noting that $\Span\{e_{\lambda}: \lambda\in\R^m\}$ is a dense subspace in $ L^p(\R^m,\mu_{\psi})$ for all $p\geqslant 1$.
	\end{proof}
	The next lemma is a version of \cite[Theorem~5.3]{Sarkar2025}.
	\begin{lemma}\label{lem:intertwining_00} 
		Let $\psi=(\sigma,b,\kappa)\in\mathcal{N}_{\log}$. Then, for all $t\geqslant 0$, 
		\begin{align}\label{eq:intertwining_00}
			P^{\sigma}_t T_{b,\kappa}=T_{b,\kappa} P^\psi_t \text{ on } C_b(\R^m),
		\end{align}
		where $T_{b,\kappa} f=f{\ast}h_{b,\kappa}$ and for all $\lambda\in\R^m$,
		\begin{align*}
			\mathcal{F}(h_{b,\kappa})(\lambda)=\exp(-\widetilde{\psi}_{-\infty}(-\lambda))
		\end{align*}
		with $\widetilde{\psi}=(0,b,\kappa)\in\mathcal N$ and $\widetilde{\psi}_{-\infty}=\lim_{t\to -\infty}\widetilde{\psi}_t$, see \eqref{eq:psi_t}.
		Moreover, $T_{b,\kappa}$ is a Markov operator with dense range, and has a unique continuous extension, also denoted by $T_{b,\kappa}: L^p(\R^m,\mu_\psi)\longrightarrow L^p(\R^m,\mu_\varrho)$ for all $p\geqslant 1$. As a result, \eqref{eq:intertwining_00} holds on $ L^p(\R^m,\mu_\psi)$ for all $p\geqslant 1$.
	\end{lemma}
	\begin{proof}
		From \cite[Theorem~17.1]{SatoBook1999} it is known that for any $f\in L^1(\R^m)\cap L^2(\R^m)$ and $t\geqslant 0$, 
		\begin{align*}
			\mathcal{F}(\widehat{P}^\psi_t f)(\lambda)=e^{\psi_t(e^{-2t}\lambda)}\mathcal{F}(f)(e^{-2t}\lambda),
		\end{align*}
		where $\psi_t$ is defined in \eqref{eq:psi_t}. In particular, we also have 
		\begin{align*}
			\mathcal{F}(\widehat{P}^\sigma_t f)(\lambda)=\exp\left(\frac{e^{-4t}-1}{4}\langle\sigma\lambda,\lambda\rangle\right)\mathcal{F}(f)(e^{-2t}\lambda)
		\end{align*}
		for all $\lambda\in\R^m$. Also, $T_{b,\kappa}$ is a Markov operator as $h_{b,\kappa}$ is the probability measure associated with the L\'evy-Khintchine exponent $\lambda\mapsto-\widetilde{\psi}_{-\infty}(-\lambda)$. As a result, $T_{b,\kappa}$ can be extended uniquely as a bounded operator on $ L^p(\R^m)$ for all $p\geqslant 1$. Let $\widehat{T}_{b,\kappa}$ denote the $ L^2(\R^m)$-adjoint of $T_{b,\kappa}$. Then, from the definition of $T_{b,\kappa}$ it follows that for all $f\in C^\infty_c(\R^m)$ and $\lambda\in\R^m$,
		\begin{align*}
			\mathcal{F}(\widehat{T}_{b,\kappa} f)(\lambda)=e^{-\widetilde{\psi}_{-\infty}(\lambda)}\mathcal{F}(f)(\lambda).
		\end{align*}
		Now, a simple computation shows that for all $\lambda\in\R^m$ and $t>0$, 
		\begin{align*}
			\psi_t(e^{-2t}\lambda)-\widetilde{\psi}_{-\infty}(e^{-2t}\lambda)=-\widetilde{\psi}_{-\infty}(\lambda)+\frac{e^{-4t}-1}{4}\langle\sigma\lambda,\lambda\rangle.
		\end{align*}
		As a result, for all $\lambda\in\R^m$ and $f\in L^2(\R^m)\cap L^1(\R^m)$ we have 
		\begin{align*}
			\mathcal{F}(\widehat{P}^\psi_t \widehat{T}_{b,\kappa} f)(\lambda)=\mathcal{F}(\widehat{T}_{b,\kappa}\widehat{P}^\sigma_t f )(\lambda),
		\end{align*}
		which is equivalent to $\widehat{P}^\psi_t\widehat{T}_{b,\kappa} f=\widehat{\Lambda}_{b,\kappa} \widehat{P}^\sigma_t f$ for all $f\in C^\infty_c(\R^m)$. Taking the $L^2$-adjoint of the last identity proves \eqref{eq:intertwining_00} on $C^\infty_c(\R^m)$. Using the density of $C^\infty_c(\R^m)$ in $C_b(\R^m)$ and the contraction property of the semigroups, we conclude that \eqref{eq:intertwining_00} holds. We further get that 
		\begin{align*}
			p_\varrho T_{b,\kappa} f=p_\varrho P^{(\varrho)}_t T_{b,\kappa}=p_{\varrho} T_{b,\kappa} P^\psi_t f
		\end{align*}
		for all $t\geqslant 0$ and $f\in C_b(\R^m)$. Since $T_{b,\kappa}$ is Markov and $\mu_\psi$ is the unique invariant distribution of $P^\psi$, the above equation forces that $\mu_\varrho T_{b,\kappa}=\mu_\psi$. As a result, $T_{b,\kappa}$ extends as a bounded operator from $ L^p(\R^m, \mu_\psi)$ to $ L^p(\R^m, \mu_\varrho)$ for all $p\geqslant 1$.
		Finally, we note that for all $\lambda\in\R^m$, $T_{b,\kappa}e_\lambda=e^{-\psi_{-\infty}(-\lambda)}e_\lambda$ where $e_\lambda(v)=e^{\i\langle\lambda, v\rangle}$. Since $\Span\{e_\lambda:\lambda\in\R^m\}$ is dense in $ L^p(\R^m,\mu_\psi)$ for all $p\in [1,\infty]$, we conclude that $T_{b,\kappa}: L^p(\R^m,\mu_\psi)\longrightarrow  L^p(\R^m,\mu_\varrho)$ has dense range for all $p\in [1,\infty]$. This completes the proof of the lemma.
	\end{proof}
	\begin{proof}[Proof of Proposition~\ref{prop:eigen_adjoint}]Since $\sigma$ is invertible with positive eigenvalues $\varrho_1,\ldots, \varrho_m$, from \cite[Proposition~4.1]{Patie_Vaidyanathan} it is known that for all $t\geqslant 0$, 
		\begin{align}\label{eq:intertwining_2}
			P^{(\varrho)}_tT_\sigma=T_\sigma P^\sigma_t \text{ on }  L^p(\R^m, \mu_\sigma),
		\end{align}
		where $T_\sigma: L^p(\R^m,\mu_\sigma)\longrightarrow L^p(\R^m,\mu_\varrho)$ is a Markov operator with a dense range. Due to the transitivity of the intertwining relationships, combining the identities \eqref{eq:Q_P}, \eqref{eq:intertwining_2} and Lemma~\ref{lem:intertwining_00}, we obtain \eqref{eq:intertwining_0} with
		\begin{align}\label{eq:Lambda_psi}
			\Lambda_\psi=T_\sigma T_{b,\kappa}\Lambda.
		\end{align}
		As $T_\sigma$, $T_{b,\kappa}$ and $\Lambda$ have dense range, we conclude that $\Lambda_\psi$ has dense range as well. This completes the proof of the proposition.
	\end{proof}
	To prove Theorem~\ref{thm:coeignefunction}, we first need to show that $J^\psi_\beta\in L^{q}(G, \mathsf{p}_\psi)$ for all $1<p<\infty$.
	\begin{lemma}\label{lem:coeigen_approx}
		Let $J^\psi_\beta$ be as in \eqref{eq:G-psi} and $\psi\in\mathcal{N}_{\log}$. If $\psi=(\sigma, b,\kappa)$ with $\sigma$ being non-singular, then for all $\beta\in\N^m_0$, $J^\psi_\beta=\Lambda^{\!*}_\psi\Phi^{(\varrho)}_\beta$, where 
		$\Phi^{(\varrho)}_\beta$ denote the generalized Hermite polynomials defined by 
		\begin{align}
			\Phi^{(\varrho)}_\beta(v)=\frac{\partial^{\beta} \mu_\varrho(v)}{\mu_\varrho(v)},
		\end{align}
		and $\Lambda_\psi$ is the same as in Proposition~\ref{prop:eigen_adjoint}.
		Moreover, for any $\psi\in\mathcal{N}_{\log}$, $J^\psi_\beta\in L^{q}(G,\mathsf{p}_\psi)$ for all $1<p<\infty$.
	\end{lemma}
	\begin{proof}
		For $\sigma$ invertible, let $\Lambda_\psi$ be the same as in Proposition~\ref{prop:eigen_adjoint}. Denoting the adjoint of $\Lambda_\psi: L^2(\R^d)\longrightarrow  L^2(G)$ by $\widehat{\Lambda}_\psi$, we note that for all $f\in L^2(\R^m, p_\varrho)$,
		\begin{align*}
			\Lambda^{\!*}_\psi f=\frac{\widehat{\Lambda}_\psi(f\mu_{\varrho})}{\mathsf{p}_\psi}.
		\end{align*}
		Therefore, we have $\Lambda^{\ast}_\psi\Phi^{(\varrho)}_\beta(h,v)=\frac{1}{\mathsf{p}_\psi(h,v)}\widehat{\Lambda}_\psi(\partial^\beta \mu_\varrho)(h,v)$. Since $\widehat{\Lambda}_\psi$ commutes with derivatives along vertical variables, the first statement of the lemma follows. Moreover, when $\sigma$ is invertible, since $\Phi^{(\varrho)}_\beta\in L^q(\R^m,\mu_\varrho)$, and $\Lambda^{\!*}_\psi: L^{q}(\R^m, \mu_\varrho)\longrightarrow L^{q}(G, \mathsf{p}_\psi)$ is bounded, we conclude that $J^\psi_\beta\in L^{q}(G, \mu_\psi)$ for all $1<p<\infty$.
		
		When $\sigma$ is not invertible, we approximate $\psi=(\sigma, b,\kappa)$ by $\psi_\epsilon=(\sigma_\epsilon, b,\kappa)$ such that $\sigma_\epsilon$ is invertible and $\sigma_\epsilon\to \sigma$ as $\epsilon\to 0$. Since by Theorem~\ref{prop:P_gamma}\eqref{it:4}, $\mathsf{p}_{\psi_\epsilon}=\mathsf{q}_{\frac12}{\ast}_v\overline{\mu}_{\psi_\epsilon}$ and $\mathsf{q}_{\frac 12}$ is smooth, for any $\beta\in\N^m_0$ we have 
		\begin{align}
			\partial^\beta_v \mathsf{p}_{\psi_\epsilon}=\partial^\beta_v \mathsf{q}_{\frac12}{\ast}_v \overline{\mu}_{\psi_\epsilon}.
		\end{align}
		As $\lim_{\epsilon\to 0}\psi_\epsilon=\psi$, we have $\lim_{\epsilon\to 0} \mu_{\psi_\epsilon}=\mu_\psi$ weakly. Using the fact that $\partial^\beta_v \mathsf{q}_{\frac12}\in C_0(G)$, we conclude that $\lim_{\epsilon\to 0} \partial^\beta_v \mathsf{p}_{\psi_\epsilon}=\partial^\beta_v \mathsf{p}_\psi$. As a result, $\lim_{\epsilon\to 0} J^{\psi_\epsilon}_\beta=J^\psi_\beta$ for all $\beta\in\N^m_0$. Let $\varrho_\epsilon$ denote the smallest eigenvalue of $\frac{\sigma_{\epsilon}}{2}$. Using Fatou's lemma we infer that 
		\begin{align*}
			\int_G |J^\psi_\beta(h,v)|^{q} \mathsf{p}_\psi(h,v) dh dv&\le \liminf_{\epsilon\to 0} \int_G |J^{\psi_\epsilon}_\beta(h,v)|^{q} \mathsf{p}_{\psi_\epsilon}(h,v)dh dv \\
			&\le \sup_{\epsilon>0} \|J^{\psi_\epsilon}_\beta\|^{q}_{ L^{q}(G, \mathsf{p}_{\psi_\epsilon})}=\sup_{\epsilon>0}\|\Lambda^{\!*}_{\psi_\epsilon} \Phi^{(\varrho)}_\beta\|_{ L^{q}(G,\mathsf{p}_{\psi_\epsilon})} \\
			&\le \sup_{\epsilon>0} \|\Phi^{(\varrho_\epsilon)}_\beta\|^{q}_{ L^{q}(\R^m, \mu_{\varrho_\epsilon})}=\|\Phi^{(1)}_\beta\|^{q}_{ L^{q}(\R^m, \mu_{1})}<\infty,
		\end{align*}
		where the last equality follows by a simple change of variable.
		This completes the proof of the lemma.
	\end{proof}
	\begin{proof}[Proof of Theorem~\ref{thm:coeignefunction}]
		First, we prove the theorem when $\sigma$ is non-singular. Taking the adjoint of \eqref{eq:intertwining_0} we get
		\begin{align}\label{eq:intertwining_adjoint}
			\mathbb{P}^{\psi*}_t\Lambda^{\!\ast}_\psi=\Lambda^{\!\ast}_\psi P^{(\varrho)}_t \text{ on }  L^{q}(\R^m,\mu),
		\end{align}
		where $q=\frac{p}{p-1}$.
		Since each of the operators $T_\sigma, T_{b,\kappa}$, and $\Lambda$ has dense range, thanks to Lemma~\ref{lem:injectivity}, $\Lambda^{\!*}_\psi$ must be injective. Also, for any $\beta\in\N^m_0$, $P^{(\varrho)}_t \Phi^{(\varrho)}_\beta= e^{-2|\beta| t}\Phi^{(\varrho)}_\beta$. As a result, by Lemma~\ref{lem:coeigen_approx}, \eqref{eq:coeigenfunction} follows when $\sigma$ is invertible. When $\sigma$ is not invertible, we use the following alternative argument. Since for any $f\in L^{q}(G, \mathsf{p}_\psi)$, 
		\begin{align*}
			\mathbb{P}^{\psi*}_t f=\frac{\widehat{\mathbb{P}}^\psi_t(f\mathsf{p}_\psi)}{\mathsf{p}_\psi},
		\end{align*}
		it suffices to prove that $\widehat{\mathbb{P}}^{\psi}_t(J^\psi_\beta \mathsf{p}_\psi)= e^{-2|\beta|t} J^\psi_\beta \mathsf{p}_\psi$ for all $\beta\in\N^m_0$. 
		We note that for all $(\lambda,\nu)\in\mathcal{Z}\times\R^m$, 
		\begin{align*}
			\mathcal{F}_G(\partial^\beta_v \mathsf{p}_\psi)(\lambda,\nu)=\i^{|\beta|}\lambda^\beta\exp(-\psi_{-\infty}(-\lambda))H^{\lambda,\nu}_{\frac 12},
		\end{align*}
		where $\lambda^\beta=\lambda^{\beta_1}_1\cdots \lambda^{\beta_m}_m$.
		Therefore, using \eqref{eq:P_hat} in Proposition~\ref{prop:fourier_transform} we obtain that for all $(\lambda,\nu)\in\mathcal{Z}\times\R^m$, 
		\begin{align*}
			&\mathcal{F}_G(\widehat{\mathbb{P}}^{\psi}_t(J^\psi_\beta \mathsf{p}_\psi))(\lambda,\nu)\\
			&=e^{\psi_{t}(-e^{-2t}\lambda)}d_{e^t} \mathcal{F}_G(\partial^\beta_v \mathsf{p}_\psi)(e^{-2t}\lambda , e^{-t}\nu ) d_{e^{-t}} H^{\lambda,\nu}_{\frac{1-e^{-2t}}{2}}\\
			&=e^{-2t|\beta|}e^{\psi_{t}(-e^{-2t}\lambda)} \i^{|\beta|} \lambda^\beta\exp(-\psi_{-\infty}(-e^{-2t}\lambda))d_{e^t} H^{\lambda e^{-2t},\nu e^{-t}}_{\frac12}d_{e^{-t}} H^{\lambda,\nu}_{\frac{1-e^{-2t}}{2}}  \\
			&= e^{-2t|\beta|}\i^{|\beta|} \lambda^\beta\exp(-\psi_{-\infty}(-\lambda))H^{\lambda,\nu}_{\frac 12} \\
			&= e^{-2t|\beta|}\mathcal{F}_G(J^\psi_\beta \mathsf{p}_\psi)(\lambda,\nu),
		\end{align*}
		where we used Lemma~\ref{lem:psi_t}.
		Therefore, $\widehat{\mathbb{P}}^\psi_t (J^\psi_\beta \mathsf{p}_\psi) =e^{-2|\beta| t}J^\psi_\beta \mathsf{p}_\psi$ for all $t\geqslant 0$ and $\beta\in\N^m_0$. Since $J^\psi_\beta\in L^{q}(G,\mathsf{p}_\psi)$ for all $\psi\in\mathcal{N}_{\log}$, the proof of the theorem is concluded.
	\end{proof}
	
	\subsection{\texorpdfstring{$L^1$}{L1}-spectrum} As noted above, the $L^p(G,\mathsf{p}_\psi)$-spectrum of $\mathbb{P}^\psi$ is discrete for all $1<p<\infty$, where $\mathsf{p}_\psi$ is the unique invariant density of $\mathbb{P}^\psi$ defined by \eqref{eq:p_gamma}. The spectrum of $\mathbb{P}^\psi$ is very different in $L^1(G,\mathsf{p}_\psi)$. The following result is analogous to \cite[Theorem~5.1]{MetafunePallaraPriola2002}.
	\begin{theorem}\label{thm:L_1_spec}
		For any $\psi\in\mathcal N_{\log}$ and $t\geqslant 0$, $\sigma(\mathbb{P}^\psi_t; L^1(G, \mathsf{p}_\psi))=\{z\in\C: |z|\le 1\}$. In fact, $\{z\in\C: |z|<1\}\subseteq\sigma_{\rm p}(\mathbb{P}^\psi_t; L^1(G, \mathsf{p}_\psi))$.
		Therefore, $\mathbb{P}^\psi_t$ is non-compact on $ L^1(G,\mathsf{p}_\psi)$ for every $t>0$.
	\end{theorem}
	\begin{proof}
		From Proposition~\ref{prop:horizontal_intertwining}, we have for all $t\geqslant 0$,
		\begin{align*}
			\mathbb{P}^\psi_t\Pi=\Pi\widetilde{P}_t \text{ on }  L^1(\R^m, \widetilde{\mu}).
		\end{align*}
		Since $\Pi: L^1(\R^m,\widetilde{\mu})\longrightarrow L^1(G,\mathsf{p}_\psi)$ is injective and $e^{-t\C_{>0}}\subseteq\sigma_{\rm p}(\widetilde{P}_t; L^1(\R^n,\widetilde{\mu}))$, see \cite[Theorem~5.1]{MetafunePallaraPriola2002}, we conclude that $e^{-t\C_{>0}}\subseteq\sigma_{\rm p}(\mathbb{P}^\psi_t; L^1(G, \mathsf{p}_\psi))$ for all $t>0$, where $\C_{\ge 0}$ (resp. $\C_{>0}$) denotes the the set of all complex numbers with non-positive (resp. strictly negative) real part. As $\sigma(\mathbb{P}^\psi_t; L^1(G,\mathsf{p}_\psi))$ is a closed subset of $\C$ and also contained in $e^{-t\C_{\geqslant 0}}\cup\{0\}$, we have
		$\sigma(\mathbb{P}^\psi_t; L^1(G, \mathsf{p}_\psi))= e^{-t\C_{\geqslant 0}}\cup\{0\}=\{z\in \C: |z|\le 1\}$. Also, $\mathbb{P}^\psi_t: L^1(G,\mathsf{p}_\psi)\longrightarrow L^1(G,\mathsf{p}_\psi)$ is not compact as the its spectrum is not a discrete subset of $\C$.
	\end{proof}
	\appendix
	\section{Generalized Fourier transform on nilpotent Lie groups}\label{sec:3}
	In the setting of locally compact groups the generalized Fourier transform (GFT) relies on the unitary dual $\widehat{G}$ of the group $G$, that is, the space of irreducible unitary representations of $G$. Recall that the GFT $\mathcal{F}_{G}(f)$ of a function $f\in L^1(G)$ is defined as
	
	\begin{align}\label{eq:fourier_transform_general}
		\mathcal{F}_G(f)(\pi):=\int_{G} f(g)\pi(g)dg, \quad \pi=\left(\pi, H_{\pi}\right) \in \widehat{G}.
	\end{align}
	Note that $\mathcal{F}_G(f)(\pi)$ is a linear operator on the representation space $H_{\pi}$. For nilpotent groups the unitary dual $\widehat{G}$ can be described using  \emph{Kirillov's orbit method}, see \cite{CorwinGreenleafBook, KirillovOrbitMethodAMSBook}.  
	For the sake of completeness, we briefly describe it here. First, we recall the following theorem due to Kirillov, we refer to \cite[Section~2.2]{CorwinGreenleafBook} for details. One of the simplifications in the case of nilpotent groups is that one can realize these representations on the same Hilbert space which is described explicitly as an $ L^{2}$-space on a Euclidean space. 
	\begin{theorem}[Kirillov]\label{t.Kirillov}
		Let $G$ be any connected nilpotent Lie group with Lie algebra $\mathfrak g$. Then, the following holds.
		\begin{enumerate}[leftmargin=*]
			\item \label{item:h_l} For any $l\in\mathfrak{g}^{\ast}$ there exists a sub-algebra $\mathfrak h_l$ of maximal dimension such that $l([\mathfrak h_l, \mathfrak h_l])=0$. $\mathfrak{h}_l$ is called polarizing for $l$.
			\item $H_l=\exp(\mathfrak h_l)$ is a closed subgroup of $G$ and $\rho_l(\exp(X))=\exp(\i l(X))$ are one-dimensional representations of $H_l$.
			\item The induced representation $\pi_l=\Ind^G_{H_l, \rho_l}$ is an irreducible unitary representation of $G$.
			\item If $\pi$ is an irreducible unitary representation of $G$ then $\pi$ is unitarily equivalent to $\pi_l$ for some $l\in\mathfrak{g}^{\ast}$.
			\item \label{item5} $\pi_{l_1}$ and $\pi_{l_2}$ are unitarily equivalent if and only if $l_1$ and $l_2$ belong to the same co-adjoint orbit.
		\end{enumerate}
	\end{theorem}
	In this case, the Hilbert space $H_{\pi_l}$ for the representation $\pi_l$ can be realized as $ L^2(\R^d)$, where $d=n-\dim(\mathfrak{h}_l)$. For any $l\in\mathfrak{g}^{\ast}$, its \emph{radical} is defined as 
	\begin{align*}
		\operatorname{rad}_l=\{X\in\mathfrak{g}: l([X,X'])=0 \text{ for all }  X'\in\mathfrak g\}.
	\end{align*}
	By \cite[Theorem~1.3.3]{CorwinGreenleafBook}, $\mathfrak{h}_l$ in Theorem~\ref{t.Kirillov}\eqref{item:h_l} is a polarizing sub-algebra if and only if 
	\begin{enumerate}
		\item \label{pol1} $\operatorname{rad}_l\subset \mathfrak{h}_l$, 
		\item \label{pol2} $\dim(\mathfrak{h}_l)=\frac{\dim (\operatorname{rad}_l)+\dim(\mathfrak g)}{2}$,
		\item \label{pol3} $l([\mathfrak{h}_l,\mathfrak{h}_l])=0$.
	\end{enumerate} The generalized Fourier transform on $G$ is formally defined as follows: for any $l\in\mathfrak{g}^{\ast}$ and $f\in L^1(G)$,
	\begin{align}
		\mathcal{F}_G(f)(\pi_l)=\int_G f(g)\pi_l(g) dg.
	\end{align}
	Then, by \cite[Theorem~7.44]{FollandHABook}, $\mathcal{F}_G(f)(\pi_l): L^2(\R^d)\longrightarrow L^2(\R^d)$ is a Hilbert-Schmidt operator. We now briefly describe a few preliminaries to state the Plancherel theorem for $\mathcal{F}_G$. The major technicalities are required to define \emph{Plancherel measure} for GFT. Let $\{X_1,\ldots, X_n\}$ be a basis of $\mathfrak g$, and define $K_l\in\mathcal{M}_n(\R)$ such that $K_l(i,j)=l([X_i, X_j])$ for all $1\leqslant i,j\leqslant n$. 
	\begin{definition}\label{def:jump_index}
		Let $l\in\mathfrak{g}^{\ast}$. Then, $i\in\{1,\ldots, n\}$ is called a \emph{jump index} for $l$ if $\rank(K^{i}_l)$ is strictly larger than $\rank(K^{i-1}_l)$, where $K^i_l$ is the $i\times n$ sub-matrix of $K_l$ formed by the first $i$ rows of $K_l$. 
	\end{definition}
	From the above definition, it is clear that for each jump index, the rank increases exactly by $1$. For each $i\in\{1,\ldots, n\}$, let $R_i=\max\{\rank(K^{i}_l): l\in\mathfrak{g}^{\ast}\}$. Let us define 
	\begin{align}\label{eq:U}
		\mathcal{U}^{\ast}=\{l\in\mathfrak{g}^{\ast}: \rank(K^i_l)=R_i \text{ for any } 1\leqslant i\leqslant n\}.
	\end{align}
	Clearly, $\mathcal U^{\ast}$ is a Zariski open subset of $\mathfrak{g}^{\ast}$, and therefore $\mathcal U^{\ast}$ has full Lebesgue measure in $\mathfrak{g}^{\ast}$. Hence, it suffices to consider irreducible unitary representation $\pi_l$ for $l\in \mathcal U^{\ast}$.
	According to \cite[Theorem~3.1.6]{CorwinGreenleafBook}, $\mathcal{U}^{\ast}$ is $G$-invariant and therefore it is an union of co-adjoint orbits in $\mathfrak{g}^{\ast}$, and each $l\in\mathcal{U}^{\ast}$ has the same set of jump indices. Denoting the common set of jump indices by $S$, we have a partition of $\{1,\ldots, n\}$ into two disjoint sets, $S$ and $T$ respectively. Note that $S$ has even number of elements as $K_l$ is a skew-symmetric matrix. Let $\{X^{\ast}_1, \ldots, X^{\ast}_n\}$ denote the dual basis of $\mathfrak{g}^{\ast}$ and let us write
	\begin{align*}
		\mathfrak{g}^{\ast}_T=\Span\{X^{\ast}_i: i\in T\}.
	\end{align*}
	Then by \cite[Theorem~3.1.9]{CorwinGreenleafBook}, every co-adjoint orbit in $\mathcal{U}^{\ast}$ intersects $\mathfrak{g}^{\ast}_T$ at a unique point. As a result, by Theorem~\ref{t.Kirillov}\eqref{item5}, all irreducible representations on $G$ are unitarily equivalent to $\pi_l$ for some $l\in\mathcal{U}^{\ast}\cap \mathfrak{g}^{\ast}_T$. Let $B_l$ denote the $|S|\times |S|$ matrix such that $\widetilde{K}_l(i,j)=l([X_i, X_j])$ for all $i,j\in S$. Then, the following Plancherel theorem holds, see \cite[Theorem~4.3.10]{CorwinGreenleafBook}.
	\begin{theorem}[Plancherel's theorem] \label{t.Plancherel}
		For any $f\in L^2(G)\cap  L^1(G)$
		\begin{align*}
			\int_G |f(g)|^2 dg=\int_{\mathcal{U}^{\ast}\cap \mathfrak{g}^{\ast}_T} \|\mathcal{F}_G(f)(\pi_l)\|^2_{\operatorname{HS}}|\Pf(l)| dl,
		\end{align*}
		where $\Pf(l)^2=\det(\widetilde{K}_l)$ is the Pfaffian, and the above integral is understood by identifying $\mathcal{U}^{\ast}\cap \mathfrak{g}^{\ast}_T$ as a subset of the Euclidean space, and $dl$ is the Lebesgue measure.
	\end{theorem}
	
	\section*{Acknowledgment} RS would like to thank Fabrice Baudoin for many insightful discussions during the preparation of this work.

	\providecommand{\bysame}{\leavevmode\hbox to3em{\hrulefill}\thinspace}
	\providecommand{\MR}{\relax\ifhmode\unskip\space\fi MR }
	\providecommand{\MRhref}[2]{%
		\href{http://www.ams.org/mathscinet-getitem?mr=#1}{#2}
	}
	\providecommand{\href}[2]{#2}


\begin{thebibliography}{10}
		
		\bibitem{AkhiezerBook2021}
		N.~I. Akhiezer, \emph{The classical moment problem and some related questions
			in analysis}, Classics in Applied Mathematics, vol.~82, Society for
		Industrial and Applied Mathematics (SIAM), Philadelphia, PA, [2021]
		\copyright 2021, Reprint of the 1965 edition [ 0184042], Translated by N.
		Kemmer, With a foreword by H. J. Landau. \MR{4191205}
		
		\bibitem{Applebaum2019}
		David Applebaum, \emph{On the spectrum of self-adjoint {L}\'evy generators},
		Commun. Stoch. Anal. \textbf{13} (2019), no.~1, Art. 4, 6. \MR{3983830}
		
		\bibitem{AsaadGordina2016}
		Malva Asaad and Maria Gordina, \emph{Hypoelliptic heat kernels on nilpotent
			{L}ie groups}, Potential Anal. \textbf{45} (2016), no.~2, 355--386.
		\MR{3518678}
		
		\bibitem{BakryBaudoinBonnefontChafai2008}
		Dominique Bakry, Fabrice Baudoin, Michel Bonnefont, and Djalil Chafa{\"{\i}},
		\emph{On gradient bounds for the heat kernel on the {H}eisenberg group}, J.
		Funct. Anal. \textbf{255} (2008), no.~8, 1905--1938. \MR{2462581
			(2010m:35534)}
		
		\bibitem{BaudoinHairerTeichmann2008}
		Fabrice Baudoin, Martin Hairer, and Josef Teichmann, \emph{Ornstein-{U}hlenbeck
			processes on {L}ie groups}, J. Funct. Anal. \textbf{255} (2008), no.~4,
		877--890. \MR{2433956 (2009d:58060)}
		
		\bibitem{BergForstBook1975}
		Christian Berg and Gunnar Forst, \emph{Potential theory on locally compact
			abelian groups}, Springer-Verlag, New York-Heidelberg, 1975, Ergebnisse der
		Mathematik und ihrer Grenzgebiete, Band 87. \MR{0481057}
		
		\bibitem{BonfiglioliLanconelliUguzzoniBook}
		A.~Bonfiglioli, E.~Lanconelli, and F.~Uguzzoni, \emph{Stratified {L}ie groups
			and potential theory for their sub-{L}aplacians}, Springer Monographs in
		Mathematics, Springer, Berlin, 2007. \MR{2363343}
		
		\bibitem{BrunoPelosoTabaccoVallarino2019}
		Tommaso Bruno, Marco~M. Peloso, Anita Tabacco, and Maria Vallarino,
		\emph{Sobolev spaces on {L}ie groups: embedding theorems and algebra
			properties}, J. Funct. Anal. \textbf{276} (2019), no.~10, 3014--3050.
		\MR{3944287}
		
		\bibitem{CorwinGreenleafBook}
		Lawrence~J. Corwin and Frederick~P. Greenleaf, \emph{Representations of
			nilpotent {L}ie groups and their applications. {P}art {I}}, Cambridge Studies
		in Advanced Mathematics, vol.~18, Cambridge University Press, Cambridge,
		1990, Basic theory and examples. \MR{1070979 (92b:22007)}
		
		\bibitem{CoulhonRussTardivel-Nachef2001}
		Thierry Coulhon, Emmanuel Russ, and Val\'{e}rie Tardivel-Nachef, \emph{Sobolev
			algebras on {L}ie groups and {R}iemannian manifolds}, Amer. J. Math.
		\textbf{123} (2001), no.~2, 283--342. \MR{1828225}
		
		\bibitem{DasguptaMolahajlooWong2011}
		Aparajita Dasgupta, Shahla Molahajloo, and Man-Wah Wong, \emph{The spectrum of
			the sub-{L}aplacian on the {H}eisenberg group}, Tohoku Math. J. (2)
		\textbf{63} (2011), no.~2, 269--276. \MR{2812454}
		
		\bibitem{DriverGrossSaloff-Coste2010}
		Bruce~K. Driver, Leonard Gross, and Laurent Saloff-Coste, \emph{Growth of
			{T}aylor coefficients over complex homogeneous spaces}, Tohoku Math. J. (2)
		\textbf{62} (2010), no.~3, 427--474. \MR{2742018}
		
		\bibitem{EngelNagelBook2006}
		Klaus-Jochen Engel and Rainer Nagel, \emph{A short course on operator
			semigroups}, Universitext, Springer, New York, 2006. \MR{2229872}
		
		\bibitem{FollandHABook}
		Gerald~B. Folland, \emph{A course in abstract harmonic analysis}, Studies in
		Advanced Mathematics, CRC Press, Boca Raton, FL, 1995. \MR{MR1397028
			(98c:43001)}
		
		\bibitem{FurutaniSagamiOtsuki1993}
		Kenro Furutani, Kei Sagami, and Nobukazu {\^{O}}tsuki, \emph{The spectrum of
			the {L}aplacian on a certain nilpotent {L}ie group}, Comm. Partial
		Differential Equations \textbf{18} (1993), no.~3-4, 533--555. \MR{1214871}
		
		\bibitem{GordinaHaga2014}
		Maria Gordina and John Haga, \emph{{L}\'evy processes in a step 3 nilpotent
			{L}ie group}, Potential Anal. \textbf{41} (2014), no.~2, 367--382.
		
		\bibitem{GordinaLaetsch2016}
		Maria Gordina and Thomas Laetsch, \emph{Sub-{L}aplacians on {S}ub-{R}iemannian
			{M}anifolds}, Potential Anal. \textbf{44} (2016), no.~4, 811--837.
		\MR{3490551}
		
		\bibitem{GordinaLuo2022}
		Maria Gordina and Liangbing Luo, \emph{Logarithmic {S}obolev inequalities on
			non-isotropic {H}eisenberg groups}, J. Funct. Anal. \textbf{283} (2022),
		no.~2, Paper No. 109500. \MR{4410358}
		
		\bibitem{Hunt1956a}
		G.~A. Hunt, \emph{Semi-groups of measures on {L}ie groups}, Trans. Amer. Math.
		Soc. \textbf{81} (1956), 264--293. \MR{MR0079232 (18,54a)}
		
		\bibitem{KirillovOrbitMethodAMSBook}
		A.~A. Kirillov, \emph{Lectures on the orbit method}, Graduate Studies in
		Mathematics, vol.~64, American Mathematical Society, Providence, RI, 2004.
		\MR{2069175 (2005c:22001)}
		
		\bibitem{LiaoMingBookLevyProcessesinLieGroups}
		Ming Liao, \emph{L\'evy processes in {L}ie groups}, Cambridge Tracts in
		Mathematics, vol. 162, Cambridge University Press, Cambridge, 2004.
		\MR{MR2060091 (2005e:60004)}
		
		\bibitem{Lust-Piquard2010}
		Fran\c{c}oise Lust-Piquard, \emph{Ornstein-{U}hlenbeck semi-groups on
			stratified groups}, J. Funct. Anal. \textbf{258} (2010), no.~6, 1883--1908.
		\MR{2578458}
		
		\bibitem{Melcher2008}
		Tai Melcher, \emph{Hypoelliptic heat kernel inequalities on {L}ie groups},
		Stochastic Process. Appl. \textbf{118} (2008), no.~3, 368--388. \MR{2389050
			(2009a:58053)}
		
		\bibitem{MetafunePallaraPriola2002}
		G.~Metafune, D.~Pallara, and E.~Priola, \emph{Spectrum of
			{O}rnstein-{U}hlenbeck operators in {$L^p$} spaces with respect to invariant
			measures}, J. Funct. Anal. \textbf{196} (2002), no.~1, 40--60. \MR{1941990}
		
		\bibitem{MicloPatieSarkar2022}
		Laurent Miclo, Pierre Patie, and Rohan Sarkar, \emph{Discrete self-similar and
			ergodic {M}arkov chains}, Ann. Probab. \textbf{50} (2022), no.~6, 2085--2132.
		\MR{4499275}
		
		\bibitem{NiuLuo2000}
		Pengcheng Niu and Xuebo Luo, \emph{Spectral properties of second order
			differential operators on two-step nilpotent {L}ie groups}, J. Partial
		Differential Equations \textbf{13} (2000), no.~1, 1--10. \MR{1745984}
		
		\bibitem{ParuiPhDThesis2005}
		Sanjay Parui, \emph{Uncertainty {P}rinciples on {N}ilpotent {L}ie {G}roups},
		Ph.D. thesis, Indian Statistical Institute, 2005, Thesis (Ph.D.)--Indian
		Statistical Institute - Kolkata, p.~90. \MR{4361275}
		
		\bibitem{PatieSarkar2023}
		Pierre Patie and Rohan Sarkar, \emph{Weak similarity orbit of
			(log)-self-similar {M}arkov semigroups on the {E}uclidean space}, Proc. Lond.
		Math. Soc. (3) \textbf{126} (2023), no.~5, 1522--1584. \MR{4588719}
		
		\bibitem{Patie_Savov}
		Pierre Patie and Mladen Savov, \emph{Spectral expansions of non-self-adjoint
			generalized {L}aguerre semigroups}, Mem. Amer. Math. Soc. \textbf{272}
		(2021), no.~1336, vii+182. \MR{4320772}
		
		\bibitem{Patie_Vaidyanathan}
		Pierre Patie and Aditya Vaidyanathan, \emph{A spectral theoretical approach for
			hypocoercivity applied to some degenerate hypoelliptic, and non-local
			operators}, Kinet. Relat. Models \textbf{13} (2020), no.~3, 479--506.
		\MR{4097722}
		
		\bibitem{Ray2001}
		S.~K. Ray, \emph{Uncertainty principles on two step nilpotent {L}ie groups},
		Proc. Indian Acad. Sci. Math. Sci. \textbf{111} (2001), no.~3, 293--318.
		\MR{1851093}
		
		\bibitem{GrandpaRudinBook}
		Walter Rudin, \emph{Functional analysis}, second ed., International Series in
		Pure and Applied Mathematics, McGraw-Hill, Inc., New York, 1991. \MR{1157815}
		
		\bibitem{Sarkar2025}
		Rohan Sarkar, \emph{{S}pectrum of {L}\'evy-{O}rnstein-{U}hlenbeck semigroups on
			$\mathbb{R}^d$}, 2025.
		
		\bibitem{SatoBook1999}
		Ken-iti Sato, \emph{L\'evy processes and infinitely divisible distributions},
		Cambridge Studies in Advanced Mathematics, vol.~68, Cambridge University
		Press, Cambridge, 1999, Translated from the 1990 Japanese original, Revised
		by the author. \MR{1739520 (2003b:60064)}
		
		\bibitem{ThangaveluBook1998}
		Sundaram Thangavelu, \emph{Harmonic analysis on the {H}eisenberg group},
		Progress in Mathematics, vol. 159, Birkh\"auser Boston, Inc., Boston, MA,
		1998. \MR{1633042}
		
		\bibitem{YangZPhDThesis2022}
		Zhipeng Yang, \emph{{H}armonic analysis on 2-step stratified {L}ie groups
			without the {M}oore-{W}olf condition}, Ph.D. thesis, Goettingen University,
		Germany, 2022.
		
	\end{thebibliography}
\end{document}